\documentclass[10pt]{amsart}
\calclayout 
\usepackage[utf8]{inputenc} 
\usepackage[T1]{fontenc} 
\usepackage{amsmath}
\allowdisplaybreaks
\usepackage{amsxtra}
\usepackage{amscd}
\usepackage{amsthm}
\usepackage{amsfonts}
\usepackage{amssymb}
\usepackage{libertine}
\usepackage{tikz-cd}
\usepackage{eucal}
\usepackage[all]{xy}
\usepackage{graphicx}
\usepackage{relsize}
\usepackage{leftidx}
\usepackage[backend=biber, style=alphabetic, maxalphanames=99, natbib=true, maxnames=99]{biblatex}

\usepackage[autostyle=true]{csquotes} 
\usepackage{mathtools}
\usepackage{mathrsfs}
\usepackage[foot]{amsaddr}

\usepackage[hidelinks]{hyperref}
\newtheoremstyle{newplain}
  {8pt}
  {8pt}
  {\itshape}  
  {}       
  {\rmfamily\scshape} 
  {.}         
  {5pt plus 1pt minus 1pt} 
  {}          

\theoremstyle{newplain}
\newtheorem{theorem}{Theorem}[section]
\newtheorem{definition}[theorem]{Definition}
\newtheorem{proposition}[theorem]{Proposition}
\newtheorem{corollary}[theorem]{Corollary}
\newtheorem{lemma}[theorem]{Lemma}
\theoremstyle{remark}
\newtheorem{remark}[theorem]{Remark}
\newcommand{\bb}[1]{\mathbb{#1}}
\newcommand{\mc}[1]{\mathcal{#1}}
\newcommand{\mf}[1]{\mathfrak{#1}}
\newcommand{\ms}[1]{\mathsmaller{#1}}

\DeclareMathOperator{\GL}{GL}

\DeclareMathOperator{\Aut}{\mathrm{Aut}}

\DeclareMathOperator{\Iso}{Iso}
\DeclareMathOperator{\Der}{Der}
\DeclareMathOperator{\ad}{Ad}
\DeclareMathOperator{\grb}{gr_{\bullet}}
\DeclareMathOperator{\Ima}{Im}
\DeclareMathOperator{\id}{id}
\DeclareMathOperator{\codim}{codim}
\DeclareMathOperator{\rk}{rank}
\DeclareMathOperator{\End}{End}
\DeclareMathOperator{\Mod}{Mod}
\DeclareMathOperator{\VB}{VB}
\DeclareMathOperator{\Lie}{Lie}
\DeclareMathOperator{\Stab}{Stab}
\newcommand{\Cher}[1][c]{\mathcal{H}_{1, #1, X, G}}
\DeclareMathOperator{\supp}{supp}
\DeclareMathOperator{\fflat}{flat}

\DeclareMathOperator{\Spf}{Spf}

\newcommand{\OO}{\mathcal{O}}

\DeclareMathOperator{\pr}{pr}
\DeclareMathOperator{\adlie}{ad}

\DeclareMathOperator{\Loc}{Loc}
\DeclareMathOperator{\gr}{Gr}
\DeclareMathOperator{\Sym}{Sym}
\DeclareMathOperator{\Sh}{Sh}
\newcommand{\pd}[1]{\frac{\partial}{\partial #1}}
\newcommand{\pdf}[2]{\frac{\partial #1}{\partial #2}}

\newcommand{\ratCherhat}{\widehat{H}_{1, c}(\bb C^l, H)}

\newcommand{\invlim}{\varprojlim}
\renewcommand{\*}{\ast}
\DeclareMathOperator{\res}{res}
\newcommand{\Cs}{\underline{\bb C}}
\newcommand{\coor}[1]{{#1}^{\textrm{coord}}}
\newcommand{\aff}[1]{{#1}^{\textrm{aff}}}
\newcommand{\presup}[2]{\prescript{\ms{#1}}{}{\mkern-3.0mu #2}}
\newcommand{\Csf}{\mathsf{C}}
\newcommand{\A}{\pmb{\alpha}}
\newcommand{\B}{\pmb{\beta}}
\newcommand{\E}{\pmb{\epsilon}}
\newcommand{\M}{\pmb{\mu}}
\newcommand{\Ga}{\pmb{\gamma}}
\newcommand{\Ii}{\pmb{I}}
\newcommand{\Jj}{\pmb{J}}
\newcommand{\HC}{\mc{A}_{n-l, l}^H}
\newcommand{\fl}{\coor{\pi}_{\*}\OO_{\textrm{flat}}(\coor{\mc{N}}\times\HC)}
\newcommand{\one}{\coor{\pi}_{\*}\OO_{\textrm{flat}}(\coor{\mc{N}}\times\mc{A}_{n-1, 1}^H)}
\DeclareMathOperator{\ind}{ind}

\DeclareMathOperator{\Indu}{\mathbf{Ind}}
\DeclareMathOperator{\Induc}{\mathbf{\mf{Ind}}}
\DeclareMathOperator{\Sheaf}{Sheaf}
\begin{document}
\title{Construction of Sheaves of Cherednik Algebras via Formal Geometry}
\author{Alexander Vitanov}
\email{\textsf{avitanov@protonmail.com}}
\begin{abstract}
In \cite{eti04} Etingof introduces a sheaf of Cherednik algebras $\mc{H}_{1, c, X, G}$, attached to a complex algebraic variety $X$ with an action by a finite group $G$, by means of generators and relations. In this note we realize the sheaf of Cherednik algebras on a general good complex orbifold $X/G$ by gluing sheaves of flat sections of flat holomorphic vector bundles on orbit type strata in $X$ which result from a localization procedure as explained in \cite{BK04}. In the case, when $c$ is formal, this construction can be interpreted as a formal deformation of $\mc{D}_X\rtimes\bb CG$ via Gel'fand-Kazhdan formal geometry.\\
Contrary to the original definition of $\mc{H}_{1, c, X, G}$ the presented construction permits the computation of trace densities, Hochschild homologies and an algebraic index theorem for formal deformations of $\mc{D}_X\rtimes\bb CG$. We also hope that the methods developed here will contribute towards a full proof of Dolgushev-Etingof's conjecture \cite{dol05}.  
\end{abstract} 
\maketitle

\tableofcontents

\section{Introduction}
In this paper we realize the sheaf of Cherednik algebras $\mc{H}_{1, c, X, G}$ on a complex global quotient orbifold $X/G$ by merging $\Cs$-algebras of flat sections of flat holomorphic bundles on orbit type strata in $X$ which result from localization procedures with respect to Harish-Chandra torsors and modules. Our construction does not depend on whether the formal parameter $c$ is complex valued or formal. Moreover, when $c$ is taken to be formal, our work generalizes Tang and Halbout's global symplectic reflection algebra on a symplectic orbifold $M/\bb Z_2$ to the case of an arbitrary group. The price that we pay for the generalization of the group action is that we give up the symplectic form and work instead on a complex $G$-manifold $X$. In \cite{HT12} Halbout and Tang quantize a $\bb Z_2$-invariant $\epsilon$-tubular neighborhood, $\epsilon>0$, of the zero section of the normal bundle to the fixed point submanifold $M_2^{\gamma}$ of codimension $2$. They then obtain via a push-forward along the exponential map a  formal deformation of the $\bb Z_2$-invariant quantum functions on an $\epsilon$-tubular neighborhood $B_{\epsilon}$ of $M_2^{\gamma}$. Their key idea is to glue the restriction of this deformation to the punctured neighborhood $B_{\epsilon}\setminus M_2^{\gamma}$ with Fedosov's standard quantization of the complement $M^{-}:=M\setminus M_2^{\gamma}$. There are two basic difficulties which obstruct a naive generalization of this approach to the case of a general finite group $G$. The first one is technical and deals with the fact that the restriction of the quantization of a tubular neighborhood of a fixed point submanifold is possible only if the the star product is local with respect to the $G$-action which in general is difficult to verify. The second one is more fundamental and deals with the fact that different fixed point submanifolds of $G$ may intersect. For these cases we need more intricate gluing procedure than a naive gluing of deformations on the fixed point-free subspace of the manifold. We are able to resolve both of these problems for the case of formal deformation of $\mc{D}_X\rtimes\bb CG$. 
We partition the $G$-manifold $X$ in orbit type strata. For each stratum $X_H^i$, we define a special Harish-Chandra module $\mc{A}_{n-l, l}^H$ and apply Gel'fand-Kazhdan's formal geometry for holomorphic vector bundles to get a Maurer-Cartan connection $(1, 0)$-form on the formal coordinate bundle attached to the normal bundle to $X_H^i$. We then prove a formal holomorphic tubular neighborhood theorem for analytic sets and use it to prove that the sheaves obtained via localization of the Harish-Chandra modules $\mc{A}_{n-l, l}^H$ yield sheaves of deformations of the skew-group algebra of holomorphic differential operators on formal neighborhoods of the strata. Since the products are manifestly local with respect to the group action, we completely resolve the problem of locality. Next, we define a series of necessary and sufficient conditions for the merging of the sheaves, obtained through localization, along the strata of different codimension. In particular, these conditions take under consideration the  general case of intersection of fixed point submanifolds. 
It is expected by the author that the work done here will contribute to the proof of the still unresolved Dolgushev-Etingof's conjecture. On more practical level, the representation of $\mc{H}_{1,c, X, G}$ in terms of flat sections of flat vector bundles makes the computation of trace densities, Hochschild homologies and algebraic index theorems for the sheaf of formal global Cherednik algebras accessible by means of standard methods as in \cite{RT12}, for instance. We deal with these computations in an upcoming paper. 
\section{Preliminaries}
\subsection*{Complex reflections in linear spaces}
\addtocontents{toc}{\protect\setcounter{tocdepth}{1}}
Let $\mathfrak{h}$ be a finite-dimensional complex vector space and let $\mathfrak{h}^*$ be the dual space of $\mathfrak{h}$.  A semisimple endomorphism $s$ of $\mathfrak{h}$ is called a \emph{complex reflection} in $\mathfrak{h}$ if $\rk(\id_{\mathfrak{h}}-s)=1$. 
The fixed point subspace $\mathfrak{h}^s:=\ker(\id_{\mathfrak{h}}-s)$ of the complex reflection $s\in\End(\mathfrak{h})$ is a hyperplane, which is referred to as the \emph{reflecting hyperplane} of $s$. If $a\in\mathfrak{h}$ spans the one-dimensional $\Ima(\id_{\mathfrak{h}}-s)$, there is a linear form $a^*\in\mathfrak{h}^*$ with $\ker(a^*)=\mathfrak{h}^s$ such that 
$s(v)=v-(v, a^*)a$ for all $v\in\mathfrak{h}$ 
where $(\cdot, \cdot)$ is the natural pairing between $\mathfrak{h}$ and $\mathfrak{h}^*$. 
Suppose $G$ is a finite subgroup of $\GL(\mathfrak{h})\subset\End(\mathfrak{h})$ and let $\mathcal{S}$ denote the set of all complex reflections in $\mathfrak{h}$ contained in $G$. The group $G$ is said to be a \emph{complex reflection group} if it is generated by $\mathcal{S}$. 
The group $G$ acts naturally on $\mathfrak{h}^*$ via the dual representation thanks to which
every complex reflection $s\in G$ defines a unique complex reflection $s^*$ in $\mathfrak{h}^*$. For clarity of the exposition we shall use the same notation $s\in G$ for complex reflections in $\mathfrak{h}$ and their unique counterparts in $\mathfrak{h}^*$.\\
Given a complex reflection $s\in\mathcal{S}$ in $\mathfrak{h}$ we denote its unique non-trivial eigenvalue by $\lambda_s^{\vee}$ and by $\alpha_s^{\vee}\in\mathfrak{h}$ an eigenvector of $s$ corresponding to $\lambda_s^{\vee}$ which we call a \emph{root}. The root generates the image of $\id_{\mathfrak{h}}-s$ and vanishes everywhere on the reflecting hyperplane of $s$ in $\mathfrak{h}^*$. 
Similarly, $\lambda_s$ designates the unique non-trivial eigenvalue of $s$ in $\mathfrak{h}^*$ and $\alpha_s\in\mathfrak{h}^*$ an eigenvector of $s$ in $\mathfrak{h}^*$ corresponding to $\lambda_s$, which we call \emph{coroot}. The eigenvalues of the root and coroot are related via $\lambda_{s}^{\vee}=\lambda_{s}^{-1}$. Analogously to $\alpha_s^{\vee}$ the linear form $\alpha_s$ generates the image of $\id_{\mathfrak{h}^*}-s$ and vanishes identically on the hyperplane of $s$ in $\mathfrak{h}$. 
From now on we assume that for any complex reflection $s$ the corresponding root $\alpha_s^{\vee}\in\mathfrak{h}$ and coroot $\alpha_s\in\mathfrak{h}^*$ are normalized such that $(\alpha_s, \alpha_s^{\vee})=2$. 
Note that this normalization assumption implies the following formulae for any complex reflection $s\in G$:
\begin{align}
&sv=v+\frac{1-\lambda_s}{2\lambda_s}(v, \alpha_s)\alpha_s^{\vee},~~~~~~~~~~~~~\text{for all}~~ v\in\mathfrak{h}\label{relation1}\\  
&\nonumber\\
&sy=y-\frac{1-\lambda_s}{2}(\alpha_s^{\vee}, y)\alpha_s,~~~~~~~~~~~~~~\text{for all}~~ y\in\mathfrak{h}^*.\label{relation2}
\end{align}
\subsection*{Rational Cherednik algbera}
\label{rational}
To the data $\mathfrak{h}$, $\mathfrak{h}^*$ and $G$ one attaches the rational Cherednik algebra $H_{t, c}(\mathfrak{h}, G)$:
\begin{definition}
\label{cherednik}
The rational Cherednik algebra $H_{t, c}(\mathfrak{h}, G)$ is defined as the quotient of the smash-product algebra $T^{\bullet}(\mathfrak{h}\oplus\mathfrak{h}^*)\rtimes\mathbb{C}G$ by the ideal generated by
\begin{align*}
&gyg^{-1}-\presup{g}{y},\quad\qquad~gug^{-1}-\presup{g}{u},\quad\qquad~[y, y'],\quad\qquad~[u, u'],\quad\qquad~[u, y]-t\left(u, y\right)-\sum_{s\in\mathcal{S}}c(s)\left(u, \alpha_s\right)\left(y, \alpha_s^{\vee}\right)s,
\end{align*} 
where $u, u'\in\mathfrak{h}$ and $y, y'\in\mathfrak{h}^*$ and $c\in\mathbb{C}[\mathcal{S}]^{\ad G}$, where $\ad$ refers to the adjoint action of $G$ on itself and $t\in\mathbb{C}$.
\end{definition}
The algebra $H_{t, c}(\mathfrak{h}, G)$ comes with two natural \emph{increasing filtrations}. The first one is the \emph{Bernstein fitration} $\widetilde{F}^{\bullet}$ which shall not be dealt with in this work.  
The second one is the \emph{geometric filtration} $F^{\bullet}$, given by the rule $\deg(\mathfrak{h}^*)=\deg(G)=0$, $\deg(\mathfrak{h})=1$. Let $\mathfrak{m}$ be a maximal ideal in $\bb{C}[\mathfrak{h}]$. 
The \emph{degree-wise completion} $\widehat{H}_{t, c}(\mathfrak{h}, G)$ of $H_{t, c}(\mathfrak{h}, G)$ as a $\bb C[\mf{h}]$-module is by definition the scalar extension of $H_{t, c}(\mathfrak{h}, G)$ to formal functions on the formal neighborhood of zero in $\mathfrak{h}$, that is, $\widehat{H}_{t, c}(\mathfrak{h}, G):=\bb{C}[[\mathfrak{h}]]\otimes_{\bb{C}[\mathfrak{h}]}H_{t, c}(\mathfrak{h}, G)$. We remark that since the underlying $\bb C[\mf{h}]$-module of the  rational Cherednik algebra is not finitely generated over the Noetherian ring $\bb{C}[\mathfrak{h}]$, the degree-wise formal completion does not coincide with the formal completion of the Cherednik algebra as a $\bb C[\mf{h}]$-module with respect to the induced $\mathfrak{m}$-adic topology on the $\bb C[\mf{h}]$-module $H_{t, c}(\mathfrak{h}, G)$, that is, the equality $\lim_{\leftharpoonup\\i}H_{t, c}(\mathfrak{h}, G)/\mathfrak{m}^iH_{t, c}(\mathfrak{h}, G)\cong\widehat{H}_{t, c}(\mathfrak{h}, G)$ does not hold true.
 \begin{remark} The degree-wise completion $\bb C[[\mf{h}]]\otimes_{\bb C[\mf{h}]}H_{1, c}(\mf{h}, G)$ of the rational Cherednik algebra has a natural structure of a $\bb C$-algebra extending the one of the rational Cherednik algebra. 
 \end{remark}
For expository purposes we shall abuse notation throughout the paper denoting the degree-wise completion of the rational Cherednik algebra by $\widehat{H}_{t, c}(\mathfrak{h}, G)$. 
This completion inherits the increasing filtration $F^{\bullet}$ from $H_{t, c}(\mathfrak{h}, G)$ by the rule $F^i\widehat{H}_{t, c}(\mathfrak{h}, G):=\bb{C}[[\mathfrak{h}]]\otimes_{\bb{C}[\mathfrak{h}]}F^iH_{t, c}(\mathfrak{h}, G)$. 
\subsubsection*{Deformation theory of the rational Cherednik algebra}
In this paper we restrict our attention to rational Charednik algebras with $t\in\bb{C}^{\times}$.  Since for any $\lambda\neq0$ the assignment $u\mapsto\lambda u$, $y\mapsto\lambda y$ and $g\mapsto g$ induces an algebra isomorphism $H_{t, c}(\mathfrak{h}, G)\cong H_{\lambda^2 t, \lambda^2 c}(\mathfrak{h}, G)$, we can assume without any loss of generality $t=1$. According to \cite[Theorem 1.3]{EG02} the Cherednik algebra is a $PBW$-deformation of $\Sym^{\bullet}(\mf{h}\oplus\mf{h}^{\*})\rtimes\bb CG$.
Furthermore, as $HH^2(\mathcal{D}(\mathfrak{h})\rtimes\bb{C}G, \mathcal{D}(\mathfrak{h})\rtimes\bb{C}G)$ is isomorphic to the space $\bb{C}[\mathfrak{h}]^{\ad G}$ and the odd Hochschild cohomology of $\mathcal{D}(\mathfrak{h})\rtimes\bb{C}G$ vanisches, the first order deformations of $H_{1, 0}(\mathfrak{h}, G)=\mathcal{D}(\mathfrak{h})\rtimes\bb{C}G$ are unobstructed. Thus there exists a universal formal deformation of $\mathcal{D}(\mathfrak{h})\rtimes\bb{C}G$, parametrized by $\bb{C}[\mathfrak{h}]^{\ad G}$. It was shown in \cite[Theorem 2.16]{EG02} that $H_{1, \hbar}(\mathfrak{h}, G)$, when $\hbar$ a formal parameter, is a universal formal deformation of $H_{1, 0}(\mathfrak{h}, G)=\mathcal{D}(\mathfrak{h})\rtimes\bb{C}G$. 
\subsection*{Dunkl operator embedding}
Let $\mathfrak{h}_{reg}:=\mathfrak{h}\diagdown\cup_{s\in\mathcal{S}}\alpha_s^{-1}(0)=\{\xi\in\mathfrak{h}: \Stab(\xi)=id_{\mathfrak{h}}\}$. This is a Zariski open $G$-invariant subset of $\mathfrak{h}$, on which $G$ acts freely. Let $\delta:=\prod_{s\in\mathcal{S}}\alpha_s\in\bb{C}[\mathfrak{h}]$ be the polynomial, called \emph{discriminant}, which vanishes on the union of all reflection planes in $\mathfrak{h}$. Then the algebra of differential operators on $\mathfrak{h}_{reg}$, $\mathcal{D}(\mathfrak{h}_{reg})$, is isomorphic to the localization $\mathcal{D}(\mathfrak{h})[\delta^{-1}]$ of the Weyl algebra at the multiplicative set $\{\delta^i | i\in\bb{Z}_{\geq0}\}$. The $G$-action on $\mathfrak{h}_{reg}$ extends naturally to an action by algebra automorphisms on  $\mathcal{D}(\mathfrak{h}_{reg})$, allowing us to define the smash product $\mathcal{D}(\mathfrak{h}_{reg})\rtimes\bb{C}G$. To any given vector $\xi\in\mathfrak{h}$ there is an associated operator $D_{\xi}\in\mathcal{D}(\mathfrak{h}_{reg})\rtimes\bb{C}G$ defined by the formula
\begin{equation*}
\label{do}
D_{\xi}:=\partial_{\xi}-\sum_{s\in\mathcal{S}}\frac{2c(s)}{1-\lambda_s}\frac{(\xi, \alpha_s)}{\alpha_s}(1-s)
\end{equation*}
where $\partial_{\xi}$ is the directional derivative of $\xi$ and $\sum_{s\in\mathcal{S}}\frac{2c(s)}{1-\lambda_s}\frac{(\xi, \alpha_s)}{\alpha_s}(1-s) \in\bb{C}[\mathfrak{h}_{reg}]\rtimes\bb{C}G$. A difference operator of the above form is called a \emph{Dunkl operator}. 
Using the filtration of $H_{t, c}(\mathfrak{h}, G)$ and the differential operator filtration on $\mathcal{D}(\mathfrak{h}_{reg})\rtimes\bb{C}[G]$ one has shown in \cite[Proposition 4.5]{EG02} that for any class function $c\in\bb{C}[\mathcal{S}]^{\ad G}$ the assignments $
\mathfrak{h}^*\ni y\mapsto y,~~~\mathfrak{h}\ni \zeta\mapsto D_{\zeta},~~~G\ni g\mapsto g$,
induce an injective algebra homomorphism $\grb(H_{1, c}(\mathfrak{h}, G))\hookrightarrow \grb(\mathcal{D}(\mathfrak{h}_{reg})\rtimes\bb{C}G)$ and hence an algebra embedding, called \emph{Dunkl embedding}, 
\begin{equation}
\label{dunkl}
\Theta_c: H_{1, c}(\mathfrak{h}, G)\hookrightarrow\mathcal{D}(\mathfrak{h}_{reg})\rtimes\bb{C}G.
\end{equation}
\subsection*{Global Cherednik algebra}
\label{global}
Assume that $X$ is a $n$-dimensional complex manifold equipped with an action by a finite group $G$ of holomorphic automorphisms of $X$. Since the action of $G$ is then properly discontinuous, the quotient $X/G$ is a complex orbifold. Let $X^g\subset X$ denote the fixed point set of $g\in G$. A \emph{nonlinear complex reflection} of $X$ is a pair $(g, Y)$ consisting of a group element $g\in G$ and a connected component $Y$ of $X^g$ of complex codimension $1$ in $X$ which we will refer to as \emph{reflection hypersurface} in accord with the terminology established in \cite{eti04}. \\
The following paragraph is by and large a shortened rehash of Section 1.1 of \cite{FT17}. We define the analogue of Dunkl-Opdam operators for complex reflection representations in the case of a complex manifold with a finite group action. Let us denote by $\mathrm{S}$ the set of all complex reflections of $X$ and let $c: \mathrm{S}\longrightarrow\bb{C}$ be a $G$-invariant function. Let $D:=\bigcup_{(g, Y)\in\mathrm{S}}Y$ and let $j: X\backslash D\longrightarrow X$ be the open inclusion map. For each complex reflection $(g, Y)\in\mathrm{S}$ let $\mathcal{O}_X(Y)$ designate the sheaf of holomorphic functions on $X\backslash Y$ taking poles of at most first order only along $Y$, let $\xi_{(g, Y)}: \mc{T}_X\rightarrow\OO_X(Y)/\OO_X$ be the natural surjective map of $\OO_X$-modules, and let $p: X\longrightarrow X/G$ denote the projection. A \emph{Dunkl operator associated to a holomorphic vector field $Z$ on $X$ } is a section $\mathbb{D}_Z$ of the sheaf $p_*j_*j^*(\mathcal{D}_X\rtimes\bb{C}G)$ over an open subset $U\subset X/G$, which in a $G$-invariant coordinate chart $U'\subset p^{-1}(U)\subset X$ has the form
\begin{equation} 
\mathbb{D}_Z=\mathcal{L}_Z +\sum_{(g, Y)\in\mathrm{S}}\frac{2c((g, Y))}{1-\lambda_{(g, Y)}}f_{(g, Y)}(g-1).    
\end{equation}
Here $\mathcal{L}_Z$ is the Lie derivative with respect to $Z$, $\lambda_{(g, Y)}$ is the nontrivial eigenvalue of $g$ on the conormal bundle to $(g, Y)$ and $f_{(g, Y)}\in\Gamma(U', \mathcal{O}_X(Y))$ is a function whose residue agrees with $V$ once both are restricted to the normal bundle of $(g, Y)$ in $X$, that is $f_{(g, Y)}\in\xi_{(g, Y)}(V)$. 
To the data $X$ and $G$ one attaches the following sheaf of non-commutative associative algebras: 
\begin{definition}
The sheaf of global Cherednik algebras $\Cher$ on the orbifold $X/G$ is a subsheaf of the sheaf $p_*j_*j^*(\mathcal{D}_X\rtimes\bb{C}G)$ generated locally by $p_{\*}j_{\*}j^{\*}\mathcal{O}_X$, $\bb{C}G$ and Dunkl operators $\mathrm{D}_Z$ associated to holomorphic vector fields $Z$ on $X$.  
\end{definition}
Note that the definition of $\mc{H}_{1,c, X, G}$ is independent on the choice of a function $f_{(g, Y)}\in\Gamma(U', \mathcal{O}_X(Y))$ in the Dunkl operators. 
The sheaf of Cherednik algebras $\Cher$ possesses a natural increasing and exhaustive filtration $\mc{F}^{\bullet}$ which is defined on the generators by $\deg(\mathcal{O}_X)=\deg(\bb{C}G)=0$ and $\deg(\bb{D}_Z)=1$ for Dunkl operators $\bb{D}_Z$ , $Z\in\mathcal{V}(X)$. It is the analogue of the geometric filtration of the rational Cherednik algebra.
\subsubsection*{Deformation theory of the sheaf of Cherednik algebras}
By \cite[Theorem 2.17]{eti04} the sheaf $\mc{H}_{1, c, X, G}$ satisfies a $PBW$-property, i.e. there is an isomorphism of $\Cs$-algebras $\gr_{\mc{F}^{\bullet}}(\mc{H}_{1, c, X, G})\cong\Sym^{\bullet}(\mc{T}_X)\rtimes\bb CG$. The algebra of holomorphic differential operators is a locally convex topological algebra. The Hochschild (co)homology for such algebras is defined the same way as for non-topological algebras by replacing everywhere in the respective definitions the algebraic tensor products with completed projective tensor products $\hat{\otimes}_{\pi}$.
By  \cite[Corollary 2.4]{eti04} we have that $HH^2(D(U)\rtimes\bb{C}G, D(U)\rtimes\bb{C}G)\cong(\oplus_{(g, Y)\in\mathrm{S}}H^{0}(Y\cap U, \bb{C}))^G\oplus H^2(U,\bb{C})^G$ and $HH^3(D(U)\rtimes\bb{C}G, D(U)\rtimes\bb{C}G)\neq0$. Nevertheless, the formal deformations of $D(U)\rtimes\bb{C}G$ are unobstructed and \cite[Theorem 2.23]{eti04} implies that 
for a family of formal parameters $\{c((g, Y))\}$, the $\Cs$-algebra $\Cher$ is a formal deformation of $\mathcal{H}_{1, 0, X, G}$.
\section{Stratification of $X$ by orbit types}
\label{orbittypes}
The \emph{fixed point set} with respect to a subgroup $H$ of $G$ is the set $\{x\in X~|~H\subset\Stab(x)\}$. It is a closed (not necessarily connected) complex submanifold of $X$. Two points $x, y\in X$ are said to be of the same \emph{isotropy type} if $\Stab(x)=\Stab(y)$. This definition give rise to equivalence relations on $X$ whose equivalence classes are called \emph{isotropy types}. For every subgroup $H$ of $G$ we shall call the according isotropy type $X_H:=\{x\in X~|~\Stab(x)=H\}$ $H$-\emph{isotropy type}. In general a $H$-isotropy type might be empty. The connected components thereof are locally closed complex submanifolds of $X$. By the continuity of the inclusion map $X_H\hookrightarrow X^H$, each connected component of $X_H$ is contained in a connected component of $X^H$. Moreover, according to the invariance domain theorem the image of $X_H^i$ in $X_i^H$ under the continuous iinclusion mapping is open in $X_i^H$ whence $X_H$ is an open submanifold of $X^H$.  
Two distinct point $x$ and $y$ in $X$ are said to be of the same \emph{orbit type} if $\Stab(x)$ and $\Stab(y)$ are conjugate subgroups of $G$, that is, there is an element $g\in G$ such that $\Stab(y)=g\Stab(x)g^{-1}$. This gives rise to an equivalence relations on $X$ whose equivalence classes are called \emph{orbit types}. Given a subgroup $H$ of $G$, denote its conjugacy class in $G$ by $(H)$. Then the corresponding orbit type, which is defined as $X_{(H)}:=\{x\in X~|~\Stab(x)=gHg^{-1}~\textrm{for some}~g\in G\}$, is called \emph{$(H)$-orbit type}. The relation between an $K$-isotropy type with $K\in(H)$ and an $(H)$-orbit type is given by the equality   
\begin{align}
\label{orbit/iso}
X_{(H)}&=\{x\in X~|~\Stab(x)=gHg^{-1}~\textrm{for some}~g\in G\}\nonumber\\
&=\{gx~|~x\in X_K, g\in G\}\nonumber\\
&=\coprod_{g\in G\ N_G(K)}g\cdot X_K\nonumber\\
&=\coprod_{g\in G\ N_G(K)}\coprod_{i_g\in I_g}X_{g^{-1}Kg}^{i_g}
\end{align}
where $N_G(H)$ is the normaliser of $H$ in $G$ and $I_g$ is the index set labeling the connected components of $X_{g^{-1}Kg}$. The latter computation shows that the $(H)$-orbit type is the smallest subset of $X$ which simultaneously is $G$-invariant and contains all translates of $X_K$. Moreover, since per assumption $G$ is finite, hence every subgroup $H$ thereof is also finite, by Proposition 2.4.4 in \cite{OR04} the $K$-isotropy type is closed in $X_{(H)}$ for every subgroup $K$ of $G$ in $(H)$. This observation is of crucial importance for the proof of the following  lemma.
\begin{lemma}
\label{orbit-isotropytype}
The connected components of an $(H)$-orbit type are the $G/N_G(K)$-translates of the connected components of the $K$-isotropy type $X_K$ for some fixed subgroup $K$ in the conjugacy class $(H)$.
\end{lemma}
\begin{proof}
Fix a subgroup $K$ of $G$ in the conjugacy class $(H)$. Since the group $G$ is finite, the number $[G:N_G(K)]$ of translates of the $K$-isotropy type $X_K$ inside  of the $(H)$-orbit type $X_{(H)}$ is finite. Thus by virtue of the aforementioned fact  that the $K$-isotropy type is closed in the $(H)$-orbit type, we infer that $X_K$ is open in $X_{(H)}$. Moreover, as a topological manifold $X_K$ is a locally connected space, hence its connected components are open inside of $X_K$. Thus, by transitivity each connected component of $X_K$ is open in $X_{(H)}$. Finally,  Equality \eqref{orbit/iso} implies that the $(H)$-orbit type is the disjoint union of connected and open in $X_{(H)}$ sets $X_{g^{-1}Kg}^{i_g}$, $g\in G/N_G(K)$, which completes the proof of the lemma.    
\end{proof}
Observe that if $J$ is the index set denoting the connected components of the $(H)$-orbit type in $X$ and $I$ is the index set labeleing the connected components of the $K$-isotropy type for some $K\in(H)$, then Equality  \eqref{orbit/iso} implies $|J|=[G:N(H)]|I$, where $|\cdot|$ denotes the cardinality of the set, whereas $[G:N(H)]$ is the multiplicity of $N(H)$ in $G$. We recall the definition of a stratification of a general complex space.
\begin{definition}[Definition 1.1, p.67, \cite{camp10}]
\label{strat}
A complex stratification of a complex space $\mc{M}$ is a countable, locally finite covering of $\mc{M}$ by disjoint complex subspaces, called strata $\bb S=(S_{\gamma})_{\gamma\in\Gamma}$ with the following properties
\begin{enumerate}
\item[i)] Each stratum $S_{\gamma}$ is a locally closed submanifold of $X$ that is Zariski-open in its closure.
\item[ii)] The boundary $\partial(S_{\gamma})=\overline{S_{\gamma}}\setminus S_{\gamma}$ of each stratum $S_{\gamma}$, where $\overline{S_{\gamma}}$ is the closure of $S_{\gamma}$ in $X$ and $\mathring{S}_{\gamma}$ is its interior, is a union of strata of (strictly) lower dimension.
\end{enumerate} 
\end{definition}    
Condition $ii)$ in the above definition is equivalent to the \emph{frontier condition}, according to which whenever  $S_{\gamma'}\cap\overline{S}_{\gamma}\neq\varnothing$ for $\gamma'\neq\gamma$ holds, we have that  $S_{\gamma'}\subset\overline{S}_{\gamma}$. Recall that the dimension of a stratified space $\mc{M}$, $\dim(\mc{M})$, is defined as $\sup_{\gamma\in\Gamma}\dim(S_{\gamma})$. If $|\Gamma|$ is finite, the stratification is called \emph{finite} and $\dim(\mc{M})=\max_{\gamma\in\Gamma}\dim(S_{\gamma})$. The ensuing lemma demonstrates how by means of the frontier condition we can get filtration of $\mc{M}$ in terms of open subsets. 
\begin{lemma} 
\label{openfiltration}
Suppose $\mc{M}$ is a complex manifold with a finite stratification $(S_{\gamma})_{\gamma\in\Gamma}$, such that $0$ is the lowest dimension of a stratum in $\mc{M}$. Let $F_n(\mc{M}):=\coprod_{\codim(S_{\gamma})\leq n}S_{\gamma}$ for every $0\leq n\leq\dim\mc{M}$. Then, $F_0(\mc{M})\subseteq F_1(\mc{M})\subseteq\dots\subseteq F_{\dim(\mc{M})}(\mc{M})=\mc{M}$ is an open finite filtration of $\mc{M}$. 
\end{lemma}
\begin{proof}
The ascendence of the chain of the above defined sets $F_n(\mc{M})$ follows directly from their definition. It remains to show that $F_n(\mc{M})$ is open for all $n\in\{0, \dots, \dim\mc{M}\}$. Clearly,  $F_0(\mc{M})\neq\varnothing$ is an open submanifold of $\mc{M}$. Fix $1\leq n\leq\dim(\mc{M})$. The complement, $F_n(\mc{M})^c$, of $F_n(\mc{M})$ in $\mc{M}$ is $\coprod_{n+1\leq \codim(S_{\gamma})}S_{\gamma}$. Then, due to the finiteness of the stratification we have that
\begin{align*}
\overline{F_n(\mc{M})^c}&=\bigcup_{n+1\leq \codim(S_{\gamma})}\overline{S_{\gamma}}\\
&=\bigcup_{n+1\leq \codim(S_{\gamma})}S_{\gamma}\bigcup\partial(S_{\gamma})\\
&=F_n(S)^c
\end{align*}
where in the second to the last line we used the frontier condition. Thus, $F_n(\mc{M})^c$ is closed and consequently $F_n(\mc{M})$ is open within $\mc{M}$ for every $0\leq n\leq\dim(\mc{M})$. \end{proof}
Note that the closure of a stratum in the complex manifold is a closed submanifold which by definition locally is the zero locus of finitely many holomorphic functions. Hence, the closure of a stratum is an analytic set in the manifold. In general, the partition of a complex manifold, equipped with a finite group action, into orbit types defines stratification, which is referred to as \emph{stratification by orbit types}.
\begin{theorem}[Definition] The connected components of the orbit type strata $X_{(H)}$ and their projections onto $X_{(H)}/G$ constitute a complex stratification of $X$ and $X/G$, respectively.   
\end{theorem}
Although $X/G$ in the above theorem is not a manifold in general, since the action of the finite group $G$ is not required to be free, $X/G$ has the structure of a complex orbifold which is an example of a complex space. This fact justifies the generality in which Definition \ref{strat} is formulated.  By virtue of Lemma \ref{orbit-isotropytype} 
every orbit type stratum $X_{(H)}^i$ in the manifold $X$ uniquely coincides with a connected component of some isotropy type $X_K^j$ with $K\in(H)$. That is the reason why throughout the paper we shall abuse language by calling the connected components of the isotropy types orbit type strata. The frontier condition of a stratification implies that whenever $X_H^i\cap\overline{X_K^j}\neq\varnothing$, then $\codim(X_K^j)<\codim(X_H^i)$. Also since $X_K^j\subset X_j^K$, it follows that $X_H^i\subset\overline{X_K^j}\subset\overline{X_j^K}=X_j^K$ whence $K$ is a subgroup of $H$.  
\subsection{Isotropy and slice representation along connected components of isotropy types}
\addtocontents{toc}{\protect\setcounter{tocdepth}{2}}
We recall that the differential of the action of the isotropy group $\Stab(x)$ of a point $x$ in $X$ defines a linear representation $\tau_x$ of $\Stab(x)$ on the holomorphic tangent space $T_x^{(1, 0)}X$ called the \emph{isotropy representation} at $x$. In the following we study properties of the isotropic representation along the connected components of an $H$-isotropy type for subgroups $H$ of $G$.
We begin by reminding the reader of the following classical result from Complex Geometry due to H. Cartan.   
\begin{lemma}[Cartan's Lemma, \cite{C57}]
Let $X$ be a complex manifold and let $G$ be a finite group acting on $X$ by biholomorphisms. Then, if $x\in X$ is a fixed point of $G$, there exists $G$-invariant open neighborhoods $U\subset X$
of $x$ and $V\subset T_x^{(1, 0)}X$ of the origin $0$ in $T_x^{(1, 0)}X$ along with a $G$-equivariant biholomorphism $\varphi: U\rightarrow V$ with $\varphi(x)=0$.
\end{lemma}
Cartan's Lemma implies that the $G$-action can be linearlized in the vicinity of a fixed point of $G$. Its application can be extended to any subgroup of $G$ which has a fixed point on $X$. Cartan's Lemma is instrumental in showing that at every point $x\in X$ the isotropy representation is faithful. 
\begin{lemma}
\label{fixpointset}
Let $H$ be a subgroup of $G$ and let $X_H^i$ be the $i$-th connected component of the $H$-isotropy type submanifold $X_H$. The isotropy representations of $H$ at all points of $X_H^i$ are equivalent to the same faithful linear representation $\sigma: H\rightarrow\GL(\bb C^n)$. 
\end{lemma}
\begin{proof}
The proof of this lemma is in essence a straightforward  adaptation of the proof of \cite[Lemma 1.1]{stewart63}. It uses elementary ideas from point set topology and is a consequence of Cartan's Lemma. We leave the details to the reader.
\end{proof}
Note that the representation $\sigma$ differ from one connected component of an isotropy type to another. We shall call the restriction of the isotropy representation at a point $x\in X$ to the unique $\Stab(x)$-invariant complement $N_x$ of the isotypic component $\big(T_x^{(1, 0)}X\big)^H$ the \emph{slice representation} at $x$. From the lemma proven above we draw the following natural conclusion. 
\begin{corollary}
The slice representation of $H$ at every point of X$_H^i$ is equivalent to the same faithful linear subrepresentation $\sigma|_{\bb C^l}: H\rightarrow\GL(\bb C^l)$, where $l=\dim_{\bb C}N_x$ for all $x\in X_H^i$.
\end{corollary}
\begin{proof}
The statement follows from the fact that for each pair of distinct points $y, y'\in X_H^i$ according to Lemma \ref{fixpointset} there is  
an $H$-equivariant linear isomorphism $T_y^{(1, 0)}X\rightarrow T_{y'}^{(1, 0)}X$ which due to its commutativity with the $H$-action restricts to an $H$-equivariant linear isomorphism between the $H$-invariant complements $N_y$ and $N_{y'}$ of the isotypic components $\big(T_{y}^{(1, 0)}X\big)^H$ and $\big(T_{y'}^{(1, 0)}X\big)^H$, respectively.   
\end{proof}
We shall abuse language by calling $\sigma|_{\bb C^l}$ the \emph{slice representation of $H$ on the stratum $X_H^i$}. Furthermore, because $\sigma|_{\bb C^l}$ is an injective group homomorphism from $H$ to $\GL(l, \bb C)$, we shall often identify $H$ with its image $\sigma(H)$ in $\GL(l, \bb C)$ without explicitly mentioning it. In all such cases we shall abuse notation by treating $H$ as a subgroup of $\GL(l, C)$. 
\subsection{Basis for the $G$-equivariant topology of $X$ and coordinate slice charts for orbit type strata}
\label{basis/charts}
Recall that a \emph{slice} at a point $x\in X$ is a $\Stab(x)$-invariant neighborhood $W_x$ such that $W_x\cap gW_x=\varnothing$ for all $g\in G\setminus \Stab(x)$. A slice $W_x$ at a point $x$ with $\Stab(x)=H$ cannot contain any element $y$ whose stabilizer $\Stab(y)$ is not contained in $H$ since the contrary would imply that there is a group element $g\in G\setminus \Stab(x)$, for which $W_x\cap gW_x\neq\varnothing$, which would then contradict the assumption, that $W_x$ is a slice. This implies that all strata, crossing a $H$-invariant slice $W_x$, have an isotropy type contained in $H$. A slice is called a \emph{linear} if there is a $\Stab(x)$-invariant open set $V$ in $\bb C^n$ such that $W_x$ is $\Stab(x)$-equivariantly biholomorphic to $V$. If every point $x$ in $X$ has a linear slice, the group action of $G$ is referred to as \emph{locally linear}. Since in the case we consider the group $G$ is finite, it acts on $X$ properly discontinuously and thus each point $x$ in $X$ possesses a slice $W_x$. By Cartan's Lemma one can shrink the slice $W_x$ till the $\Stab(x)$-action is linearized which implies that the $G$-action is locally linear. It is of crucial importance for our work that a linear slice $W_x$ at a point $x$ on a stratum $X_H^i$ intersects apart from $X_H^i$ only strata $X_K^j$ of codimension lower than or equal to the codimension of $X_H^i$ such that $x$ lies in the intersection of their closures which by the frontier condition of the stratification means that $X_H^i\subseteq\bigcap\{\overline{X_K^j}~|~X_K^j\cap W_x\neq\varnothing\}$. 
Let $p:X\rightarrow X/G$ be the holomorphic projection. The $G$-equivariant topology of $X$ is by definition comprised of the preimages of open sets in the quotient topology of $X/G$. Since $p$ is surjective, $pp^{-1}(U)=U$ for any set $U$ in $X/G$. This means in particular that every open set $U$ in the orbifold is the image of some open set $V$ in $X$ under $p$ whence the $G$-equivariant topology of $X$ can be written as  
\begin{equation}
\label{equivtopology}
\mf{T}_X^G:=\{p^{-1}(p(V))=\bigcup_{g\in G}gV~|~V~\textrm{is an arbitrary open set in}~X\}.
\end{equation}
We note that the main difference of that topology as opposed to the standard one is that the $G$-equivariant topology is not Hausdorff. Next, we define a basis for the $G$-equivariant topology $\mf{T}_X^G$ on $X$. 
it is a matter of a straightforward verification that for any given open set $V$ in $X$ and every $x\in V$ with $\Stab(x):=H\leq G$, there is an $H$-invariant linear slice $W_x\subset V$ which is $H$-equivariantly biholomorphic to a box in $\bb C^{n-l}\times\bb C^l$ where $\bb C^{n-l}$ is the fixed point subspace of $\bb C^n$ with respect to $H$. Hence, $\ind_H^G(W_x):=\coprod_{g\in G/H}gW_x\subset\bigcup_{g\in G}gV$. Hence, the set 
\begin{equation*}
\label{equivbasis}
\mf{B}_X^G:=\Bigg\{\ind_H^G(W_x)~|~\begin{matrix}\mathsmaller{W_x~\textrm{a $H$-invariant linear slice,}}\\\mathsmaller{\textrm{biholomorphic to a box in $\bb C^{n-l}\times\bb C^l$ for all $H\leq  G$}}\end{matrix} \Bigg\}
\end{equation*}
forms a basis for the $G$-eqiuvariant topology \eqref{equivtopology}. A short computation yields  for every $H$-invariant linear slice $W_x$, that $\mc{T}_{X|W_x}\cong\OO_{X|W_x}^{\oplus n}$ implying that the holomorphic tangent bundle and consequently every holomorphic subbundle thereof trivialize over $H$-invariant linear slices. This basic observation will play a role in Section \ref{localisationstratumofcodim1orhigher}.\\
%
Note that for each $x\in X_H^i$, every linear slice $W_x$ constitutes a holomorphic slice chart for the orbit type submanifold $X_H^i$. That is, 
if $x^1, x^2, \dots, x^{n-l}$ are the holomorphic coordinates on the complex vector subspace $\big(\bb C^n\big)^H=\bb C^{n-l}$ and $y^1, \dots, y^l$ are the holomorphic coordinates on the $l$-dimensional complement of $\bb C^{n-l}$ in $\bb C^n$, then $(x^1, \dots, x^{n-l}, y^1, \dots, y^l)$ define local holomorphic coordiantes of $X$ on $W_x$ such that $(x^1, \dots, x^{n-l})$ are local holomorphic coordinates of $X_H^i$ on $W_x\cap X_H^i$ and $(y^1, \dots y^l)$ are local holomorphic coordinates on $W_x$ in transversal direction to the stratum. 
\section{Harish-Chandra torsors on the orbit type strata of $X$}
\label{hctorsorsonstrata}
 Let $H$ be a subgroup of $G$ which has at least one fixed point in $X$. Because of the isotropy representation all the fibers of the holomorphic tangent bundle $T^{(1, 0)}X$ restricted to the stratum $X_H^i$ of codimenion $l\geq0$ are closed with respect to the $H$-action and thus, the trivial $H$-action on $X_H^i$ can be lifted to the total space of $T^{(1, 0)}X|_{X_H^i}$ by $(x, v)\mapsto(x, \tau_x(h)v)$ for all $x\in X_H^i$ and all $v\in T_x^{(1, 0)}X$. This way, the tangent space over the stratum $X_H^i$ becomes an $H$-equivariant holomorphic vector bundle which leads to a reduction of the structure group $\GL(n, \bb C)$ of $T^{(1, 0)}X|_{X_H^i}$.
The holomorphic normal bundle $\pi_{H}^{i}:N\rightarrow X_{H}^{i}$ of rank $l$ over the stratum $X_H^i$, $i\in\bb Z_{\geq0}$, is defined by the short exact sequence 
\begin{equation}
\label{normalbundle}
0\rightarrow T^{(1, 0)}(X_H^i)\rightarrow T^{(1, 0)}X|_{X_H^i}\rightarrow N\rightarrow0 
\end{equation}
in the category of holomorphic vector bundles over $X_H^i$. 
In general, there is no way per se how to identify the holomorphic normal bundle of the submanifold $X_H^i$ with some holomorphic subbundle of $T^{(1, 0)}X|_{X_H^i}$. However, since the holomorphic tangent bundle $T^{(1, 0)}X|_{X_H^i}$ is equivariant with respect to the group action of $H$, there is a natural way how to do that. Indeed, define a map $F: T^{(1, 0)}X|_{X_H^i}\rightarrow T^{(1, 0)}X|_{X_H^i}$ by $(x, v)\mapsto(x, Pv)$ with $P:=\frac{1}{|H|}\sum_{h\in H}\tau_x(h)$, $x\in X_H^i$, and $v\in T_x^{(1, 0)}X|_{X_H^i}$. By construction, $F$ is a bundle map over the identity map of $X_H^i$. Furthermore, note that, since $\tau$ varies holomorphically with the base point, $F$ is holomorphic, too. Moreover, since $X_H^i$ is a connected manifold, the dimension of the fixed point space $T_{x}^{(1, 0)}X_H^i$ for each $x\in X_H^i$ is the same, wherefore the rank of the complex linear map $F_x$ is constant. Hence, the holomorphic bundle map $F$ has a constant rank. Consequently, the kernel of $F$, $\ker(F)=\coprod_{x\in X_H^i}N_x$, is a holomorphic subbundle of $TX|_{X_H^i}$. Form the direct sum of holomorphic vector bundles $TX_H^i\oplus\ker(F)$ over $X_H^i$. Next, define a map $T^{(1, 0)}X|_{X_H^i}\rightarrow T^{(1, 0)}X_H^i\oplus \ker(F)$ by $(x,v)\mapsto\big(x,(Pv, (P-\id)v)\big)$. Evidently, since $H$ is finite, $P^2=P$, whence $(P-\id)v\in\ker(F_x)$ for every $v\in T_x^{(1, 0)}X$. Ergo, the map is well-defined and we check that it is in fact a bundle mapping. It has a holomorphic inverse $T^{(1, 0)}X_H^i\oplus\ker(F)\rightarrow T^{(1, 0)}X$ given by $\big(x, (v, w)\big)\mapsto v+w$, $v\in T_x^{(1, 0)}X_H^i$, $w\in \ker(F_x)$. So, 
\begin{equation}
\label{holbundliso}
T^{(1, 0)}X|_{X_H^i}\cong T^{(1, 0)}X_H^i\oplus\ker(F)
\end{equation}
in the category of holomorphic vector bundles. Finally, we deduce from \eqref{normalbundle} and Isomorphism \eqref{holbundliso} that
\begin{align*}
N
\cong\coprod_{x\in X_H^i}N_x
\end{align*} 
as holomorphic bundles. From now on we shall not make any distinction between the holomorphic normal bundle defined by the short exact sequence \eqref{normalbundle} and the holomorphic subbundle $\coprod_{x\in X_H^i}N_x$ of $T^{(1, 0)}X|_{X_H^i}$. We shall abuse the language and shall treat the holomorphic normal bundle $N$ as a subbundle of the holomorphic tangent bundle of $X$ restricted to $X_H^i$ and shall consider the projection $\pi_{H}^{i}$ of $N$ onto $X_{H}^{i}$ as the restriction of $\pi: T^{(1, 0)}X\big|_{X_{H}^{i}}\rightarrow X_{H}^{i}$ to $N$.\\
The fibers of the holomorphic normal bundle are simultaneously exposed to the action of the group $H$ via the isotropy representation and the matrix Lie group $\GL(l, \bb C)$, which operates by coordinate transformations on the fibers induced by transition maps, where $l$ is the rank of the normal bundle. Both group actions are compatible with each other in the following sense. Let $(U\cap X_H^i, \psi_{U})$ and $(V\cap X_H^i, \psi_{V})$ be local trivializations of the normal bundle over $X_H^i$ such that $(U\cap V)\cap X_H^i=\varnothing$. Let $\psi_{U,x}\circ\psi_{V, x}^{-1}\in\GL(l, \bb C)$ be the respective transition function and let $\psi_{V, x}\circ\tau_x\circ\psi_{V, x}^{-1}$ be the isotropy representation on $T_{x}^{(1, 0)}X_H^i$ expressed in the coordinates on $V$. Then, we have
\begin{equation}
\psi_{U,x}\circ\psi_{V, x}^{-1}(\psi_{V, x}\circ\tau_x(h)\circ\psi_{V, x})=(\psi_{U, x}\circ\tau_x(h)\circ\psi_{U, x}^{-1})\circ\psi_{U, x}\circ\psi_{V, x}^{-1}
\end{equation}
Considering that $\tau_x\sim\sigma$ for every $x\in X_H^i$, the above computation shows that the transition map commutes with $H$ expressed in the proper coordinates. That means that the structure group of $N$ gets reduced from $\GL(l, \bb C)$ to the centralizer $Z$ of $H$ in $\GL(l, \bb C)$. 
Since all the fibers of the holomorphic normal bundle of $X_H^i$ are endowed with an $H$-action, equivalent to the same linear representation $\sigma: H\rightarrow\GL(l, \bb C)$, one can assign to the fiber of the holomlorphic normal bundle at every point $x\in X_H^i$ a degree-wise completed rational Cherednik algebra $\ratCherhat$. 
The structure group of the holomorphic normal bundle $N\rightarrow X_H^i$ acts naturally on $\ratCherhat$ by $\bb C$-algebra automorphisms.  
\begin{lemma}
\label{centralizeraction}
Let $\sigma: H\rightarrow\GL(l, \bb C)$ be a faithful linear representation. Then the centralizer $Z$ of $H\equiv\sigma(H)$ in $\GL(l, \bb C)$ acts by algebra automorphisms on the degree-wise completed rational Cherednik algebra $\ratCherhat$.   
\end{lemma}  
\begin{proof}
The group $\GL(l, \bb{C})$ acts naturally by automorphisms on the tensor algebra $T^{\bullet}(\mathbb{C}^l\oplus\mathbb{C}^{l*})$ by 
\begin{equation*}
g\cdot(v_1\oplus w_1\otimes\dots\otimes v_n\oplus w_n)=g\cdot v_1\oplus g\cdot w_1\otimes\dots\otimes g\cdot v_n\oplus g\cdot w_n,\end{equation*} 
for every matrix $g\in\GL(l, \mathbb{C})$ and every element $v_1\oplus w_1\otimes\dots\otimes v_n\oplus w_n\in T^n(\mathbb{C}^l\oplus\mathbb{C}^{l*})$. Here the dot signifies the group action of $\GL(l, \mathbb{C})$ on $\mathbb{C}^l$, respectively its dual action on $\mathbb{C}^{l*}$. This gives rise to a $\GL(l, \mathbb{C})$-linear action $\theta: \GL(l, \mathbb{C})\rightarrow\GL(T^{\bullet}(\mathbb{C}^l\oplus\mathbb{C}^{l*})\rtimes\mathbb{C}H$ on the skew-group algebra $T^{\bullet}(\mathbb{C}^l\oplus\mathbb{C}^{l*})\rtimes\mathbb{C}H$  defined by $\theta(g)(f*h)=(g\cdot f)* h$, where $f\in T^{\bullet}(\mathbb{C}^l\oplus\mathbb{C}^{l*})$, $h\in H$, the asterix $\*$ denotes the twisted multiplication in the skew-group algebra and the dot signifies again the group action. The centralizer $Z$ is a matrix Lie subgroup of $\GL(l, \mathbb{C})$, thus it acts on $T^{\bullet}(\mathbb{C}^l\oplus\mathbb{C}^{l*})\rtimes\mathbb{C}H$ by linear transformations via $\theta$. It remains to show that for every $g\in Z$ the linear homomorphism $\theta(g)$ respects the ring structure of $T^{\bullet}(\mathbb{C}^l\oplus\mathbb{C}^{l*})\rtimes\mathbb{C}H$, as well. The straightforward computation
\begin{align*}
\theta(g)((f_1*h_1)*(f_2*h_2))&=
\theta(g)(f_1\otimes\sigma(h_1)(f_2)*h_1h_2)\\
&=g\cdot(f_1)\otimes g\cdot(\sigma(h_1)(f_2))*h_1h_2\\
&=g\cdot f_1\otimes\sigma(h_1)(g\cdot f_2)*h_1h_2\\
&=\theta(g)(f_1*h_1)*\theta(g)(f_2*h_2)
\end{align*}
corroborates that $\theta$ acts on $T^{\bullet}(\bb C^l\oplus\bb C^{l\ast})\rtimes\bb CH$ by automorphisms.\\
We are left to check whether the action $\theta$ preserves the ideal
 \[I= \left\langle [w, w'],~ [v, v'],~  [v, w]-(v, w)+\sum_{s\in \mc{S}}c(s)(v, \alpha_s)(w, \alpha_s^{\vee})s: v, v'\in \mathbb{C}^{l}, w, w'\in \mathbb{C}^{l*}\right\rangle,\]
as well. We let the centralizer $Z$ act on each of the generators in $I$:
\begin{align}
\label{gen1}
&\theta(g)([w, w'])= [g\cdot w, g\cdot w' ]=0 \mod I\\
\label{gen2}
&\theta(g)([v, v'])= [g\cdot v, g\cdot v']=0 \mod I.
\end{align}
Before we proceed with the third generator, we note that since by definition 
\begin{align}
(1-s)\alpha_s&=\mu\alpha_s\label{mu}\\
(1-s)\alpha_s^{\vee}&=\mu'\alpha_s^{\vee},\label{mu'}
\end{align}
$\mu$, $\mu'\in\mathbb{C}^{\times}$, an action of an element $g\in Z$ on \eqref{mu} and \eqref{mu'} results in $(1-s)\cdot g\cdot\alpha_s=\mu g\cdot\alpha_s$ and 
$(1-s)\cdot g\cdot\alpha_s^{\vee}=\mu' g\cdot\alpha_s^{\vee}$, respectively. This shows that $g\cdot\alpha_s$ and $g\cdot\alpha_s^{\vee}$ are eigenvectors to $(1-s)|_{\bb C^{l\*}}$ and $(1-s)|_{\bb C^{l}}$, accordingly. On the other side, since $\Ima(1-s|_{\bb C^{l\*}})$ and $\Ima(1-s|_{\bb C^{l}})$ are one-dimensional complex vector spaces , spanned by $\alpha_s$ and $\alpha_s^{\vee}$, respectively, there are some $\xi_g, \nu_g\in\mathbb{C}^{\times}$ satisfying $g\cdot\alpha_s=\xi_g\alpha_s$ and $g\cdot\alpha_s^{\vee}=\nu_g\alpha_s^{\vee}$. By the normalization condition $(\alpha_s, \alpha_s^{\vee})=(g\cdot\alpha_s, g\cdot\alpha_s^{\vee})=2$ we have $\xi_g\nu_g=1$. Then, applying $\theta(g)$ for $g\in Z$ on the last generator of the $I$ yields 
\begin{align}
\label{gen3}
&\theta(g)\big([v, w]-(v, w)+\sum_{s\in\mc{S}}c(s)(v, \alpha_s)(w, \alpha_s^{\vee})s\big)\nonumber\\
&=[g\cdot v, g\cdot w]-(v, w)+ \sum_{s\in\mc{S}}c(s)(v, \alpha_s)(w, \alpha_s^{\vee})s\nonumber\\
&=(v, w)-\sum_{s'\in S}c(s')(g\cdot v, \alpha_{s'})(g\cdot w, \alpha_{s'}^{\vee})s'-(v, w)+\sum_{s\in\mc{S}}c(s)(v, \alpha_s)(w, \alpha_s^{\vee})s \mod I\nonumber\\
&=-\sum_{s'\in S}c(s')(v, g^{-1}\cdot\alpha_{s'})(w, g^{-1}\cdot\alpha_{s'}^{\vee})s'+ \sum_{s\in\mc{S}}c(s)(v, \alpha_s)(w, \alpha_s^{\vee})s \mod I\nonumber\\
&=-\sum_{s'\in S}c(s')(v, \xi_{g^{-1}}\alpha_{s'})(w, \nu_{g^{-1}}\alpha_{s'}^{\vee})s'+ \sum_{s\in\mc{S}}c(s)(v, \alpha_s)(w, \alpha_s^{\vee})s \mod I\nonumber\\
&=-\sum_{s'\in S}c(s')\underbrace{\xi_{g^{-1}}\nu_{g^{-1}}}_{=1}(v, \alpha_{s'})(w, \alpha_{s'}^{\vee})s'+ \sum_{s\in\mc{S}}c(s)(v, \alpha_s)(w, \alpha_s^{\vee})s \mod I\nonumber\\
&=\mod I
\end{align}
It is conspicuous from \eqref{gen1}, \eqref{gen2} and \eqref{gen3} that $\theta(g)I\subseteq I$ which implies that $Z$ acts by algebra automorphisms on $H_{1, c}(\bb C^l, H)$. As $Z$ acts on $\bb C[[\bb C^l]]$ by algebra automorphisms via the dual representation, $Z$ acts consequently by $\bb C$-algebra automorphisms on the algebra $\bb{H}$ generated by $\bb C[[\bb C^l]]$ and $H_{1, c}(\bb C^l, H)$. Finally, the natural $\bb C$-algebra isomorphisms between $\ratCherhat$ and $\bb H$ induces the desired $H$-action on the former $\bb C$-algebra. 
\end{proof}
In the course of this section we will need the following two easy technical lemmas. For expository purposes we omit their proofs.
\begin{lemma}\label{(co)root} Let $h\in H$, let $s\in H$ be a complex reflection and let $\alpha_s$ and $\alpha_s^{\vee}$ be the corresponding root and coroot. Then, 
\begin{enumerate}
 \item[i)] there are some $\gamma, \gamma'\in\mathbb{C}$ such that $h\cdot\alpha_s=\gamma\alpha_{hsh^{-1}}$ and $h\cdot\alpha_s^{\vee}=\gamma'\check{\alpha}_{hsh^{-1}}$, respectively, and 
\item[ii)] the nontrivial eigenvalues $\lambda_s$ and $\lambda_s^{\vee}$ of the root $\alpha_s$ and the coroot $\alpha_s^{\vee}$ satisfy $ \lambda_{hsh^{-1}}=\lambda_s$ and $\lambda_{hsh^{-1}}^{\vee}=\lambda_s^{\vee}$, accordingly.
\end{enumerate}
\end{lemma}
\begin{lemma}
\label{4subtasks}
Let $A, B\in Z$ and let $s\in H$ be a complex reflection. Then there is a complex number $\lambda_{A, s}\in\bb C^{\times}$ such that
\begin{enumerate}
\item[i)] $A\cdot\alpha_s=\lambda_{A, s}\alpha_s$ and $A\cdot\alpha_s^{\vee}=-\lambda_{A, s}\alpha_s^{\vee}$, 
\item[ii)] $\lambda_{A, hsh^{-1}}=\lambda_{A, s}$ for every $h\in H$,
\item[iii)] $\lambda_{[A, B], s}=0$,
\item[iv)] The element $\sum_{s\in\mc{S}}\frac{2c(s)}{1-\lambda_s}\lambda_{A, s}(1-s)$ is a central element in the algebra $\bb C H$.
 \end{enumerate}
\end{lemma}
Let $x$ be a point lying on the stratum $X_H^i$ and let $(x^1, \dots, x^{n-l})$ be the local coordinates at $x$ along the stratum and $(y^1, \dots, y^l)$ be the local coordinates in transversal direction. The degree-wise completed Weyl algebra associated to the tangent space $T_y^{(1, 0)}X_H^i$ at every point $y$ on $X_H^i$ in the vicinity of $x$ can be expressed as $\widehat{D}_{n-l}=\big<\bb C[[x^1, \dots, x^{n-l}]], \pd{x^1}, \dots, \pd{x^{n-l}}\big>$. The Lie group $\GL(n-l, \bb C)$ possesses a natural left action on $\widehat{D}_{n-l}$ by automorphisms via 
\begin{equation}
\label{glaction}
\eta: D\mapsto A^{-1\*}\circ D\circ A^{\*}
\end{equation}
for $A\in\GL(n-l, \bb C)$, $D\in\widehat{D}_{n-l}$. On generators the action $\eta$  has the explicit form  
\begin{equation*}
\eta:A\mapsto\eta(A):
\begin{cases}
f\mapsto f\circ A^{-1}, ~~f\in\bb C[[x^1, \dots, x^{n-l}]]\\
\pd{x^i}\mapsto\sum_{j, i}(A^T)_{ij}\pd{x^j},~ i=1, \dots, n-l
\end{cases}
\end{equation*}  
The differential $\eta_{\*}$ of \eqref{glaction} yields an action  
\begin{equation*}
\eta_{\*}: a\mapsto[-\sum_{i, j=1}^{n-l}(a^T)_{ij}x^i\pd{x^j}, \cdot~]
\end{equation*}
by inner derivations of the Lie algebra $\mf{gl}(n-l, \bb C)$ on $\widehat{D}_{n-l}$.\\ 
Similarly, the degree-wise completed Cherednik algebra attached to the fiber of the holomorphic normal bundle at every point of the stratum near $x$ is isomorphic to $\ratCherhat$.\\
Take the algebraic tensor product of $\widehat{D}_{n-l}$ and $\ratCherhat$ over $\bb C$. The increasing filtrations $F'$ and $F$ by order of degrees of generators of $\widehat{D}_{n-l}$ and $\ratCherhat$, respectively, induce an increasing filtration on $\widehat{D}_{n-l}\otimes\ratCherhat$ by 
\[F_p\big(\widehat{D}_{n-l}\otimes\ratCherhat\big)=F'^{p}(\widehat{D}_{n-l})\otimes\ratCherhat+\widehat{D}_{n-l}\otimes F_p(\ratCherhat)\]
for any $p\in\bb Z_{\geq0}$. The completed tensor product $\widehat{D}_{n-l}\hat{\otimes}\ratCherhat$ is accordingly given by the completion \[\invlim_{p}\widehat{D}_{n-l}\otimes\ratCherhat/\big(F'^{p}(\widehat{D}_{n-l})\otimes\ratCherhat+\widehat{D}_{n-l}\otimes F_p(\ratCherhat)\big)\] of the algebraic tensor product with respect to the aforementioned induced filtration. For more clarity of the exposition we shall henceforth use the notation $\HC$ for the completed tensor product $\widehat{D}_{n-l}\hat{\otimes}\ratCherhat$. Denote by $\widehat{D}(\bb C^l_{\textrm{reg}})$ the completed $\bb C$-algebra $\bb C[[\bb C^l]]\otimes_{\bb C[\bb C^l]}\mc{D}(\bb C_{\textrm{reg}}^l)$. In the ensuing propositions we first demonstrate that the $\bb C$-algebras $\HC$ and $\widehat{D}_{n-l}\hat{\otimes}\widehat{\mc{D}}(\bb C^l_{\textrm{reg}})\rtimes\bb CH$ are Harish-Chandra modules and second that the map $\id\otimes\widehat{\Theta}_{c}: \HC\hookrightarrow\widehat{D}_{n-l}\hat{\otimes}\widehat{\mc{D}}(\bb C^l_{\textrm{reg}})\rtimes\bb CH$, 
where $\widehat{\Theta}_{c}=\id\otimes\Theta_c$ is the completion of the Dunkl embedding \eqref{dunkl}, is a morphism of Harish-Chandra modules.  
\begin{proposition}
$\HC$ is a Harish-Chandra $(W_{n-l}\ltimes\mf{z}\otimes\hat{\OO}_{n-l}, \GL(n-l, \bb C)\times Z)$-module. 
\end{proposition} 
\begin{proof}
We proceed as follows: First, we construct a Lie algebra homomorphism $W_{n-l}\ltimes\mf{z}\otimes\hat{\OO}_{n-l}\rightarrow\End(\HC)$. Subsequently we show that the differential of the $\GL(n-l, \bb C)\times Z$-action on $\HC$ coincides with the restriction of the above defined Lie algebra homomorphism to the Lie subalgebra $\mf{gl}(n-l, \bb C)\oplus\mf{z}\subset W_{n-l}\ltimes\mf{z}\otimes\hat{\OO}_{n-l}$.
Let $(u_i)$ and $(y_i)$, $i=1, \dots, l$ be accordingly the basis of $\bb C^l$ and the dual basis of $\bb C^{l\*}$. We claim that the mapping $\varphi_{c}: \mf{z}\rightarrow\ratCherhat$ given by $A\mapsto-\sum_{i, j}A_{ij}y_ju_i+\sum_{s\in\mc{S}}\frac{2c(s)}{1-\lambda_s}\lambda_{A, s}(1-s)$ is a Lie algebra homomorphism, where the Lie algebra bracket on $\HC$ is given by the commutator. Indeed, 
for $A, B\in\mf{z}$, we have:
\begin{align*} 
&[\varphi_{c}(A), \varphi_{c}(B)]=\\
&=[-\sum_{i,j}A_{ij}y_ju_i+\sum_{s\in\mc{S}}\frac{2c(s)}{1-\lambda_s}\lambda_{A, s}(1-s), -\sum_{k,l}B_{kl}y_lu_k+\sum_{t\in\mc{S}}\frac{2c(t)}{1-\lambda_t}\lambda_{B, t}(1-t)]\\
&=\underbrace{\sum_{i, j, k, l}A_{ij}B_{kl}[y_ju_i, y_lu_k]}_{(1)}+\underbrace{\sum_{i, j}A_{ij}\frac{2c(t)}{1-\lambda_t}\lambda_{B, t}[y_ju_i, t]}_{(2)}+\underbrace{\sum_{k, l}B_{kl}\frac{2c(s)}{1-\lambda_s}\lambda_{A, s}[s, y_lu_k]}_{(3)}.
\end{align*}
In the following computation of $(1)$, $(2)$ and $(3)$ we will constantly make use of Relations \eqref{relation1} and \eqref{relation2}, as well as of Lemma \eqref{4subtasks} without explicitly saying it for the sake of brevity. Also, to keep the exposition clear, we designate the actions of Lie groups and the corresponding Lie algebras suggestively by a dot.  
\begin{align*}
(1):~~\sum_{i, j, k, l}&A_{ij}B_{kl}[y_ju_i, y_lu_k]=
\sum_{i, j, k, l}\left(A_{ij}B_{kl}-A_{kl}B_{ij}\right)y_j[u_i, y_l]u_k\\
&=\sum_{i, j, k, l}\left(A_{ij}B_{kl}-A_{kl}B_{ij}\right)\left(\delta_{il}y_ju_k-\sum_{s\in\mc{S}}c(s)(u_i, \alpha_s)(\alpha_s^{\vee}, y_l)y_jsu_k\right)\\
&=-\sum_{k, j}[A, B]_{k, j}y_ju_k-\sum_{i, j, k, l; s\in\mc{S}}c(s)\left(A_{ij}B_{kl}-A_{kl}B_{ij}\right)(\alpha_s)_i(\alpha_s^{\vee})_ly_jsu_k\\
&=-\sum_{k, j}[A, B]_{k, j}y_ju_k+\sum_{s\in\mc{S}}c(s)(A\cdot\alpha_s)s(B\cdot\alpha_s^{\vee})-\sum_{s\in\mc{S}}c(s)(B\cdot\alpha_s)s(A\cdot\alpha_s^{\vee})\\
&=-\sum_{k, j}[A, B]_{k, j}y_ju_k+\sum_{s\in\mc{S}}c(s)\lambda_{A, s}\lambda_{B, s}\alpha_ss\alpha_s^{\vee}-\sum_{s\in\mc{S}}c(s)\lambda_{B, s}\lambda_{A, s}\alpha_ss\alpha_s^{\vee}\\
&=-\sum_{k, j}[A, B]_{k, j}y_ju_k.
\end{align*}
\begin{align*}
(2):~~\sum_{i, j; t\in\mc{S}}&A_{ij}\frac{2c(t)}{1-\lambda_t}\lambda_{B, t}[y_ju_i, t]=\sum_{i, j; t\in\mc{S}}A_{ij}\frac{2c(t)}{1-\lambda_t}\lambda_{B, t}\left(y_j[u_i, t]+[y_j, t]u_i\right)\\
&=\sum_{i, j; t\in\mc{S}}A_{ij}\frac{2c(t)}{1-\lambda_t}\lambda_{B, t}\left(y_j(u_it-tu_i)+(y_jt-ty_j)u_i\right)\\
&=\sum_{i, j; t\in\mc{S}}A_{ij}\frac{2c(t)}{1-\lambda_t}\lambda_{B, t}(y_j(u_i-t\cdot u_i)t+(y_j-t\cdot y_j)tu_i) \\
&=\sum_{i, j; t\in\mc{S}}A_{ij}\frac{2c(t)}{1-\lambda_t}\lambda_{B, t}\left(y_j\left(-\frac{1-\lambda_t}{2\lambda_t}\right)(u_i, \alpha_t)\alpha_t^{\vee}t+\left(\frac{1-\lambda_t}{2}\right)(\alpha_t^{\vee}, y_j)\alpha_ttu_i\right)\\
&=\sum_{t\in\mc{S}}\frac{2c(t)}{1-\lambda_t}\lambda_{B, t}\left(\frac{1-\lambda_t}{2\lambda_t}(A\cdot\alpha_t)\alpha_t^{\vee}t+\frac{1-\lambda_t}{2}\alpha_tt(A\cdot\alpha_t^{\vee})\right)\\
&=\sum_{t\in\mc{S}}\frac{2c(t)}{1-\lambda_t}\lambda_{B, t}\lambda_{A, t}\left(\frac{1-\lambda_t}{2\lambda_t}\alpha_t\alpha_t^{\vee}t-\frac{1-\lambda_t}{2}\alpha_tt\alpha_t^{\vee}\right).
\end{align*}
\begin{align*}
(3):~~\sum_{k, l; s\in\mc{S}}&B_{kl}\frac{2c(s)}{1-\lambda_s}\lambda_{A, s}[s, y_lu_k]=\sum_{k, l; s\in\mc{S}}B_{kl}\frac{2c(s)}{1-\lambda_s}\lambda_{A, s}\left(y_l[s, u_k]+[s, y_l]u_k\right)\\
&=\sum_{k, l; s\in\mc{S}}B_{kl}\frac{2c(s)}{1-\lambda_s}\lambda_{A, s}\left(y_l(su_k-u_ks)+(sy_l-y_ls)u_k\right)\\
&=\sum_{k, l; s\in\mc{S}}B_{kl}\frac{2c(s)}{1-\lambda_s}\lambda_{A, s}\left(-y_l(u_k-s\cdot u_k)s-(y_l-s\cdot y_l)su_k\right)\\
&=-\sum_{k, l; s\in\mc{S}}B_{kl}\frac{2c(s)}{1-\lambda_s}\lambda_{A, s}\left(-\frac{1-\lambda_s}{2\lambda_s}(u_k, \alpha_s)y_l\alpha_s^{\vee}s+\frac{1-\lambda_s}{2}(\alpha_s^{\vee}, y_l)\alpha_ssu_k\right)\\
&=-\sum_{s\in\mc{S}}\frac{2c(s)}{1-\lambda_s}\lambda_{A, s}\left(\frac{1-\lambda_s}{2\lambda_s}(B\cdot\alpha_s)\alpha_s^{\vee}s+\frac{1-\lambda_s}{2}\alpha_ss(B\cdot\alpha_s^{\vee}) \right)\\
&=-\sum_{s\in\mc{S}}\frac{2c(s)}{1-\lambda_s}\lambda_{A, s}\lambda_{B, s}\left(\frac{1-\lambda_s}{2\lambda_s}\alpha_s\alpha_s^{\vee}s-\frac{1-\lambda_s}{2}\alpha_ss\alpha_s^{\vee} \right).
\end{align*}
Adding $(1)$, $(2)$ and $(3)$ together yields
\begin{align}
\label{[varphi(A),varphi(B)]}
[\varphi_{c}(A), \varphi_{c}(B)]&
=-\sum_{kj}[A, B]_{kj}y_ju_k.
\end{align}
On the other hand, we have by Lemma \ref{4subtasks} that $\lambda_{[A, B], s}=0$, ergo 
\begin{equation}
\label{varphi([A,B])}
\varphi_{c}([A, B])=-\sum_{ij}[A, B]_{ij}y_ju_i.
\end{equation}
The equality of \eqref{[varphi(A),varphi(B)]} with \eqref{varphi([A,B])} establishes $\varphi_{c}$ as Lie algebra homomorphism.
In turn it induces an injective Lie algebra homomorphism 
\begin{align}
\label{Phi_cembedding}
&\Phi_{c}: W_{n-l}\ltimes\mf{z}\otimes\hat{\OO}_{n-l}\rightarrow\HC\nonumber\\
&v+A\otimes p\mapsto v\otimes1+p\otimes\varphi_{c}(A).
\end{align}
Indeed, a straightforward computation shows that 
\begin{align*}
&[\Phi_{c}(v+A\otimes p), \Phi_{c}(w+b\otimes q)]=\\
&=[v\otimes1, w\otimes1]+[v\otimes1, q\otimes\varphi_{c}(B)]+[p\otimes\varphi_{c}(A), w\otimes1]+[p\otimes\varphi_{c}(A), q\otimes\varphi_{c}(B)]\\
&=[v, w]\otimes1+v(q)\otimes\varphi_{c}(B)-w(p)\otimes\varphi_{c}(A)+pq\otimes\varphi_{c}([A, B])\\
&=\Phi_{c}([v, w]+[A, B]\otimes pq+B\otimes v(q)-A\otimes w(p))\\
&=\Phi_{c}([v+A\otimes p, w+B\otimes q]),
\end{align*}
where in the last line we used the definition of the Lie bracket in $W_{n-l}$.
Composing the adjoint action $\adlie$ of $\HC$ 
with $\Phi_{c}$ yields the desired Lie algebra representation 
\begin{align}
\label{Phi_representation}
\Phi: W_{n-l}\ltimes\mf{z}\otimes\hat{\OO}_{n-l}&\longrightarrow\End(\HC)\nonumber\\
v+A\otimes p&\mapsto\adlie(\Phi_{c}(v+A\otimes p)) 
\end{align}
The Lie group $\GL(n-l, \bb C)\times Z$ acts from the left on $\HC$ via the tensor action $\eta\otimes\theta$ induced by the left action $\eta$ of $\GL(n-l, \bb C)$ on $\widehat{D}_{n-l}$ and the left acion $\theta$ of $Z$ on $\ratCherhat$, discussed in Proposition \ref{centralizeraction}.
The infinite-dimensional Lie algebra $W_{n-l}\ltimes\mf{z}\otimes\hat{\OO}_{n-l}$ contains a finite-dimensional subalgebra which is isomorphic to the Lie algebra $\mf{gl}(n-l, \bb C)\oplus\mf{z}$. This identification is achieved via the Lie algebra embedding $i:\mf{gl}(n-l, \bb C)\oplus\mf{z}\hookrightarrow W_{n-l}\ltimes\mf{z}\otimes\hat{\OO}_{n-l}$ given by
\begin{align}
\label{lieembedding}
 &(A, B)\mapsto-\sum_{i, j}(A^T)_{ij}x^i\pd{x^j}+B\otimes1.
 \end{align} 
Let $(B, A)\in\mf{gl}(n-l, \bb C)\oplus\mf{z}$. Then, making use of the embedding $i$, we compute
\begin{align*}
\Phi(i(A, B))(d\otimes\zeta)&=[\Phi_{c}(i(B, A)), d\otimes\zeta]\\
&=[\Phi_{c}(-\sum_{i, j}(B^T)_{ij}x^i\pd{x^j}+A\otimes 1), d\otimes\zeta]=\\
&=[-\sum_{i, j}(B^T)_{ij}x^i\pd{x^j}\otimes 1+ 1\otimes\varphi_{c}(A), d\otimes\zeta]\\
&=[-\sum_{i, j}(B^T)_{ij}x^i\pd{x^j}, d]\otimes\zeta+d\otimes[\varphi_{c}(A), \zeta]\\
&=\eta_{*}(B)(d)\otimes\zeta+d\otimes[\varphi_{c}(A), \zeta].
\end{align*}
Next, in order to see how the endomorphism $[\varphi_{c}(A), \cdot]$ acts on arbitrary elements of the degree-wise completed Cherednik algebra, we apply it on the generators of the Cherednik algebra $\mathbb{C}^l$, $\mathbb{C}^{l*}$ and $\rho(H)$. Let $y\in\bb C^{l\*}$ be an arbitrary linear form. Then
\begin{align*}
&[\varphi_{c}(A), y]=-\sum_{i, j}A_{ij}[y_ju_i, y]-\sum_{s\in\mc{S}}\frac{2c(s)}{1-\lambda_s}\lambda_{A, s}[s, y]\\
&=-\sum_{i, j}A_{ij}y_j[u_i, y]+\sum_{s\in\mc{S}}\frac{2c(s)}{1-\lambda_s}\lambda_{A, s}[y, s]\\
&=-\sum_{i, j}A_{ij}(u_i, y)y_j+\sum_{i, j, s}c(s) A_{ij}(u_i, \alpha_s)(\alpha_s^{\vee}, y)y_js+\sum_{s\in\mc{S}}\frac{2c(s)}{1-\lambda_s}\lambda_{A, s}(y-s\cdot y)s\\
&=-yA+\sum_{s\in\mc{S}}c(s) (\alpha_s^{\vee}, y)(\alpha_s A)s+\sum_{s\in\mc{S}}c(s)\lambda_{A, s}(\alpha_s^{\vee}, y)\alpha_ss\\
&=-yA-\sum_{s\in\mc{S}}c(s) (\alpha_s^{\vee}, y)(A\cdot\alpha_s)s+\sum_{s\in\mc{S}}c(s) (\alpha_s^{\vee}, y)(A\cdot\alpha_s)s\\
&=-yA\\
&=\theta_*(A)(y).
\end{align*}
In the third term on the fourth line of the calculation we applied Relation \ref{relation2} and then in the third term on the second to the last line we used Lemma \ref{4subtasks}, $i)$. Let $u\in\bb C^l$ an arbitrary vector. Then
\begin{align*}
&[\varphi_{c}(A), u]=-\sum_{i, j}A_{ij}[y_ju_i, u]-\sum_{s\in\mc{S}}\frac{2c(s)}{1-\lambda_s}\lambda_{A, s}[s, u]\\
&=-\sum_{i, j}A_{ij}[y_j, u]u_i+\sum_{s\in\mc{S}}\frac{2c(s)}{1-\lambda_s}\lambda_{A, s}[u, s]\\
&=\sum_{i, j}A_{ij}(u, y_j)u_i-\sum_{i, j, s\in\mc{S}}c(s) A_{ij}(u, \alpha_s)(\alpha_s^{\vee}, y_j)su_i+\sum_{s\in\mc{S}}\frac{2c(s)}{1-\lambda_s}\lambda_{A, s}(u-s\cdot u)s\\
&=Au-\sum_{s\in\mc{S}}c(s)(u, \alpha_s)s(A\alpha_s^{\vee})-\sum_{s\in\mc{S}}c(s)\frac{\lambda_{A, s}}{\lambda_s}(u, \alpha_s)s(s^{-1}\alpha_s^{\vee}s)\\
&=Au-\sum_{s\in\mc{S}}c(s)(u, \alpha_s)s(A\alpha_s^{\vee})-\sum_{s\in\mc{S}}c(s)\frac{\lambda_{A, s}}{\lambda_s}(u, \alpha_s)s\lambda_s\alpha_s^{\vee}\\
&=Au-\sum_{s\in\mc{S}}c(s)(u, \alpha_s)s(A\alpha_s^{\vee})-\sum_{s\in\mc{S}}c(s)\lambda_{A, s}(u, \alpha_s)s\alpha_s^{\vee}\\
&=Au-\sum_{s\in\mc{S}}c(s)(u, \alpha_s)s(A\cdot\alpha_s^{\vee})+\sum_{s\in\mc{S}}c(s)(u, \alpha_s)s(A\cdot\alpha_s^{\vee})\\
&=Au\\
&=\theta_*(A)(u)
\end{align*}
As was the case with the previous computation, in the third term on the fourth line we applied Relation \ref{relation1}, then in the next line we used the fact that $s^{-1}\alpha_s^{\vee}=\lambda_s$ and finally, in the third to the last line we took again advantage of Lemma \ref{4subtasks}, $i)$. Let $h\in\bb CH$. Then
\begin{align*}
&[\varphi_{c}(A), h]=-\sum_{i, j}A_{ij}[y_ju_i, h]-\sum_{s\in\mc{S}}\frac{2c(s)}{1-\lambda_s}\lambda_{A, s}[s, h]\\
&=-\sum_{i, j}A_{ij}(y_ju_ih-hy_ju_i)\\
&=-\sum_{i, j}A_{ij}(y_ju_i-(y_jh^{-1})(hu_i))h\\
&=-\sum_{i, j}A_{ij}(y_ju_i-\sum_{k,l}((h^{-1})_{jk}(h)_{li}w_kv_l)h\\
&=-\sum_{i, j}A_{ij}y_ju_ih+\sum_{i, j}\sum_{k, l}h_{li}A_{ij}h^{-1}_{jk}w_kv_lh\\
&=-\sum_{i, j}A_{ij}y_ju_ih+\sum_{k, l}(hAh^{-1})_{lk}w_kv_lh\\
&=-\sum_{i, j}A_{ij}y_ju_ih+\sum_{k, k}A_{lk}w_kv_lh\\
&=0\\
&=\theta_{*}(A)(h).
\end{align*}
For the calculation of the last bracket we first made use of Lemma \ref{4subtasks}, $iv)$ in the second line and then in the second to the last line we used the fact that since $A\in Z$, it commutes with the matrices $h$ and $h^{-1}$.\\
Since $[\varphi_{c}(A),\cdot]=\theta_*(A)$ on the generators of $\ratCherhat$ for every  $A\in\mf{z}$, it follows that both maps coincide on an arbitrary element $\zeta$ of the Cherednik algebra. A direct ramification of this fact is
\begin{align*}
 \Phi(i(A, B))&=[\Phi_{c}(i(B, A)), d\otimes\zeta]\\
 &=\eta_{\*}(B)(d)\otimes\zeta+d\otimes\theta_{\*}(A)(\zeta)\\
 &=\eta\otimes\theta_{\*}(B, A)(d\otimes\zeta),
 \end{align*}
i.e. the Lie algebra representation $(\eta\otimes\theta)_{\*}$ of $\mf{gl}(n-l, \bb C)\oplus\mf{z}$ on $\HC$ factors through $\Phi$ via the embedding $i$ defined by \eqref{lieembedding}. This concludes the proof.
\end{proof}
Recall that the bundle of formal frames attached to the normal bundle $\mc{N}$ of rank $l$ over $X_H^i$ is designated by $\coor{\mc{N}}$. It is a known fact from Section \ref{ecoor} in Appendix  \ref{Chapter2} that $\coor{\mc{N}}\rightarrow\aff{\mc{N}}$ is a transitive Harish-Chandra $(W_{n-l}\rtimes\mf{z}\otimes\hat{\OO}_{n-l}, \GL(n-l, \bb C)\times Z)$-torsor. Thus, in light of the theory exhibited in Section \ref{ecoor} $\coor{\mc{N}}\times\HC$ is a flat $\GL(n-l, \bb C)\times Z$-equivariant vector bundle over $\coor{\mc{N}}$ equipped with a flat connection $d+\adlie(\phi\circ\omega)$, where $\omega$ is the holomorphic connection $1$-form of $\coor{\mc{N}}$ with values in the Lie algebra $W_{n-l}\rtimes\mf{z}\otimes\hat{\OO}_{n-l}$. 
\begin{proposition}
$\widehat{D}_{n-l}\hat{\otimes}\widehat{\mc{D}}(\bb C^l_{\textrm{reg}})\rtimes\bb CH$ is a Harish-Chandra $(W_{n-l}\ltimes\mf{z}\otimes\hat{\OO}_{n-l}, \GL(n-l, \bb C)\times Z)$-module. 
\end{proposition}
\begin{proof}
We proceed here verbatim as in the proof of the preceding proposition. Start by defining a mapping $\sigma: W_{n-l}\ltimes\mf{z}\otimes\hat{\OO}_{n-l}\hookrightarrow\widehat{D}_{n-l}\hat{\otimes}\widehat{\mc{D}}(\bb C^l_{\textrm{reg}})\rtimes\bb CH$ by
\begin{equation}
\label{sigmaembedding}
v+A\otimes p\mapsto v\otimes\id-p\otimes\sum_{i, j=1}^{n-l}(A^T)_{ij}y^i\pd{y^j}.
\end{equation}
The so-defined map $\sigma$ is a Lie algebra embedding. Indeed, for any pair of elements $v+A\otimes p$ and  $w+B\otimes q$ from $W_{n-l}\ltimes\mf{z}\otimes\hat{\OO}_{n-l}$, we have
\begin{align*}
[\sigma(v+A\otimes p), &\sigma(w+B\otimes q)]=[v\otimes1-p\otimes\hspace{-0.32em}\sum_{i, j=1}^{n-l}(A^T)_{ij}y^i\pd{y^j}, w\otimes1-q\otimes\hspace{-0.57em}\sum_{k,m =1}^{n-l}\hspace{-0.22em}(B^T)_{km}y^k\pd{y^m}]\\
&=[v, w]\otimes\id-v(q)\otimes\sum_{k, m=1}^{n-l}(B^T)_{km}y^k\pd{y^m}+w(p)\otimes\sum_{i, j=1}^{n-l}(A^T)_{ij}y^i\pd{y^j}+\\
&\quad\quad\quad\quad\quad\quad\quad\quad\quad\quad\quad\quad\quad\quad\quad\quad+pq\otimes\sum(A^T)_{ij}(B^T)_{km}[y^i\pd{y^j}, y^k\pd{y^m}]\\
&=[v, w]\otimes\id+v(q)\otimes\big(-\hspace{-0.25em}\sum_{k, m=1}^{n-l}(B^T)_{km}y^k\pd{y^m}\big)+w(p)\otimes\sum_{i, j=1}^{n-l}(A^T)_{ij}y^i\pd{y^j}+\\
&\quad\quad\quad\quad\quad\quad\quad\quad\quad\quad\quad\quad\quad\quad\quad\quad\quad+pq\otimes\big(-\sum_{i, j=1}^{n-l}([A, B]^T)_{ij}y^i\pd{y^j}\big)\\
&=\sigma([v, w]+[A, B]\otimes pq+B\otimes v(q)-A\otimes w(p))\\
&=\sigma([v+A\otimes p, w+B\otimes q]).
\end{align*}  
Moreover, injectivity follows directly from the definition of $\sigma$. 
Successive composition of $\sigma$ with the adjoint action $\adlie$ of $\widehat{D}_{n-l}\hat{\otimes}\widehat{\mc{D}}(\bb C^l_{\textrm{reg}})\rtimes\bb CH$ on itself delivers the Lie algebra map 
\begin{align}
\label{Psi_representation}
\Psi: W_{n-l}\ltimes\mf{z}\otimes\hat{\OO}_{n-l}&\rightarrow\Der(\widehat{D}_{n-l}\hat{\otimes}\widehat{D}(\bb C^l_{\textrm{reg}})\rtimes\bb CH)\nonumber\\
v+A\otimes p&\mapsto\adlie(\sigma(a+A\otimes p)). 
\end{align}
Define a left action $\tau$ of the Lie group $\GL(n-l, \bb C)\times Z$ on $\widehat{D}_{n-l}\hat{\otimes}\widehat{\mc{D}}(\bb C^l_{\textrm{reg}})\rtimes\bb CH$ by 
\begin{align}
\tau(A, B)(d\otimes\xi):=A^{-1\*}\circ d\circ A^{\*}\otimes\xi+d\otimes B^{-1\*}\circ\xi\circ B^{\*}
\end{align}
for $(A, B)\in\GL(n-l, \bb C)\times Z$ and $d\otimes\xi\in\widehat{D}_{n-l}\hat{\otimes}\widehat{\mc{D}}(\bb C^l_{\textrm{reg}})\rtimes\bb CH$. 
Finally, the differential $\tau_{\*}$ of the representation map factors through $\sigma$ via the embedding $i:\mf{gl}(n-l)\oplus\mf{z}\hookrightarrow W_{n-l}\ltimes\mf{z}\otimes\hat{\OO}_{n-l}$ given by \eqref{lieembedding} which concludes the proof. 
\end{proof} 
\begin{proposition}
\label{hcmorphism}
The $\bb C$-algebra morphism $\id\otimes\widehat{\Theta}_{c}: \HC\hookrightarrow\widehat{D}_{n-l}\hat{\otimes}\widehat{\mc{D}}(\bb C^l_{\textrm{reg}})\rtimes\bb CH$, where $\Theta_c$ denotes the Dunkl embedding, is a morphism of $(W_{n-l}\ltimes\mf{z}\otimes\hat{\OO}_{n-l}, \GL(n-l, \bb C)\times Z)$-modules.
\end{proposition}
\begin{proof}
The $\GL(n-l)\times Z$-equivariance  of $\id\otimes\Theta_c$ is conspicuous. As for the $W_{n-l}\ltimes\mf{z}\otimes\hat{\OO}_{n-l}$-equivariance of $\id\otimes\hat{\Theta}_c$, take an arbitrary element $v+A\otimes p$ in $ W_{n-l}\ltimes\mf{z}\otimes\hat{\OO}_{n-l}$ and compute 
\begin{align}
\label{factorizationofsigma}
&\id\otimes\Theta_c\circ\Phi_c(v+A\otimes p)=\id\otimes\hat{\Theta}_c(v\otimes\id+p\otimes\big(-\sum_{i, j}A_{ij}y_ju_i+\sum_{s\in\mc{S}}\frac{2c(s)}{1-\lambda_s}\lambda_{A, s}(1-s)\big)\Big)\nonumber\\
&=v\otimes\id+p\otimes\Big(\hspace{-0.3em}-\sum_{i, j}A_{ij}y_j\big(\partial_{u_i}-\sum_{s\in\mc{S}}\frac{2c(s)(u_i, \alpha_s)}{(1-\lambda_s)\alpha_s}(1-s)\big)+\sum_{s\in\mc{S}}\frac{2c(s)}{1-\lambda_s}\lambda_{A, s}(1-s)\Big)\nonumber\\
&=v\otimes\id+p\otimes\Big(\hspace{-0.3em}-\sum_{i, j}A_{ij}y_j\pd{y_i}-\sum_{s\in\mc{S}}\frac{2c(s)(A\cdot\alpha_s)}{(1-\lambda_s)\alpha_s}(1-s)+\sum_{s\in\mc{S}}\frac{2c(s)}{1-\lambda_s}\lambda_{A, s}(1-s)\Big)\nonumber\\
&=v\otimes\id+p\otimes\Big(\hspace{-0.3em}-\sum_{i, j}A_{ij}y_j\pd{y_i}-\sum_{s\in\mc{S}}\frac{2c(s)\lambda_{A, s}}{(1-\lambda_s)}(1-s)+\sum_{s\in\mc{S}}\frac{2c(s)}{1-\lambda_s}\lambda_{A, s}(1-s)\Big)\nonumber\\
&=v\otimes\id-p\otimes\sum_{i, j}A_{ij}y_j\pd{y_i}\nonumber\\
&=\sigma(v+A\otimes p)
\end{align}
where $\Phi_{c}$ is the Lie algebra map, defined in  \eqref{Phi_cembedding}. 
This implies that the injective Lie algebra homomorphism $\sigma: W_{n-l}\ltimes\mf{z}\otimes\hat{\OO}_{n-l}\hookrightarrow\widehat{D}_{n-l}\hat{\otimes}\widehat{\mc{D}}(\bb C^l_{\textrm{reg}})\rtimes\bb CH$, given by \eqref{sigmaembedding}, factors through the Lie algebra map $\Phi_{c}$,
by means of the morphism $\id\otimes\hat{\Theta}_c$. We utilise the factorization \eqref{factorizationofsigma} of $\sigma$ in the evaluation of the following mapping  
\begin{align*}
\id\otimes\hat{\Theta}_c\circ\Phi(v+A\otimes p)(d\otimes\zeta)&=\id\otimes\hat{\Theta}_c\big([\Phi_c(v+A\otimes p), d\otimes\zeta]\big)\\
&=[\id\otimes\hat{\Theta}_c\circ\Phi_c(v+A\otimes p), \id\otimes\hat{\Theta}_c(d\otimes\zeta)]\\
&=[\sigma(v+A\otimes p),  \id\otimes\hat{\Theta}_c(d\otimes\zeta)]\\
&=\Psi(v+A\otimes A)\circ\id\otimes\hat{\Theta}_c(d\otimes\zeta))
\end{align*}
where the $\Phi$ and $\Psi$ are the Lie algebra homomorphism defined in \eqref{Phi_representation} and \eqref{Psi_representation}, respectively, which affirms the $W_{n-l}\ltimes\mf{z}\otimes\hat{\OO}_{n-l}$-equivariance of $\id\otimes\hat{\Theta}_c$. With that shown the claim of the proposition is proved. 
\end{proof}
\section{Localization associated to the torsor of local formal coordinates at the strata in $X$} 
Let $H$ be a subgroup of $G$ of type $(H)$ such that $\Stab(x)=H$ for some $x\in X$ and let $X_H^i$ be a connected component of $X_{(H)}$ of codimension $l$. This section is devoted to the study of the localization of the Harish-Chandra module $\HC$ associated to the Harish-Chandra $(W_{n-l}\rtimes\mf{z}\otimes\hat{\OO}_{n-l}, \GL(n-l, \bb C)\times Z)$-torsor $\coor{\mc{N}}$ over $\aff{\mc{N}}$ for all strata in $X$. In the discussion we basically distinguish between two main cases: the principal stratum corresponding to the trivial subgroup, denoted $\mathring{X}$, and any other stratum $X_H^i$ of codimension $1$ or higher corresponding to an arbitrary nontrivial proper subgroup $H$ of $G$. In the ensuing two subsections we consequently handle 
each of both cases. 
\subsection{Localisation on the principal stratum $\mathring{X}$}
Throughout this subsection let $U$ denote an open set in $X$, which wholly lies within the principal stratum $\mathring{X}$. Consequently, it is open in the subspace topology of $\mathring{X}$, too. The normal bundle of $\mathring{X}$ has rank $l=0$. Therefore, $\HC$ reduces to $\hat{D}_n$ in that case, which is a $(W_n, \GL(n, \bb C))$-module, and $\coor{\mc{N}|_{U}}=\coor{U}$. The flat connection on the bundle $\coor{U}\times\hat{D}_n\rightarrow\coor{U}$ has the form $d+\adlie\circ j\circ\omega=d+[\omega,\cdot]$ where $j$ denotes the injection of $W_n$ in $\hat{D}_n$.  
\begin{lemma}
\label{diffoperatorflatsec}
Given a holomorphic differential operator $D\in\mc{D}(U)$, for every $[\phi]\in\coor{U}$, the assignment 
\begin{equation}
\label{flatsec}
s: [\phi]\mapsto T_{\mathbf{x}=0}\big(\phi^{\*}\circ D\circ\phi^{-1\*}\big),
\end{equation}
where $T_{\mathbf{x}=0}$ denotes the Taylor series operator at $\mathbf{x}=0$, defines a flat section of the bundle $\coor{U}\times\hat{D}_n$.  
\end{lemma}
\begin{proof}
Recall that the algebra $\mc{D}(U)$, generated by holomorphic functions $\OO(U)$ and holomorphic vector fields $\mc{V}(U)$, has an increasing filtration $\OO(O)=\mc{D}_0(U)\subset\mc{D}_1(U)\subset\mc{D}_2(U)\subset\dots$, where $\mc{D}_1(U)=\OO(U)+\mc{V}(U)$ and $\mc{D}_n(U)=\mc{D}_1(U)\cdot\dots\cdot\mc{D}_1(U)$ ($n$-times), for any $n\in\bb Z_{\geq1}$. By definition the Taylor series operator $T_{\mathbf{x}=0}$ acts on a differential operator of arbitrary order through Taylor expansion of its coefficients. 
It is well-known that the Taylor series operator respects multiplication of functions. We now demonstrate that it respects composition of differential operators of arbitrary order, as well. To that aim, we first show by induction that 
\begin{equation}
\label{induction}
T_{\pmb{x}=0}(D_{1}\cdot\dots\cdot D_{n})=T_{\pmb{x}=0}(D_{1})\cdot\dots\cdot T_{\pmb{x}=0}(D_{n})
\end{equation}
for every $n\in\bb N_{0}$ and $D_{1}, \dots, D_{n}\in\mc{D}_1(U)$. Indeed, take first and $m$-th order differential operators $D_1:=\sum_{i=1}^mf_i\pd{x^i}$ with $f_i\in\OO(U)$ for all $i=1,\dots, n$, and $D_2:=\sum_{|\alpha|\leq m}g_{\alpha}\partial^{\alpha}$ with $g_{\alpha}\in\OO(U)$ for every $\alpha\in\bb N^n$ with $|\alpha|\leq m$. By $\bb C$-linearity of the Taylor series operator, Leibniz's rule and Schwartz's Theorem, we have
$T_{\mathbf{x}=0}(D_1D_2)=T_{\mathbf{x}=0}\big(\sum_{i=1}^n\sum_{\alpha}(f_i\frac{\partial g_{\alpha}}{\partial x^i}+f_ig_{\alpha}\pd{x^i})\partial^{\alpha}\big)\nonumber
=T_{\mathbf{x}=0}(D_1)T_{\mathbf{x}=0}(D_2)$.
From this we infer under the assumption $T_{\mathbf{x}=0}(D_2\cdots D_{n+1})=T_{\mathbf{x}=0}(D_2)\cdots T_{\mathbf{x}=0}(D_{n+1})$ for $D_2, \dots, D_{n+1}\in\mc{D}_1(U)$ that $T_{\mathbf{x}=0}(D_1\cdots D_{n+1})=T_{\mathbf{x}=0}(D_1)T_{\mathbf{x}=0}(D_2)\cdots T_{\mathbf{x}=0}(D_{n+1})$ for $D_1, D_2, \dots, D_{n+1}\in\mc{D}_1(U)$
which corroborates the validity of \eqref{induction}. Now, for differential operators $D, D'\in\mc{D}(U)$ of order $m\geq1$ and $n\geq1$, respectively, the filtration of $\mc{D}(U)$ 
stipulates the existence of first order differential operators $D_1, \dots, D_m$ and $D_1', \dots, D_n'$ satisfying the equalities $D=D_1\cdots D_m$ and $D'=D_1'\cdots D_n'$, respectively. Then, with the help of \eqref{induction} we check that \begin{align}
\label{taylormult}
T_{\mathbf{x}=0}(DD')
&=T_{\mathbf{x}=0}(D_1)\cdots T_{\mathbf{x}=0}(D_m)T_{\mathbf{x}=0}(D_1')\cdots T_{\mathbf{x}=0}(D_n')\nonumber\\
&=T_{\mathbf{x}=0}(D)T_{\mathbf{x}=0}(D').
\end{align}
This demonstrates that $T_{\mathbf{x}=0}$ is compatible with composition of arbitrary differential operators. With this in mind and the $\bb C$-linearity of the Taylor series operator, the mapping $\mc{D}(U)\rightarrow\hat{D}_n$ given by $D\mapsto T_{\mathbf{x}=0}(D)$, for any differential operator $D\in\mc{D}(U)$, is actually a $\bb C$-algebra homomorphism.  
For any $X\in T_{[\phi]}\coor{U}$, let $[\phi_t]$ be a path in $\coor{U}$ such that $[\frac{d}{dt}|_{t=0}\phi_t]=X$ and $[\phi_{t=0}]=[\phi]$. Introduce $\rho_t:=\phi^{-1}\circ\phi_t$, where the inverse of $\phi$ is taken on some small region of $U\subset\mathring{X}$, on which the inverse function theorem holds true. Then,  
 \begin{align*}
 (ds_{[\phi]}, X)&=\frac{d}{dt}\Big|_{t=0}s([\phi\circ\rho_t])\nonumber\\
 &=\frac{d}{dt}\Big|_{t=0}T_{\mathbf{x}=0}\big(\rho_t^{\*}\circ(\phi^{\*}\circ D\circ\phi^{-1\*})\circ\rho_t^{-1\*}\big)\nonumber\\
 &=T_{\mathbf{x}=0}\big( \frac{d}{dt}\Big|_{t=0}\rho_t^{\*}(\phi^{\*}\circ D\circ\phi^{-1\*})\circ\id-\id\circ(\phi^{\*}\circ D\circ\phi^{-1\*})\frac{d}{dt}\Big|_{t=0}\rho_t^{\*}\big)\nonumber\\
&=[T_{\mathbf{x}=0}(\frac{d}{dt}\Big|_{t=0}\rho_t^{\*}), T_{\mathbf{x}=0}(\phi^{\*}\circ D\circ\phi^{-1\*})]\nonumber\\
&=[-\omega(X), s([\phi])]
 \end{align*}
where in the second to the last line we implicitly used \eqref{taylormult} and in the last line we recalled the definition of the $W_{n}$-valued connection $(1, 0)$-form $X\mapsto T_{\mathbf{x}=0}\big(-(d\phi_x)^{-1}(\frac{d}{dt}|_{t=0}\phi_t)\big)$. The claim follows.
 \end{proof}
The succeeding proposition yields an important constraint on the flat sections of $\coor{U}\times\widehat{D}_n$.
  \begin{proposition} 
 \label{determiningcoeff}
A flat section $s: [\phi]\mapsto \sum_{|\A|\leq m}\big(\sum_{\B\in\bb Z_{\geq0}^n}f_{\A\B}([\phi])x^{\B}\big)\partial^{\A}$ of the trivial holomorphic pro-finite vector bundle $\coor{U}\times\widehat{D}_{n}$ with respect to the flat connection $\nabla:=d+[\omega, \cdot]$ is an $\Aut_n$-equivariant map. It is uniquely determined by the  family of holomorphic maps $\{f_{\A\pmb{0}}\}_{|\A|\leq m}$ on $\coor{U}$. 
 \end{proposition}
  \begin{proof}
We start by showing that flat sections of $\coor{\mathring{X}}\times\widehat{D}_n$  are $\Aut_n$-equivariant. Let $\rho_t: \bb C^n\rightarrow\bb C^n$ be a family of germs of biholomorphisms of $\bb C^n$ fixing the origin. Then $[\rho_t]^k$ represents a smooth curve in $\Aut_{n,k}$ such that $\rho_{t=0}=\id$ and $[\dot\rho_{t=0}]^k\in W_{n, k}^0=\oplus_i\hat{\mc{O}}_{n,k}\pd{x^i}$. 
For any point $[\phi]^k$ in $U^k$ the composition $[\phi\circ\rho_t]^k$ determines a path in $U^k$ with a starting point $[\phi]^k$ such that $X^k=[\frac{d}{dt}\Big|_{t=0}\phi\circ\rho_t]^k$ is a tangent vector of $U^k$ at $[\phi]^k$. Consequently, for any such tangent vector we have 
$\frac{d}{dt}\Big|_{t=0}s([\phi\circ\rho_t]^k)
=ds_{[\phi]^k}(X^k)$. 
Hence $
\frac{d}{dt}\Big|_{t=0}\big(s([\phi\circ\rho_t]^k)-[\rho_t^{\*}]^k\circ s([\phi]^k)\circ[\rho_t^{\*-1}]^k\big)
=(ds_{[\phi]^k}, X)+[\omega_{[\phi]^k}(X^k), s([\phi]^k)]=0$ which implies that $s([\phi\circ\rho_t]^k)-[\rho_t^{\*}]^k\circ s([\phi]^k)\circ[\rho_t^{\*-1}]^k=c\in W_{n, k}^0$ for  $t=0$ whence 
\begin{equation}
\label{liegroupequivariant}
s([\phi\circ\rho_t]^k)-[\rho_t^{\*}]^k\circ s([\phi]^k)\circ[\rho_t^{\*-1}]^k=0
\end{equation}
for all $t\in\bb R_{\geq0}$. Since by default $[\rho_t]^k$ is a smooth curve in $\Aut_{n, k}$ and Equality \eqref{liegroupequivariant} is valid for any smooth path in $t$, the equality infers the $\Aut_{n, k}$-equivariance of any flat section $s$ of the trivial vector bundle $U^k\times\widehat{D}_{n, k}$ for any positive integer $k$. Consequently, a flat section $\invlim{s}$ of $\coor{U}\times\widehat{D}_n$ is $\Aut_n$-equivariant. \\
For every $k\in\bb Z_{\geq0}$, every multi-index $\pmb{\nu}$ with $|\pmb{\nu}|\leq k$ and every $j=1, \dots, n$ denote by $z_{\pmb{\nu}}^j$ the $(\nu,j)$-th coordinate $\frac{1}{\pmb{\nu}!}D^{\pmb{\nu}}(x^j\circ\phi\circ\rho_t)(0)$ on $\coor{U}$. With that we compute
  \begin{align*}
  (ds_{[\phi]^k}, X^k)&=\frac{d}{dt}\Big|_{t=0}s([\phi\circ\rho_t]^k)\nonumber\\
  &=\sum_{|\A|\leq m}\big(\sum_{\B\in\bb Z_{\geq0}^n}\sum_{\nu}\sum_j\frac{\partial f_{\A\B}}{\partial z_{\pmb{\nu}}^j}([\phi\circ\rho_t]^k)\frac{d}{dt}\Big|_{t=0}([\phi\circ\rho_t]^k)_{\pmb{\nu}}^jx^{\B}\big)\partial^{\A}.
  \end{align*}
 If we write $\omega(X)=\sum_{j}\sum_{\M}\xi_{\pmb{\mu}}^jx^{\pmb{\mu}}\pd{x^j}$, then
\begin{align}
 \label{summand}
\omega(X)(s([\phi]^k))
&=\sum_{\substack{|\A|\leq m\\\B}}\sum_{\substack{\pmb{\mu}\\j}}\xi_{\pmb{\mu}}^jf_{\A\B-\M+\E_j}([\phi]^k)(\beta_j-\mu_j+1)x^{\B}\partial^{\A}+\sum_{\substack{|\A|\leq m\\\B}}\sum_{\substack{\pmb{\mu}\\j}}
\xi_{\M}^jf_{\A\B-\M}([\phi]^k)x^{\B}\partial^{\A+\E_j}\\
 \label{negativesummand}
s([\phi]^k)(\omega(X))
&=\sum_{|\A|\leq m}\sum_{\B\in\bb Z_{\geq0}^n}\sum_{\substack{\M\\j}}f_{\A\B-\M}([\phi]^k)\xi_{\M}^jx^{\B}\partial^{\A+\E_j}\nonumber\\
&\quad+\sum_{|\A|\leq m}\sum_{\B\in\bb Z_{\geq0}^n}\sum_{\substack{\M\\j}}\sum_{\Ga\leq\{\A, \M\}}f_{\A\B-\M+\Ga}([\phi]^k)\xi_{\M}^j\binom{\A}{\Ga}\frac{\M !}{(\M-\Ga) !}x^{\B}\partial^{\A+\E_j-\Ga}.\nonumber\\
\end{align}
Taking out \eqref{negativesummand} from \eqref{summand} yields the Lie bracket 
\begin{align*}
[\omega(X), s([\phi]^k)]
&=\sum_{\substack{|\A|\leq m\\\B}}\sum_{\substack{\pmb{\mu}\\j}}\xi_{\pmb{\mu}}^jf_{\A\B-\M+\E_j}([\phi]^k)(\beta_j-\mu_j+1)x^{\B}\partial^{\A}\nonumber\\
&\quad-\sum_{\substack{|\A|\leq m\\\B}}\sum_{\substack{\pmb{\mu}\\j}}\sum_{\Ga\leq\{\A, \M\}}f_{\A\B-\M+\Ga}([\phi]^k)\xi_{\M}^j\binom{\A}{\Ga}\frac{\M !}{(\M-\Ga) !}x^{\B}\partial^{\A+\E_j-\Ga}.
\end{align*}
Finally, the flatness of $s$ ultimately implies 
\begin{align}
\label{originalrecursiveformula}
&\sum_{\substack{|\A|\leq m\\\B}}\sum_{j}\xi_{\pmb{0}}^jf_{\A\B+\E_j}([\phi]^k)(\beta_j+1)x^{\B}\partial^{\A}\nonumber\\
&=\sum_{\substack{|\A|\leq m\\\B}}\Big(\sum_{\nu}\sum_j\frac{\partial f_{\A\B}}{\partial z_{\pmb{\nu}}^j}([\phi\circ\rho_t]^k)\frac{d}{dt}\Big|_{t=0}([\phi\circ\rho_t]^k)_{\pmb{\nu}}^j-\sum_{\substack{\pmb{\mu}>\pmb{0}\\j}}\xi_{\pmb{\mu}}^jf_{\A\B-\M+\E_j}([\phi]^k)(\beta_j-\mu_j+1)\Big)x^{\B}\partial^{\A}\nonumber\\
&\quad+\sum_{j}\sum_{\substack{\E_j\leq\A\\1\leq|\A|\leq m\\\B}}\sum_{\pmb{\mu}}\sum_{\Ga\leq\{\A, \M\}}f_{\A+\Ga-\E_j\B-\M+\Ga}([\phi]^k)\xi_{\M}^j\binom{\A+\Ga-\E_j}{\Ga}\frac{\M !}{(\M-\Ga) !}x^{\B}\partial^{\A}.
\end{align}
We consider the situations $\A=\pmb{0}$ and $\A>\pmb{0}$ separately. First, we handle the case where $\A=\pmb{0}$. Then \eqref{originalrecursiveformula} reduces to
\begin{align*}
&\sum_{\B}\sum_{j}\xi_{\pmb{0}}^jf_{\pmb{0}\B+\E_j}([\phi]^k)(\beta_j+1)x^{\B}\nonumber\\
&=\sum_{\B}\Big(\sum_{\nu}\sum_j\frac{\partial f_{\pmb{0}\B}}{\partial z_{\pmb{\nu}}^j}([\phi\circ\rho_t]^k)\frac{d}{dt}\Big|_{t=0}([\phi\circ\rho_t]^k)_{\pmb{\nu}}^j-\sum_{\substack{\pmb{\mu}>\pmb{0}\\j}}\xi_{\pmb{\mu}}^jf_{\pmb{0}\B-\M+\E_j}([\phi]^k)(\beta_j-\mu_j+1)\Big)x^{\B}.
\end{align*}
For each fixed multiindex $\B$, a comparison of coefficients of the terms on the left and on the right hand sides of the above equality yields the following equality
\begin{align*}
&\sum_{j}\xi_{\pmb{0}}^jf_{\pmb{0}\B+\E_j}([\phi]^k)(\beta_j+1)\nonumber\\
&=\Big(\sum_{\nu}\sum_j\frac{\partial f_{\pmb{0}\B}}{\partial z_{\pmb{\nu}}^j}([\phi\circ\rho_t]^k)\frac{d}{dt}\Big|_{t=0}([\phi\circ\rho_t]^k)_{\pmb{\nu}}^j-\sum_{\substack{\pmb{\mu}>\pmb{0}\\j}}\xi_{\pmb{\mu}}^jf_{\pmb{0}\B-\M+\E_j}([\phi]^k)(\beta_j-\mu_j+1)\Big).
\end{align*}
For a set of $n$ linearly independent tangent vectors $X_1, \dots, X_n$ in the tangent space of $U^k$ at $[\phi]^k$ the above equality evolves into a system of $n$ linear equations
\begin{align*}
&\sum_{j=1}^n\xi_{\pmb{0}}^{rj}f_{\pmb{0}\B+\E_j}([\phi]^k)(\beta_j+1)\nonumber\\
&=\Big(\sum_{\nu}\sum_j\frac{\partial f_{\pmb{0}\B}}{\partial z_{\pmb{\nu}}^j}([\phi\circ\rho_t]^k)\frac{d}{dt}\Big|_{t=0}([\phi\circ\rho_t]^k)_{\pmb{\nu}}^j-\sum_{\substack{\pmb{\mu}>\pmb{0}\\j}}\xi_{\pmb{\mu}}^{rj}f_{\pmb{0}\B-\M+\E_j}([\phi]^k)(\beta_j-\mu_j+1)\Big)
\end{align*}
in $n$ indeterminates $f_{\pmb{0}\B+\E_j}([\phi]^k)$ where $r=1,\dots, n$. The matrix $(\xi_{\pmb{0}}^{rj})$ is invertible and thus by applying Cramer's rule we arrive at a recursive formula for the coefficients $f_{\pmb{0}\B}$
\begin{align}
\label{recursiveformula1}
f_{\pmb{0}\B+\E_j}=\mathsmaller{\frac{\sum_{r}C_{rj}\Big(\sum_{\nu}\frac{\partial f_{\pmb{0}\B}}{\partial z_{\pmb{\nu}}^r}([\phi\circ\rho_t]^k)\frac{d}{dt}\Big|_{t=0}([\phi\circ\rho_t]^k)_{\pmb{\nu}}^r-\sum_{\pmb{\mu}>\pmb{0}}\xi_{\pmb{\mu}}^{rj}f_{\pmb{0}\B-\M+\E_r}([\phi]^k)(\beta_r-\mu_r+1)\Big)}{(\beta_j+1)\det(\xi_{\pmb{0}}^{ij})}}
\end{align}
where $C_{rj}$ are cofactors of $(\xi_{\pmb{0}}^{ij})$. Similarly, in the case, where $\A\geq\pmb{\epsilon}_k$, $k$ is fixed, Equality \eqref{originalrecursiveformula} reduces to
\begin{align*}
&\sum_{\substack{\E_k\leq\A\\|\A|\leq m\\\B}}\sum_{j}\xi_{\pmb{0}}^jf_{\A\B+\E_j}([\phi]^k)(\beta_j+1)x^{\B}\partial^{\A}\nonumber\\
&=\sum_{\substack{\E_k\leq\A\\|\A|\leq m\\\B}}\Big(\sum_{\nu}\sum_j\frac{\partial f_{\A\B}}{\partial z_{\pmb{\nu}}^j}([\phi\circ\rho_t]^k)\frac{d}{dt}\Big|_{t=0}([\phi\circ\rho_t]^k)_{\pmb{\nu}}^j-\sum_{\substack{\pmb{\mu}>\pmb{0}\\j}}\xi_{\pmb{\mu}}^jf_{\A\B-\M+\E_j}([\phi]^k)(\beta_j-\mu_j+1)\Big)x^{\B}\partial^{\A}\nonumber\\
&+\sum_{\substack{\E_k\leq\A\\1\leq|\A|\leq m\\\B}}\sum_{\pmb{\mu}}\sum_{\Ga\leq\{\A, \M\}}f_{\A+\Ga-\E_k\B-\M+\Ga}([\phi]^k)\xi_{\M}^k\binom{\A+\Ga-\E_k}{\Ga}\frac{\M !}{(\M-\Ga) !}x^{\B}\partial^{\A}.
\end{align*}
Comparison of coefficients of terms $x^{\B}\partial^{\A}$ yields 
\begin{align*}
&\sum_{j}\xi_{\pmb{0}}^jf_{\A\B+\E_j}([\phi]^k)(\beta_j+1)\nonumber\\
&=\sum_{\nu}\sum_j\frac{\partial f_{\A\B}}{\partial z_{\pmb{\nu}}^j}([\phi\circ\rho_t]^k)\frac{d}{dt}\Big|_{t=0}([\phi\circ\rho_t]^k)_{\pmb{\nu}}^j-\sum_{\substack{\pmb{\mu}>\pmb{0}\\j}}\xi_{\pmb{\mu}}^jf_{\A\B-\M+\E_j}([\phi]^k)(\beta_j-\mu_j+1)\nonumber\\
&+\sum_{\pmb{\mu}}\sum_{\Ga\leq\{\A, \M\}}f_{\A+\Ga-\E_k\B-\M+\Ga}([\phi]^k)\xi_{\M}^k\binom{\A+\Ga-\E_k}{\Ga}\frac{\M !}{(\M-\Ga) !}.
\end{align*}
Analogous argumentation as in the previous case leads to 
\begin{align}
\label{recursiveformula2}
f_{\A\B+\E_j}=\frac{\sum_rC_{rj}}{(\beta_j+1)\det(\xi_{\pmb{0}}^{ij})}\Bigg(& \sum_{\nu}\frac{\partial f_{\A\B}}{\partial z_{\pmb{\nu}}^r}([\phi\circ\rho_t]^k)\frac{d}{dt}\Big|_{t=0}([\phi\circ\rho_t]^k)_{\pmb{\nu}}^r\nonumber\\
&-\sum_{\pmb{\mu}>\pmb{0}}\xi_{\pmb{\mu}}^{rj}f_{\A\B-\M+\E_r}([\phi]^k)(\beta_r-\mu_r+1)\nonumber\\
&+\sum_{\pmb{\mu}}\sum_{\Ga\leq\{\A, \M\}}f_{\A+\Ga-\E_k\B-\M+\Ga}([\phi]^k)\xi_{\M}^k\binom{\A+\Ga-\E_k}{\Ga}\frac{\M !}{(\M-\Ga) !}\Bigg).
\end{align}
Formulae \eqref{recursiveformula1} and \eqref{recursiveformula2} show that each coefficient $f_{\A\B}$ is uniquelly determined by the whole family of coefficients $\{f_{\A\pmb{0}}\}_{|\A|\leq m}$ as desired. 
\end{proof}
Let us denote by $\mc{D}_{\mathring{X}}$ the sheaf of holomorphic differential operators on the principal stratum $\mathring{X}$. Further, let us designate by $\OO(\coor{\mathring{X}}\times\widehat{D}_n)$ the $\Cs$-submodule of flat sections of the holomorphic pro-finite-dimensional bundle $\coor{\mathring{X}}\times\widehat{D}_n$ over $\coor{\mathring{X}}$ with respect to the flat holomorphic connection $\nabla:=d+[\omega, \cdot]$. Observe that this $\Cs$-module is in particular a $\Cs$-algebra. With that notation at hand along with the preceding lemma and proposition we are equipped to state and prove the ensuing important proposition. 
\begin{proposition}
The map \eqref{flatsec} induces an isomorphism $\mc{Y}:\mc{D}_{\mathring{X}}\rightarrow\coor{\pi}_{\*}\OO(\coor{\mathring{X}}\times\widehat{D}_n)$ of $\Cs$-algebras.
\end{proposition} 
\begin{proof}
The map \eqref{flatsec} in Lemma \ref{diffoperatorflatsec} gives a well-defined map of $\Cs$-algebras $\mc{Y}: \mc{D}_{\mathring{X}}\rightarrow\coor{\pi}_{\*}\OO(\coor{\mathring{X}}\times\widehat{D}_n)$. Moreover,  this morphism of sheaves is injective. Indeed, for any two differential operators $D_1$ and $D_2$ on an arbitrary open set $U$ in $\mathring{X}$ with corresponding local coordinate representations $\sum_{\A}g_{\A}\partial^{\A}$ and $\sum_{\A}h_{\A}\partial^{\A}$ on $U$, the equality $\mc{Y}|_{U_i}(U)(D_1)=\mc{Y}|_{U_i}(U)(D_2)$ implies $D^{\B}(g_{\A}\circ\phi)(0)=D^{\B}(h_{\A}\circ\phi)(0)$ for every multi-index $\B$ and every local parametrization $\phi:\bb C^n\rightarrow U$. This entails that $g_{\A}=h_{\A}$ on $U$, consequently $D_1=D_2$ on $U$ whence $\mc{Y}(U)$ is injective. 
The principal stratum is covered by local holomorphic coordinate charts $(U_{i}, \psi_{i})_{i\in I}$ which in turn induce a local holomorphic section $\varphi_i: x\mapsto[\varphi_{x}^{i}]$ of $\coor{\mathring{X}}$ over each $U_{i}$ given by $\varphi_{x}^{i}(y)=\psi_{i}^{-1}(\psi_{i}(x)+y)$ for all $y\in U_{i}$. 
In order to complete the proof of the claim in the lemma, it suffices to verify that on every open set $U_i$ of the given cover the corresponding restricted morphism of sheaves $\mc{Y}|_{U_i}: \mc{D}_{\mathring{X}}|_{U_i}\rightarrow\coor{\pi}_{\*}\OO(\coor{\mathring{X}}\times\widehat{D}_n)|_{U_i}$ is surjective.\\ 
Fix a chart $U_i$ from the open cover $(U_i, \psi)_{i\in I}$ and let  $U\subset U_i$ be an arbitrary open set within. 
Given a flat section $s:[\phi]\mapsto\sum_{\A, \B}f_{\A\B}([\phi])x^{\B}\partial^{\A}$ of $\coor{\mathring{X}}\times\widehat{D}_{n}$ with $\partial^{\A}\in\End(\widehat{\OO}_n)$, define a holomorphic differential operator $D:=\sum_{\A}(f_{\A\pmb{0}}\circ\varphi_i)(x)~~\varphi_{x}^{i~\*-1}\circ\partial^{\A}\circ\varphi_{x}^{i~\*}$ on $U$, $x\in U$ varying, with respect to $(U_{i}, \psi_{i})$. 
By Lemma \ref{diffoperatorflatsec} the image of $D$ under $\mc{Y}|_{U_i}(U)$ defines a flat section $s'$ of $\coor{\mathring{X}}\times\widehat{D}_n$ given by
\begin{align}
\label{coordinatetransformation}
&s': [\phi]\mapsto T_{\pmb{x}=\pmb{0}}\Big(\sum_{\A}(f_{\A\pmb{0}}\circ\varphi_{i}\circ\phi)~~\phi^{\*}\circ(\varphi_{x}^{i~\*})^{-1}\circ\partial^{\A}\circ\varphi_{x}^{i~\*}\circ\phi^{-1\*}\Big)
\end{align}
with $x=\phi(0)$. Since $\coor{\mathring{X}}$ is in particular a principal $\Aut_n$-bundle, the profinite Lie group $\Aut_n$ acts  transitively from the right on its fibers. Ergo, there is an element $g\in\Aut_{n}$ such that $[\varphi_{x}^{i}]=[\phi\circ g]$. Consequently, we have for \eqref{coordinatetransformation}
\begin{equation*}
s'([\phi])=\sum_{\A}T_{\pmb{x}=\pmb{0}}\Big(f_{\A\pmb{0}}\circ\varphi_{i}\circ\phi\Big)g^{-1\*}\circ\partial^{\A}\circ g^{\*},
\end{equation*}
accordingly for the zeroth order term thereof in $\pmb{x}$, 
\begin{align*}
s'([\phi])|_{\pmb{x}=\pmb{0}}&
=\sum_{\A}f_{\A\pmb{0}}([\phi\circ g])g|_{\pmb{x}=\pmb{0}}^{-1\*}\circ\partial^{\A}\circ g|_{\pmb{x}=\pmb{0}}^{\*}\\
&=\sum_{\A}f_{\A\pmb{0}}([\phi])\partial^{\A}
\end{align*}
where we accounted that $g|_{\pmb{x}=\pmb{0}}\in\GL(n, \bb C)$ and in the last line we employed the $\Aut_{n}$-equivariance of flat sections of $\coor{\mathring{X}}\times\widehat{D}_{n}$, see Proposition \ref{determiningcoeff}. The latter computation shows that similarly to $s$ the flat section $s'$ is uniquely determined by the family of holomorphic maps $(f_{\A\pmb{0}})_{\A\in\mc{A}}$. 
Therefore, by the uniqueness of the coefficients, established in Proposition \ref{determiningcoeff}, it follows that $s'=s$ which corroborates the surjectivity of $\mc{Y}|_{U_i}(U)$. Consequently, the sheaf $\mc{Y}|_{U_i}$ is a surjective monomorphism, that is,  an isomorphism for every local holomorphic coordinate chart $U_{i}$. Finally, this entails that $\mc{Y}$ is an isomorphism of $\Cs$-algebras, as desired.           
\end{proof}
\subsection{Localisation on a stratum of codimension equal to and higher than $1$}
\label{localisationstratumofcodim1orhigher}
The orbit type strata associated to all non-trivial subgroups of $G$ are in general not closed, but rather only locally closed in $X$, whence they are not analytic in $X$.  Nevertheless, it can be shown that the intersection of an orbit type stratum with a linear slice is analytic within the slice. This fact is established in the following proposition.
\begin{proposition}
Let $H$ be a no-trivial subgroup of $G$. Let $W_x$ be a $H$-invariant linear slice at a point $x$ with $\Stab(x)=H$.
Then $X_H^i\cap W_x$ is analytic within $W_x$.
\end{proposition}      
\begin{proof}
Assume $\codim(X
_H^i)=l$. Let $\psi: W_x\rightarrow\bb C^n$ be the coordinate chart map corresponding to $W_x$. As explained in Subsection \ref{basis/charts}, $W_x$ is a holomorphic coordinate $(n-l)$-slice chart of the submanifold $X_H^i$. Take the last $l$ component mappings $\psi^{n-l+1}, \dots, \psi^{n}$ of the chart map $\psi$ on $W_x$. Then for every $y$ in $W_x$ and every neighborhood $V$ at $y$ lying in $W_x$, the intersection $V\cap X_H^i$ is identical to the collection of all points $y'$ in $V$, for which $\psi^{n-l+1}|_V(y')=\dots\psi^{l}|_V(y')=0$. This means that  $X_H^i\cap W_x$ is analytic in $W_x$.     
\end{proof}
From now on till the end of this subsection $W_x$ stands for an $H$-invariant linear slice at a point $x$ on the orbit type stratum $X_H^i$. A choice of a basis in $T_x^{(1,0)}X$ defines an $H$-equivariant embedding $\varphi_{x}: W_{x}\hookrightarrow T_{x}^{(1, 0)}X$ whose image $\varphi(W_x)$ is open in the product topology of $T_{x}^{(1, 0)}X_{H}^{i}\oplus N_{x}$. We denote by $W_{x, H}^i$ the intersection of $W_x$ with $X_H^i$, and denote by $\widehat{W}_{x}=(W_{x, H}^i, \OO_{\widehat{W}_x})$ the completion of $W_x$ along the analytic subset $W_{x, H}^i$. 
Further, denote by $\Sigma_0$ the image of the zero section on the restriction of the holomorphic normal bundle $N$ on $X_H^i$ to the open subset $W_{x, H}^{i}$ in the subspace topology. For the given date the following statement holds true. 
\begin{theorem}
\label{completediso}
There is an $H$-equivariant open embedding $\zeta: W_x\hookrightarrow N|_{W_{x, H}^i}$ with $\Ima(\zeta|_{W_{x, H}^i})\subseteq\Sigma_0$ which induces an $H$-equivariant isomorphism of~~$\bb C$-ringed spaces $\widehat{\zeta}:\widehat{W}_x\rightarrow\widehat{V}_{\epsilon}$ where $\widehat{V}_{\epsilon}:=\big(\Sigma_0\cap V_{\epsilon}, \OO_{\widehat{V}_{\epsilon}}\big)$ is the formal completion of~~$V_{\epsilon}:=\zeta(W_x)$ with respect to the analytic subset $\Sigma\cap V_{\epsilon}$.
\end{theorem}
\begin{proof}
As expounded in Subsection \ref{basis/charts}, the slice $W_x$ induces an $H$-equivariant trivialization $\psi$ of the holomorphic normal bundle of $X_H^i$ over $W_{x, H}^i$. Moreover, by the $H$-equivariance of $\varphi_x$ the image $\varphi_x(W_{x, H}^i)$ lies in $T_x^{(1, 0)}(X_H^i)$. These facts are instrumental in constructing an open embedding of $W_{x}$ into the total space of $N|_{W_{x, H}^{i}}$. Let $\Psi$ be the isomorphism between the trivial holomorphic vector bundles $N|_{W_{x, H^{i}}}$ over $W_{x, H^{i}}$ and $\varphi_{x}(W_{x, H}^{i})\times N_{x}$ over $\varphi_{x}(W_{x, H^{i}})$ arising from the ensuing composition of bundle isomorphism  
\begin{equation}
\label{composediso}
\begin{tikzcd}
N|_{W_{x, H}^i}\arrow[r, "\psi"]\arrow[d]&W_{x, H}^i\times\bb C^l\arrow[r, "\theta"]\arrow[d]&\varphi_x(W_{x, H}^i)\times N_x\arrow[d]\\
W_{x, H}^i\arrow[r, "\id"]&W_{x, H}^i\arrow[r, "\varphi_x"]&\varphi_x(W_{x, H}^i)   
\end{tikzcd}
\end{equation}
where the latter isomorphism is given by $\theta(q,v)=(\varphi_x(q), \psi_x^{-1}(v))$ for all $(q, v)\in W_{x, H}^{i}\times\bb C^{l}$. By means of $\Psi$ we can finally define a holomorphic embedding of $W_{x}$ in $N|_{W_{x, H}^i}$ as the composition 
\begin{equation*}
\begin{tikzcd}
\zeta: W_x\arrow[r, hook, "\varphi_x"]&\varphi_x(W_{x, H}^i)\times N_x\arrow[r, "\Psi^{-1}"]&N|_{W_{x, H}^i}
\end{tikzcd}
\end{equation*} 
where we utilize the above-mentioned assumption that $\varphi_{x}(W_{x})$ is open in the box topology of $T_{x}^{(1, 0)}X$. 
Moreover, as a biholomorphism $\Psi$ is an open map. Hence, the composed mapping $\zeta$ is an open injective embedding of $W_x$ into the total space $N|_{W_{x, H}^i}$ whose restriction to $W_{x, H}^i$ by definition is a section of $N|_{W_{x, H}^i}$. Clearly, the $H$-equivariance of $\varphi_x$ and $\Psi$ imply that the map $\zeta$ is $H$-equivariant, too. A restriction of the range of that morphism to the open image $\Ima(\zeta)=: V_{\epsilon}$ yields a biholomorphism $\zeta: W_x\rightarrow V_{\epsilon}$. Moreover, from the standard fact that the inverse images of analytic sets of holomorphic maps are analytic sets, see e.g. \cite[Ch.~4, \S~1.6]{GR84}, it follows that $\zeta(W_{x, H}^i)=(\zeta^{-1})^{-1}(W_{x, H}^i)=\Sigma_0\cap V_{\epsilon}$ is analytic in $V_{\epsilon}$. Let $\mc{I}\subset\OO_{V_{\epsilon}}$ be the ideal sheaf of the analytic set $\Sigma\cap V_{\epsilon}$ in $V_{\epsilon}$ and let $\mc{J}\subset\OO_{W_x}$ be the respective ideal sheaf of the analytic set $W_{x, H}^i$ in $W_x$. For every open set $U\subset V_{\epsilon}$, define a natural morphism of $\bb C$-algebras $F: \OO_{V_{\epsilon}}(U)\rightarrow\zeta_{\*}\OO_{W_x}(U)$ by $f\mapsto f\circ\zeta$. The so-defined mapping is compatible with restrictions and for any non-negative integer $\nu$ and any open set $U\subset V_{\epsilon}$, we have $F(\mc{I}(U)^{\nu+1})\subset\zeta_{\*}\mc{J}^{\nu+1}(U)$. Therefore, $F$ gives rise to a morphism of sheaves of $\bb C$-algebras
 \begin{equation}
 \label{Fsharp}
 \zeta^{(\nu)\#}: \OO_{V_{\epsilon}}/\mc{I}^{\nu+1}\rightarrow\zeta_{\*}\big(\OO_{W_x}/\mc{J}^{\nu+1}\big).
 \end{equation}
For every element $y$ in the set $W_{x, H}^i$, the above morphism of sheaves of $\bb C$-algebras induces a morphism of $\bb C$-algebras $\zeta_y^{(\nu)\#}:\OO_{V_{\epsilon}, \zeta(y)}/\mc{I}_{\zeta(y)}^{\nu+1}\rightarrow\OO_{W_x, y}/\mc{J}_y^{\nu+1}$ at the level of stalks which is given by $f \mod \mc{I}_{\zeta(y)}^{\nu+1}\mapsto f\circ\zeta \mod \mc{J}_y^{\nu+1}$. Due to the fact that $\zeta$ is invertible this map is an isomorphism of $\bb C$-algebras. This implies that the map \eqref{Fsharp} is an isomorphism of sheaves of $\bb C$-algebras. The direct image functor $\zeta_{\*}$ is left exact and commutes with limits. Thus,  after taking the projective limit of \eqref{Fsharp}, we get an isomorphism of completed sheaves of $\bb C$-algebras $\zeta^{\#}:\OO_{\widehat{V}_{\epsilon}}\rightarrow\zeta_{\*}\OO_{\widehat{W}_x}$. Thus the pair $\hat{\zeta}=(\zeta, \zeta^{\#})$ defines an $H$-equivariant isomorphism between the $\bb C$-ringed spaces $\widehat{V}_{\epsilon}$ and $\widehat{W}_x$, as stated in the proposition.                        
\end{proof}
By virtue of the fact that the orbit type strata are locally analytic we can define  completions of coherent $\OO_X$-modules locally along the orbit strata. In the following we digress from the flow of the text so far in order to generalize the notion of a formal completion of a sheaf of a global Cherednik algebra at a point of a good orbifold, introduced in \cite{eti04}, to the case of an analytic subset. Such a generalization is needed because the output of the  localization of the Harish-Chandra torsors $\HC$ on the different strata of dimension $l$ is $\Cs$-algebras which are to  be interpreted precisely as formal completions of the sheaf of global Cherednik algebras with respect to particular analytic subsets. Let us be more precise. 
\subsubsection{Formal completion of the sheaf of global Cherednik algebras along analytic sets}
Assume now that $X$ and $G$ are as specified at the beginning of this paper. Define a sheaf $\mc{R}$ of multiplicative subsets of the sheaf of rings $\OO_X$ by 
\begin{equation*}
U\mapsto\mc{R}(U):=\big\{f\in\OO_X(U)|~f|_{D\cap U}=0~\textrm{and}~f(p)=0~\textrm{for all}~p\in U\setminus D\big\}\cup\{1\}.
\end{equation*}
where $D=\cup_{s\in\mc{S}}U\cap Y^s$, $Y^s$ is the codimension $1$ connected component of the fixed point set $X^s$. The localized sheaf of rings $\OO_X[\mc{R}^{-1}]$ has a natural structure of a $\OO_X$-bimodule by the canonical morphism of sheaves of commutative rings $\OO_X\rightarrow\OO_X[\mc{R}^{-1}]$. The sheaf $\mc{D}_X\rtimes G$ is a left $\OO_X$-module in the $G$-equivariant topology of $X$. Thus the extension of scalars $\OO_X[\mc{R}^{-1}]\otimes_{\OO_X}\mc{D}_X\rtimes G$ defines a left $\OO_X[\mc{R}^{-1}]$-module. For every open $U\subseteq X$, any derivation $\theta\in\Der(\OO_X(U))$ composed with the canonical embedding $\OO_X(U)\hookrightarrow\OO_X(U)[\mc{R}(U)^{-1}]$ maps the set $\mc{R}(U)$ in the group of units of the localized ring $\OO_X(U)[\mc{R}(U)^{-1}]$. Thus by the universal property of the localization functor any derivation $\theta$ of $\OO_X(U)$ induces a unique derivation of $\OO_X(U)[\mc{R}(U)^{-1}]$. Hence, the product on $\mc{D}_X\rtimes G$, induced by the isomorphism of $\OO_X$-modules between $\mc{D}_X\rtimes G$ and the $\Cs$-algebra generated by $\mc{D}_X$ and $G$, respects the localization. This way, the extension of scalars $\OO_X[\mc{R}^{-1}]\otimes_{\OO_X}\mc{D}_X\rtimes G$ inherits the structure of a $\Cs$-algebra. We remark that sections of $\OO_X[\mc{R}^{-1}]\otimes_{\OO_X}\mc{D}_X\rtimes G$ contrary to the ones of the sheaf $j_{\*}j^{\*}(\mc{D}_X\rtimes G)$, where $j: X\setminus D\hookrightarrow X$ is the open embedding, do not admit essential singularities. Therefore, $\OO_X[\mc{R}^{-1}]\otimes_{\OO_X}\mc{D}_X\rtimes G$ can be identified with a $\Cs$-subalgebra of $j_{\*}j^{\*}(\mc{D}_X\rtimes G)$. Since the Dunkl operators are by definition sections of $j_{\*}j^{\*}(\mc{D}_X\rtimes G)$ with poles of order $1$ on $D$, they can be seen as local sections of $\OO_X[\mc{R}^{-1}]\otimes_{\OO_X}\mc{D}_X\rtimes G$. Thus, the sheaf $\mc{H}_{1, c, X, G}$ of global Cherednik algebras can be interpreted as a subsheaf of the smaller sheaf $\OO_X[\mc{R}^{-1}]\otimes_{\OO_X}\mc{D}_X\rtimes G$ of $\bb C$-algebras generated locally by $p_{\*}\OO_X$, Dunkl operators and $G$. In particular, $\mc{H}_{1, c, X, G}$ is a $\OO_X$-submodule of $\OO_X[\mc{R}^{-1}]\otimes_{\OO_X}\mc{D}_X\rtimes G$.\\
Let $Y$ be an analytic subset in $X$ and let $\widehat{X}:=(Y, \OO_{\widehat{X}})$ be the formal completion of $X$ with respect to $Y$. Then the $\Cs$-algebra $\OO_{\widehat{X}}$ is a $\OO_X$-bimodule via the natural embedding $\chi: \OO_X\hookrightarrow\OO_{\widehat{X}}$ of sheaves of $\bb C$-algebras. This way the extension of scalars $\OO_{\widehat{X}}\otimes_{\OO_X}\mc{H}_{1, c, X, G}$ becomes a left $\OO_{\widehat{X}}$-module. Since $\OO_{X}$ is a coherent module over itself, its completion $\widehat{\OO_X}=\OO_{\widehat{X}}$ is a flat $\OO_X$-bimodule. Thus the injection of left $\OO_X$-modules $\mc{H}_{1, c, X, G}\hookrightarrow\OO_X[\mc{R}^{-1}]\otimes_{\OO_X}\mc{D}_X\rtimes G$ naturally induces an embedding of left $\OO_{\widehat{X}}$-modules $\OO_{\widehat{X}}\otimes_{\OO_X}\mc{H}_{1, c, X, G}\hookrightarrow\OO_{\widehat{X}}\otimes_{\OO_X}\OO_X[\mc{R}^{-1}]\otimes_{\OO_X}\mc{D}_X\rtimes G$. The right hand side of the embedding is in particular a $\Cs$-algebra, which thanks to the injection induces a $\Cs$-algebra structure on $\OO_{\widehat{X}}\otimes_{\OO_X}\mc{H}_{1, c, X, G}$, as well. In the remainder of this paper we shall interchangeably refer to this $\Cs$-algebra as the \emph{formal completion of the sheaf of global Cherednik algebras on $X$} and the \emph{formally completed sheaf of global Cherednik algebras on $X$}.
\subsubsection{Localization}
From now on till the end of this subsection $V_{\epsilon}$ designates an open set on the total space of the restricted holomorphic bundle $N|_{W_{x, H}^i}$ with properties as in Theorem \ref{completediso}. 
The normal bundle of $X_H^i$ has rank $l>0$. For brevity we shall denote the $(W_{n-l}\ltimes\mf{z}\otimes\hat{\OO}_{n-l}, \GL(n-l)\times Z)$-torsor of formal coordinates of the holomorphic normal bundle $N|_{W_{x, H}^i}$ over $\aff{\mc{N}}$ by $\coor{\mc{N}}$. Also for clarity of the exposition we shall use the notation $\HC$ for the $\bb C$-algebra and $(W_{n-l}\ltimes\mf{z}\otimes\hat{\OO}_{n-l}, \GL(n-l)\times Z)$-module $\widehat{D}_{n-l}\hat{\otimes}\widehat{H}_{1, c}(\bb C^l, H)$. As explained in the last paragraph of Section \ref{hctorsorsonstrata}, $\coor{\mc{N}}\times\HC$ is a flat $\GL(n-l)\times Z$-equivariant holomorphic vector bundle over $\coor{\mc{N}}$ with a flat holomorphic connection $\nabla_H^i:=d+[\Phi_c\circ\omega,~\cdot~]$ 
where $\Phi_c$ is the Lie algebra embedding of $W_{n-l}\ltimes\mf{z}\otimes\hat{\OO}_{n-l}$ into $\HC$ and $\omega$ is the holomorphic $1$-connection form on $\coor{\mc{N}}$ with values in $W_{n-l}\ltimes\mf{z}\otimes\hat{\OO}_{n-l}$, defined in \eqref{connection1formonE}. 
\begin{proposition}
\label{flatsectionarb.stratum}
Given an element $\widehat{D}\in\OO_{\widehat{V}_{\epsilon}}(\Sigma_0\cap V_{\epsilon})\otimes_{\OO_{V_{\epsilon}}(V_{\epsilon})}H_{1, c}(V_{\epsilon}, H)$, the assignment
\begin{equation}
\label{nontrivialflatsection}
s(\widehat{D}): [\phi]\mapsto1\otimes\widehat{\Theta}_{c}^{-1}\Big(T_{\pmb{x}=\pmb{0}}\big(\phi^{\*}\circ\widehat{D}\circ\phi^{-1\*}\big)\Big)
\end{equation}
for every $[\phi]\in\coor{\mc{N}}$, where $T_{\pmb{x}=\pmb{0}}$ denotes the Taylor series operator at $\pmb{x}=\pmb{0}$, 
defines a flat section of the trivial holomorphic bundle $\coor{\mc{N}}\times\HC$.
\end{proposition}
\begin{proof}
First, we verify whether for any given element in $\OO_{\widehat{V}_{\epsilon}}(\Sigma_0\cap V_{\epsilon})\otimes_{\OO_{V_{\epsilon}}(V_{\epsilon})}H_{1, c}(V_{\epsilon}, H)$ the image of every point in $\coor{\mc{N}}$ under the section \eqref{nontrivialflatsection} lies in $\HC$. It suffices to explicitly check this for the generators only, that is, for formally completed functions on $\widehat{V}_{\epsilon}$, group elements of $G$ and Dunkl operators. 
To that aim, pay attention that every parametrization $\phi$ of the holomorphic normal bundle $N$ fixes a different set of local coordinates $(x^1, \dots, x^{n-l}, y^1, \dots, y^l)$ on $V_{\epsilon}\subset N$ where $(x^1, \dots, x^{n-l})$ are the local coordinates on $V_{\epsilon}\cap\Sigma_0$ and $(y_1, \dots, y^l)$ are the fiber coordinates. Take a formal holomorphic function $\hat{f}=(f_{\nu} \mod \mc{J}^{\nu+1})_{\nu\in\mc{A}}$ on $\widehat{V}_{\epsilon}$. Its coordinate representation with respect to a given parametrization $\phi$ is $\sum_{|\A|\leq\nu}\frac{1}{\A !}D^{\A}(f_{\nu}\circ\phi)(\*, 0)\otimes y^{\A}$. Thus, for every $[\phi]\in\coor{\mc{N}}$, we have
\begin{equation}
\label{formalfunctionsfilt}
s(\hat{f})([\phi])=\sum_{|\A|\leq\nu}\frac{1}{\A !\B !}D^{\B}D^{\A}(f_{\nu}\circ\phi)(0, 0)x^{\B}\otimes y^{\A},
\end{equation}
which is squarely an element in $\HC$. 
For every group element $h\in H$ and every infinite jet $[\phi]$, we trivially get an element \begin{equation}
\label{groupelementsfilt}
s(h)([\phi])=h
\end{equation}
 from $\HC$. As for the Dunkl operators introduce new coordinates $(x^1, \dots, x^{n-l}, z_{(h, Y^h)}^1, \dots, z_{(h, Y^h)}^l)$ on $V_{\epsilon}$ for every complex reflection $(h, Y^h)\in\mc{S}$ with
 \begin{equation}
 \label{newfibercoord}
 z_{(h, Y^h)}^j=\sum_{k=1}^lA(h)_{jk}y^k~~\textrm{for every}~~j=1, \dots, l,\end{equation} 
in which $A(h)\in\GL(l, \bb C)$ is chosen  such that the hypersurface $Y^h\cap V_{\epsilon}=\{u\in V_{\epsilon}| z_{(y, Y^h)}^1(u)=0\}$. 
For every holomorphic vector field $Z$ on $V_{\epsilon}$, let $
\sum_{j=1}^{n-l}Z^j\pd{x^j}+\sum_{k=1}^lZ^{n-l+k}\pd{y^k}$
and
$\sum_{j=1}^{n-l}\tilde{Z}^j\pd{x^j}+\sum_{k=1}^l\tilde{Z}^{n-l+k}\pd{z^k}$ 
be the representations of $Z$ with respect to the coordinates, determined by the parametrization $\phi$, and the new coordinates, respectively. The relation between the components of $Z$ in the two set of coordinates is 
\begin{align*}
&\tilde{Z}^j=Z^j~~~\textrm{for all}~~j=1, \dots, n-l\\
&\tilde{Z}^{n-l+k}=\sum_{i=1}^lA(h)_{ki}Z^{n-l+i}~~~\textrm{for all}~~k=1, \dots, l.
\end{align*}
Thus, with respect to the new coordinates the map $\xi_{Y^h}$ acquires the form 
\begin{equation*}
\label{xi_{Y^h}2}
\xi_{Y^h}:Z\mapsto\frac{\sum_{k=1}^lA(h)_{1k}Z^{n-l+k}}{z_{(h, Y^h)}^1} \mod \OO_{V_{\epsilon}}(V_{\epsilon}\cap\Sigma_0).
\end{equation*}
The Dunkl operator corresponding to $Z$ has in the new set of coordinates the  representation 
\begin{equation*}
\bb D_{Z}= \mc{L}_Z+\sum_{(h, Y_h)\in\mc{S}}\sum_{k=1}^l\frac{2c((h, Y_h))}{1-\lambda_{(h, Y^h)}}\frac{A(h)_{1k}Z^{n-l+k}}{z_{(h, Y_h)}^1}(h-1)
\end{equation*}
which as an element embedded in $\OO_{\widehat{V}_{\epsilon}}\otimes_{\OO_{V_{\epsilon}}}H_{1, c}(V_{\epsilon}, H)$ has the form 
\begin{align}
\label{formallycompleteddunkl}
\sum_{\pmb{\alpha}}y^{\pmb{\alpha}}\otimes\Bigg(\sum_{j=1}^{n-l}\frac{D^{\pmb{\alpha}}Z^j(\cdot, 0)}{\pmb{\alpha} !}&\pd{x^j}+\sum_{k=1}^l\frac{D^{\pmb{\alpha}}Z^{n-l+k}(\cdot, 0)}{\pmb{\alpha} !}\pd{y^k}+\nonumber\\
&+\sum_{(h, Y_h)\in\mc{S}}\sum_{k=1}^l\frac{2c((h, Y_h))}{1-\lambda_{(h, Y^h)}}\frac{A(h)_{1k}D^{\pmb{\alpha}}Z^{n-l+k}(\cdot, 0)}{\pmb{\alpha} ! z_{(h, Y_h)}^1}(h-1)\Bigg).
\end{align}   
The Taylor expansion of \eqref{formallycompleteddunkl} at $\pmb{x}=\pmb{0}$ 
can be identified with the element 
\begin{align*}
&\sum_{\pmb{\alpha}}\Bigg(\sum_{j=1}^{n-l}\sum_{\pmb{\beta}}\frac{D^{\pmb{\beta}}D^{\pmb{\alpha}}Z^j(0, 0)}{\pmb{\beta} !\pmb{\alpha} !}x^{\pmb{\beta}}\pd{x^j}\otimes y^{\pmb{\alpha}}\Bigg)\otimes1+\\
&+\sum_{k=1}^l\Bigg(\sum_{\pmb{\alpha}}\sum_{\pmb{\beta}}\frac{D^{\pmb{\beta}}D^{\pmb{\alpha}}Z^{n-l+k}(0, 0)}{\pmb{\beta} !\pmb{\alpha} !}x^{\pmb{\beta}}\otimes y^{\pmb{\alpha}}\otimes\Big(\pd{y^k}+\sum_{(h, Y_h)\in\mc{S}}\frac{2c((h, Y_h))}{1-\lambda_{(h, Y^h)}}\frac{A(h)_{1k}}{z_{(h, Y_h)}^1}(h-1)\Big)\Bigg)
\end{align*}
As $z_{(h, Y^h)}^1$ is a linear form on $\bb C^{l}$ with a codimension $1$ kernel, depending on a complex reflection $h$ in $\bb C^{l}$, we can view $z_{(h, Y_h)}^1$ as a root and set $\alpha_s:=z_{(h, Y^h)}^1$. Furthermore, Equation \eqref{newfibercoord} implies $A(h)_{1k}=(u_k, \alpha_s)$. Hence the above expression can be rewritten in the form 
\begin{align}
\label{longexpr}
\sum_{\pmb{\alpha}}\Bigg(\sum_{j=1}^{n-l}\sum_{\pmb{\beta}}\frac{D^{\pmb{\beta}}D^{\pmb{\alpha}}Z^j(0, 0)}{\pmb{\beta} !\pmb{\alpha} !}x^{\pmb{\beta}}\pd{x^j}&\otimes y^{\pmb{\alpha}}\Bigg)\otimes1+\nonumber\\
&+\sum_{k=1}^l\sum_{\pmb{\alpha}}\sum_{\pmb{\beta}}\frac{D^{\pmb{\beta}}D^{\pmb{\alpha}}Z^{n-l+k}(0, 0)}{\pmb{\beta} !\pmb{\alpha} !}x^{\pmb{\beta}}\otimes y^{\pmb{\alpha}}\otimes D_{u_k}
\end{align}
where $D_{u_k}=\pd{y^k}+\sum_{(h, Y_h)\in\mc{S}}\frac{2c((h, Y_h))}{1-\lambda_{(h, Y^h)}}\frac{(u_k, \alpha_s)}{\alpha_s}(h-1)$ is the Dunkl operator to the basis vector $u_k$. Expression \eqref{longexpr} is clearly an element in $\widehat{D}_{n-l}\hat{\otimes}\widehat{\mc{D}}(\bb C_{\textrm{reg}}^l)\rtimes\bb CG$. Therefore, a subsequent application of the map $\id\otimes\widehat{\Theta}_c^{-1}$ on \eqref{longexpr} delivers the result
\begin{align}
\label{dunklfiltr}
s(\bb D_Z)([\phi])=\sum_{\pmb{\alpha}}\Bigg(\sum_{j=1}^{n-l}\sum_{\pmb{\beta}}\frac{D^{\pmb{\beta}}D^{\pmb{\alpha}}Z^j(0, 0)}{\pmb{\beta} !\pmb{\alpha} !}&x^{\pmb{\beta}}\pd{x^j}\otimes y^{\pmb{\alpha}}\Bigg)\otimes1+\nonumber\\
&+\sum_{k=1}^l\sum_{\pmb{\alpha}}\sum_{\pmb{\beta}}\frac{D^{\pmb{\beta}}D^{\pmb{\alpha}}Z^{n-l+k}(0, 0)}{\pmb{\beta} !\pmb{\alpha} !}x^{\pmb{\beta}}\otimes y^{\pmb{\alpha}}\otimes u_k
\end{align}
which, as desired, is an element of $\HC$. It can be shown by induction in a similar fashion as in Proposition \ref{diffoperatorflatsec} that the Taylor series operator at $\pmb{x}=0$ is compatible with the composition of Dunkl operators. More generally, the Taylor operator naturally respects all operations between generators of the $\bb C$-algebra $\OO_{\hat{V}_{\epsilon}}(V_{\epsilon}\cap\Sigma_0)\otimes_{\OO_{V_{\epsilon}}(V_{\epsilon})}H_{1, c}(V_{\epsilon}, H)$. Thus, for every element $\widehat{D}$ in that $\bb C$-algebra and every parametrization of the normal bundle, the image $\id\otimes\widehat{\Theta}_c^{-1}\big(T_{\pmb{x}=\pmb{0}}(\phi^{\*}\circ D\circ\phi^{-1\*})\big)$ lies in the $\bb C$-algebra $\HC$.  
The rest of the proof of this proposition follows, almost verbatim, the proof of Proposition \ref{diffoperatorflatsec}. Take a one-parameter family of $H$-equivariant complex diffeomorphisms $\rho_t: V\times\bb C^l\rightarrow V\times\bb C^l$ with $(x, y)\mapsto(f_t(x), A_t(x)y)$, where $V$ is an open neighborhood of $0$ in $\bb C^{n-l}$, $f_t$ is a family of holomorphic automorphisms of $V$ and $A_t: V\rightarrow Z$ is a holomorphic map  such that $A_t(x)$ varies smoothly in $t$ for all $x\in V$. We have 
\begin{align}
\label{pullbackequalssigma}
T_{\pmb{x}=\pmb{0}}\Big(\frac{d}{dt}\Big|_{t=0}\rho_t^{\*}\big)=
&=\sum_{j=1}^{n-l}\frac{1}{\pmb{\alpha} !}D^{\pmb{\alpha}}(\dot f_{\mathsmaller{t=0}}^{\mathsmaller{j}})(0)x^{\pmb{\alpha}}\pd{x^j}+\sum_{k, m=n-l+1}^n\frac{1}{\pmb{\beta} !}\big(D^{\pmb{\beta}}(\dot A_{\mathsmaller{t=0}})(0)\big)_{mk}x^{\pmb{\beta}}y^k\pd{y^m}\nonumber\\
&=\sigma\Big(T_{\pmb{x}=\pmb{0}}\big(\frac{d}{dt}\Big|_{t=0}\rho_t\big)\Big) 
\end{align}
where $\sigma$ is the Lie algebra embedding of $W_{n-l}\ltimes\mf{z}\otimes\hat{\OO}_{n-l}$ into $\widehat{D}_{n-l}\hat{\otimes}\widehat{D}(\bb C^l_{\textrm{reg}})\rtimes\bb CH$ given by \eqref{sigmaembedding}. Just as in the proof of Proposition \ref{diffoperatorflatsec} it can be  shown here that $T_{\pmb{x}=\pmb{0}}$ is a multiplicative operator. We leave the lengthy verification of this straightforward fact to the interested reader. As a result, for any $[\phi]\in\coor{\mc{N}}$ and any tangent vector $X\in T^{(1, 0)}_{[\phi]}\coor{\mc{N}}$ with $X=[\frac{d}{dt}\Big|_{t=0}\phi\circ\rho_t]$, we have 
\begin{align*}
ds_{[\phi]}(X)&=\frac{d}{dt}\Big|_{t=0}1\otimes\Theta_c^{-1}\Big(T_{\pmb{x}=\pmb{0}}\big((\phi\circ\rho_t)^{\*}\circ\widehat{D}\circ(\phi\circ\rho_t)^{-1\*}\big)\Big)\\
&=1\otimes\Theta_c^{-1}\Big(T_{\pmb{x}=\pmb{0}}\big(\frac{d}{dt}\Big|_{t=0}\rho_t^{\*}(\phi^{\*}\circ\widehat{D}\circ\phi^{-1\*})\circ\id+\id\circ(\phi^{\*}\circ\widehat{D}\circ\phi^{-1\*})\frac{d}{dt}\Big|_{t=0}\rho_t^{-1\*}\big)\Big)\\
&=1\otimes\Theta_c^{-1}\Big(-[T_{\pmb{x}=\pmb{0}}(-\frac{d}{dt}\Big|_{t=0}\rho_t^{\*}), T_{\pmb{x}=\pmb{0}}(\phi^{\*}\circ\widehat{D}\circ\phi^{-1\*})]\Big)\\
&=1\otimes\Theta_c\Big(-[\sigma\Big(T_{\pmb{x}=\pmb{0}}(-\frac{d}{dt}\Big|_{t=0}\rho_t\big)\big), T_{\pmb{x}=\pmb{0}}(\phi^{\*}\circ\widehat{D}\circ\phi^{-1\*})]\Big)\\
&=1\otimes\Theta_c^{-1}\Big(-\big[\sigma(\omega_{[\phi]}(X)), T_{\pmb{x}=\pmb{0}}(\phi^{\*}\circ\widehat{D}\circ\phi^{-1\*})\big]\Big)\\
&=1\otimes\Theta_c^{-1}\Big(-\big[1\otimes\Theta_c\circ\Phi_c(\omega_{[\phi]}(X)), T_{\pmb{x}=\pmb{0}}(\phi^{\*}\circ\widehat{D}\circ\phi^{-1\*})\big]\Big)\\
&=-\big[\Phi_c(\omega_{[\phi]}(X)), s([\phi])\big]
\end{align*}
where in the sixth line we applied \eqref{pullbackequalssigma}, in the seventh line we employed Definition \eqref{connection1formonE} of the holomorphic connection one-form on $\coor{\mc{N}}$, in the fourth to the last line the factorization \eqref{factorizationofsigma} of $\sigma$ was utilized, and in the line afterwards we used Proposition \ref{hcmorphism}. The claim follows immediately.  
\end{proof}
We recall that, given a basis $(u_i)$, $i=1, \dots, l$, of $\bb C^l$ with a corresponding dual basis $(y_i)$, $i=1, \dots, l$, of $\bb C^{l\*}$, the degree-wise completed rational Cherednik algebra $\widehat{H}_{1, c}(\bb C^l, H)$ is spanned by the elements $h\cdot u^{\pmb{I}}\otimes y^{\pmb{J}}$, where $h\in H$ and $\pmb{I}, \pmb{J}, \in\bb N_0^l$ are multi-indices with $|\pmb{J}|<\infty$. Let $\mu_h$ be an index with a domain $\{1, \dots, |H|\}$. Having made these notational clarifications, we are able to formulate the analog of Proposition \ref{determiningcoeff} in the context of an orbit type stratum of codimension $l>0$.   
\begin{proposition}
\label{uniquenessoffs}
A flat section $s: [\phi]\mapsto\sum_{\pmb{\alpha, \beta}}\sum_{\mu_h, \pmb{I}, \pmb{J}}F_{\pmb{\alpha\beta}}^{\mu_h\pmb{IJ}}([\phi])x^{\pmb{\beta}}\partial^{\pmb{\alpha}}\otimes (y^{\pmb{I}}\otimes u^{\pmb{J}}h)$ of the trivial holomorphic bundle $\coor{\mc{N}}\times\HC$ with respect to the flat holomorphic connection $\nabla_H^i$ is an $\Aut_{n-l}\times Z(\OO_{n-l})$-equivariant map which is uniquely determined by the family of holomorphic maps $\{F_{\pmb{\alpha0}}^{\mu_h\pmb{IJ}}\}$ on $\coor{\mc{N}}$.   
\end{proposition}
\begin{proof}
The proof repeats verbatim the proof of Proposition \ref{determiningcoeff}.
\end{proof}
Let us denote by $\OO_{\textrm{flat}}(\coor{\mc{N}}\times\HC)$ the $\Cs$-submodule of flat sections of the holomorphic pro-finite dimensional vector bundle $\coor{\mc{N}}\times\HC$ with respect to the flat holomorphic connection $\nabla_H^i$. We recall that this $\OO_{\coor{\mc{N}}}$-module has the structure of a $\Cs$-algebra, inherited from the $\bb C$-algebra structure of the Harish-Chandra torsor $\HC$. For the purposes of the next proposition we recall that the sheaf of Cherednik algebras possesses a natural increasing and exhaustive filtration $\mc{F}^{\bullet}$, defined by assigning grades to its generators, which induces an increasing and exhaustive filtration on its formal completion. By abuse of notation we denote the inherited filtration on the formally completed sheaf of Cherednik algebras by $\mc{F}^{\bullet}$, as well. Likewise, we define an increasing filtration $\mc{G}^{\bullet}$ on the $\Cs$-algebra $\OO_{\textrm{flat}}(\coor{\mc{N}}\times\HC)$ by 
\begin{equation}
\mc{G}^p:=\mc{G}^p\OO_{\textrm{flat}}(\coor{\mc{N}}\times\HC):=\OO_{\textrm{flat}}(\coor{\mc{N}}\times F^p\HC)
\end{equation}
for each integer $p\geq0$, where $F^p\HC=\sum_{s+t=p}F^s\widehat{D}_{n-l}\hat{\otimes}F^t\ratCherhat$. It is evident that so defined, the filtration is exhaustive. Let in the following $\tau$ denote the right $\Aut_{n-l}\times Z(\widehat{\OO}_{n-l})$-action on $\coor{\mc{N}}$.
\begin{proposition}
\label{filterediso}
Let $\hat{\zeta}=(\zeta, \zeta^{\#}): \widehat{W}_x\rightarrow\widehat{V}_{\epsilon}$ be the  isomorphism of~~$\bb C$-ringed spaces in Theorem \ref{completediso}. There is a filtered isomorphism of~~$\Cs$-algebras \[\mc{Y}: \zeta^{-1}\big(\OO_{\hat{V}_{\epsilon}}\otimes_{\OO_{V_{\epsilon}}}\mc{H}_{1, c, V_{\epsilon}, H}\big)\rightarrow\coor{\pi}_{\*}\OO_{\textrm{flat}}(\coor{\mc{N}}\times\HC)\] on $W_{x, H}^i$.
\end{proposition}    
\begin{proof}
Since $\id\otimes\widehat{\Theta}_c$ together with the Taylor series operator at $\pmb{x}=\pmb{0}$ are injective $\bb C$-algebra morphisms, the mapping \eqref{nontrivialflatsection} in Proposition \ref{flatsectionarb.stratum} induces a well-defined injective morphism of $\Cs$-algebras $\mc{Y}: \zeta^{-1}\big(\OO_{\widehat{V}_{\epsilon}}\otimes_{\OO_{V_{\epsilon}}}\mc{H}_{1, c, V_{\epsilon}, H}\big)\rightarrow\coor{\pi}_{\*}\OO_{\textrm{flat}}(\coor{\mc{N}}\times\HC)$ on $W_{x, H}^i$. Expressions \eqref{formalfunctionsfilt}, \eqref{groupelementsfilt} and \eqref{dunklfiltr} in Proposition \ref{flatsectionarb.stratum} imply $\mc{Y}(\zeta^{-1}\mc{F}^0)\subseteq\coor{\pi}_{\*}\mc{G}^0$ and $\mc{Y}(\zeta^{-1}\mc{F}^1)\subseteq\coor{\pi}_{\*}\mc{G}^1$, respectively. Since all the generators of $\zeta^{-1}\big(\OO_{\widehat{V}_{\epsilon}}\otimes_{\OO_{V_{\epsilon}}}\mc{H}_{1, c, V_{\epsilon}, H}\big)$ lie within $\zeta^{-1}\mc{F}^1$, and for every integer $p\geq0$, $\mc{F}^p=\mc{F}^1\cdot\dots\cdot\mc{F}^1$ ($p$ factors), we get $\mc{Y}(\zeta^{-1}\mc{F}^p)\subseteq\coor{\pi}_{\*}\mc{G}^p$ for all integers $p\geq0$. This entails that $\mc{Y}$ is in fact a filtered monomorphism. 
By virtue of that and the exhaustiveness of the filtartions $\mc{F}^{\bullet}$ and $\mc{G}^{\bullet}$, respectively, to prove that $\mc{Y}$ is an isomorphism, it suffices to show that the map $\mc{Y}:\zeta^{-1}\mc{F}^p\rightarrow\coor{\pi}_{\*}\mc{G}^p$ is a surjection  for every integer $p\geq0$. For that we carry out a proof by induction on the index $p$ of the filtrations $\mc{F}^{\bullet}$ and $\mc{G}^{\bullet}$, respectively. To that aim, let $\phi_0$ be a fixed parametrization of the normal bundle over $W_{x, H}^i$. The set $W_{x, H}^i$ is a coordinate chart on $X_H^i$. Thus, there is a holomorphic section $\varphi: W_{x, H}^i\rightarrow\coor{\mc{N}}$ given by $p\mapsto[\varphi_p]$ where $\varphi_p(x, y)=\phi_0(\phi_0^{-1}\circ\zeta(p)+(x, y))$ for all $p\in W_{x, H}^i$. Suppose $W$ is an open subset of $W_{x, H}^i$.\\    
{\bf{Base Case}}: Let $p=0$. Take a section $s:[\phi]\mapsto\sum_{\A\B}f_{\A\B}([\phi])x^{\A}\otimes y^{\B}$ of $\mc{G}^{0}(W)$ where $(x^{i}, y^{j})$ are the canonical coordinates on $\bb C^{n-l}\times\bb C^l$. Then, for all $u\in\zeta(W)$ with $p=\pi(u)$,  
\[\hat{\xi}(u):=\Big(\sum_{|\B|\leq\nu}(f_{\pmb{0}\B}\circ\varphi\circ\pi)(u)~~\varphi_{p}^{-1\*}\circ y^{\B}\circ\varphi_p^{\*} \mod \mc{J}(\zeta(W))^{\nu+1}\Big)_{\nu},\]
where $\mc{J}$ is the sheaf ideal of $V_{\epsilon}\cap\Sigma$ inside of $V_{\epsilon}$, is a formal holomorphic function on $\Sigma_0|_W$.
The fibers of $\coor{\mc{N}}$ are homegeneous $\Aut_{n-l}\times Z(\widehat{\OO}_{n-l})$-spaces. Therefore, for every $[\phi]$ in $\coor{\mc{N}}$ satisfying $\pi(\phi(0, 0))=p\in W$, there is a unique group element $g$ in $\Aut_{n-l}\times Z(\widehat{\OO}_{n-l})$ such that $\varphi(\pi(\phi(0, 0)))=[\phi\circ g]$. By abuse of notation we denote the unique germ of  $Z$-equivariant holomorphic isomorphisms of trivial bundles $\bb C^{n-l}\times\bb C^{l}\rightarrow\bb C^{n-l}\times\bb C^{l}$ over $\id$ at $0$ in $\bb C^{n-l}$, which represents $g\in\Aut_{n-l}\times Z(\widehat{\OO}_{n-l})$, by $g$, as well.      
With respect to a particular parametrization $\phi$, the formal holomorphic function $\hat{\xi}$ has the coordinate representation 
\begin{align}
\label{suppressprojlimit}
\Big(\sum_{|\B|\leq\nu}T_{\pmb{y}=\pmb{0}}\big(f_{0\B}\circ\varphi\circ\pi\circ\phi\big)T_{\pmb{y}=0}&(y^{\B}\circ\varphi_p^{-1}\circ\phi) \mod \mf{n}^{\nu+1}\Big)_{\nu}\nonumber\\
&=\Big(\sum_{|\B|\leq\nu}f_{0\B}\circ\varphi\circ\pi\circ\phi(\*, 0)(b(x)^{-1}y)^{\B} \mod \mf{n}^{\nu+1}\Big)_{\nu}\nonumber\\
&\equiv\sum_{|\B|<\infty}f_{0\B}\circ\varphi\circ\pi\circ\phi(\*, 0)(b(x)^{-1}y)^{\B}
\end{align}
where $\mf{n}=(y^{1}, \dots, y^{l})$ and $g(x, y)=(f(x), b(x)y)$ for all $(x, y)\in \bb C^{n-l}\times\bb C^l$ with $f$ a germ of biholomorphisms of $\bb C^{n-l}$ at $0$ and $b(x)\in Z$ for every $x\in\bb C^{n-l}$. 
Thus, the image $\mc{Y}(W)(\hat{\xi})$ is the flat section 
\[s':[\phi]\mapsto\sum_{\A\B}F_{\A\B}([\phi])~g^{-1\*}\circ x^{\A}\otimes y^{\B}\circ g^{\*}\]
of $\coor{\mc{N}}\times\HC$ where $F_{\A\B}([\phi])=\frac{1}{\A !}D^{\A}\big(f_{\pmb{0}\B}\circ\varphi\circ\pi\circ\phi\big)(0, 0)$ for every $[\phi]\in\coor{\mc{N}}$ and $g=(\id, T_{\pmb{x}=\pmb{0}}(b^{-1}))$. 
For every $[\phi]\in\coor{\mc{N}}$, we have $F_{\pmb{0}\B}([\phi])=f_{\pmb{0}\B}(\varphi(\pi(\phi(0, 0))))=f_{\pmb{0}\B}([\phi\circ g])$. Hence, for the zeroth order term of $s'([\phi])$ in $\pmb{x}$ we get
\begin{align*}
s'([\phi])|_{\pmb{x}=\pmb{0}}&=\sum_{\B}f_{0\B}([\phi\circ g])~g|_{\pmb{x}=\pmb{0}}^{-1\*}|\circ (1\otimes y^{\B})\circ g|_{\pmb{x}=\pmb{0}}^{\*}\\
&= \sum_{\B}f_{0\B}([\phi])1\otimes y^{\B}
\end{align*}
where in the last line we used the $\Aut_{n-l}\times Z(\widehat{\OO}_{n-l})$-equivariance of the flat section $s'$. As by Proposition \ref{uniquenessoffs} $s'$ is uniquely determined by the holomorphic coefficients $F_{\pmb{0}\B}$, the flat sections $s$ and $s'$ are one and the same. Hence, the monomorphism $\mc{Y}(W): \mc{F}^{0}(\zeta(W))\rightarrow\mc{G}^{0}(\coor{\mc{N}})$ is surjective. Since $W$ is an arbitrary opens subset of $W_{x, H}^{i}$, this implies that $\mc{Y}: \zeta^{-1}\mc{F}^{0}\rightarrow\coor{\pi}_{\*}\mc{G}^{0}$ is an isomorphism of $\Cs$-modules on $W_{x, H}^{i}$.\\   
{\bf{ Induction Hypothesis}}: Assume that $\mc{Y}: \zeta^{-1}\mc{F}^{p}\rightarrow\coor{\pi}_{\*}\mc{G}^{p}$, $p>0$ an integer, is an isomorphism of $\Cs$-modules on $W_{x, H}^{i}$.\\
{\bf{ Induction Step}}: The injectivity of $\mc{Y}$ naturally transfers to the morphism of factor $\Cs$-modules $\gr_{p+1}(\mc{Y}): \zeta^{-1}\mc{F}^{p+1}/\zeta^{-1}\mc{F}^{p}\rightarrow\coor{\pi}_{\*}\mc{G}^{p+1}/\coor{\pi}_{\*}\mc{G}^{p}$. We want to show that $\gr_{p+1}(\mc{Y})$ is surjective. To that aim, let $s:[\phi]\mapsto\sum_{\substack{|\A|+|\Jj|=p+1\\\B~\Ii~\mu_{h}}}F_{\A\B}^{\Ii\Jj\mu_{h}}([\phi])~x^{\B}\partial^{\A}\otimes(y^{\Ii}\otimes u^{\Jj}h) \mod \mc{G}^{p}(\coor{\mc{N}})$ be a general section of the $\Cs$-module $\mc{G}^{p+1}/\mc{G}^{p}$ over $\coor{\mc{N}}$, where $(x^{i}, y^{j})$ as usual are the coordinates on $\bb C^{n-l}\times\bb C^{l}$. For all $u\in\zeta(W)$ with $p=\pi(u)$, the operator
\[D:=\Big(\hspace{-1em}\sum_{\substack{|\A|+|\Jj|=p+1\\|\Ii|\leq\nu,~\mu_{h}}}\hspace{-1.2em}\big(F_{\A\pmb{0}}^{\mu_{h}\Ii\Jj}\circ\varphi\circ\pi\big)(u)~\varphi_{p}^{\*-1}\circ(\partial^{\A}y^{\Ii}\partial^{\Jj}h)\circ\varphi_{p}^{\*} \mod \mc{J}(\zeta(W))^{\nu+1}\Big)_{\nu} \mod \mc{F}^{p}(\zeta(W))\]
with $\partial^{\A}=\Big(\pd{x^{1}}\Big)^{{\alpha_{1}}}\dots\Big(\pd{x^{n-l}}\Big)^{\alpha_{l}}$ and $\partial^{\Jj}:=\Big(\pd{y^{1}}\Big)^{{j_{1}}}\dots\Big(\pd{y^{l}}\Big)^{j_{l}}$ defines a section of the $\Cs$-module $\zeta^{-1}\mc{F}^{p+1}/\zeta^{-1}\mc{F}^{p}$ over $W$. The coordinate representation of $D$ with respect to a fixed parametrization $\phi$ is given by
\begin{align*} 
\Big(\hspace{-1em}\sum_{\substack{|\A|+|\Jj|=p+1\\|\Ii|\leq\nu,~\mu_{h}}}\hspace{-1.2em}\big(F_{\A\pmb{0}}^{\mu_{h}\Ii\Jj}\circ\varphi\circ\pi\circ\phi\big)(\*, 0)~(\varphi_{p}^{-1}\circ\phi)^{\*}\circ(\partial^{\A}y^{\Ii}\partial^{\Jj}h)\circ&(\varphi_{p}^{-1}\circ\phi)^{-1~\*} \mod \mf{n}^{\nu+1}\Big)_{\nu}\\
& \mod \mc{F}^{p}(\phi^{-1}(\zeta(W))) 
\end{align*}
Consequently, suppressing $\mf{n}^{\nu+1}$ as in \eqref{suppressprojlimit}, the image $\mc{Y}(W)(D)$ is the flat section 
\begin{align*}
s'\hspace{-0.2em}:\hspace{-0.1em}[\phi]\mapsto\hspace{-1.9em}\sum_{\substack{|\A|+|\Jj|=p+1\\\B~\Ii~\mu_{h}}}\hspace{-1.3em}\id\otimes\widehat{\Theta}_{c^{-1}}T_{\pmb{x}=\pmb{0}}\big(F_{\A\pmb{0}}^{\mu_{h}\Ii\Jj}\hspace{-0.4em}\circ\varphi\circ\pi\circ\phi(\*, 0)\big)T_{\pmb{x}=\pmb{0}}\big((\varphi_{p}^{-1}\circ\phi)^{\*}\hspace{-0.2em}\circ(\partial^{\A}y^{\Ii}&\partial^{\Jj}h)\hspace{-0.2em}\circ\hspace{-0.2em}(\varphi_{p}^{-1}\circ\phi)^{-1~\*}\big)\\
&\mod \mc{G}^{p}(\coor{\mc{N}}).
\end{align*}
Accordingly, the zeroth order term $s'([\phi])|_{\pmb{x}=\pmb{0}}$ in $\pmb{x}$ is equal to 
\begin{align*}
&\mathsmaller{\sum_{\substack{|\A|+|\Jj|=p+1\\\B~\Ii~\mu_{h}}}\big(F_{\A\pmb{0}}^{\mu_{h}\Ii\Jj}\circ\varphi\circ\pi\circ\phi\big)(0, 0)~g|_{\pmb{x}=\pmb{0}}^{-1~\*}\circ(\partial^{\A}\otimes y^{\Ii}\otimes u^{\Jj}h)\circ g|_{\pmb{x}=\pmb{0}}^{\*} \mod \mc{G}^{p}(\coor{\mc{N}})}\\
&=\hspace{-1em}\sum_{\substack{|\A|+|\Jj|=p+1\\\B~\Ii~\mu_{h}}}\hspace{-1em}F_{\A\pmb{0}}^{\mu_{h}\Ii\Jj}([\phi\circ g])~g|_{\pmb{x}=\pmb{0}}^{-1~\*}\circ
(\partial^{\A}\otimes y^{\Ii}\otimes u^{\Jj}h)\circ g|_{\pmb{x}=\pmb{0}}^{\*} \mod \mc{G}^{p}(\coor{\mc{N}})\\
&=\hspace{-1em}\sum_{\substack{|\A|+|\Jj|=p+1\\\B~\Ii~\mu_{h}}}\hspace{-1em}F_{\A\pmb{0}}^{\mu_{h}\Ii\Jj}([\phi])~\partial^{\A}\otimes y^{\Ii}\otimes u^{\Jj}h \mod \mc{G}^{p}(\coor{\mc{N}})
\end{align*}
whence by Proposition \ref{uniquenessoffs} $s=s'$. This means that $\gr_{p+1}(\mc{Y})(W)$ is surjective. Consequently, $\gr_{p+1}(\mc{Y})$ is surjective, too. Hence, it is an isomorphism. The induction hypothesis finally implies $\mc{Y}: \zeta^{-1}\mc{F}^{p+1}\rightarrow\coor{\pi}_{\*}\mc{G}^{p+1}$ is an isomorphism. Ergo, $\mc{Y}:\zeta^{-1}\mc{F}^{p}\cong\coor{\pi}_{\*}\mc{G}^{p}$ for all integers $p\geq0$.\\
It is a straightforward exercise to check that $\mc{Y}$ commutes with the natural injective connecting morphisms $i_{p}: \zeta^{-1}\mc{F}^{p}\rightarrow\zeta^{-1}\mc{F}^{p+1}$ and $j_{p}:\coor{\pi}_{\*}\mc{G}^{p}\rightarrow\coor{\pi}_{\*}\mc{G}^{p+1}$ of the inductive systems $(\zeta^{-1}\mc{F}^{p}, i_{p})$ and $(\coor{\pi}_{\*}\mc{G}^{p}, j_{p})$, respectively. Hence, $\mc{Y}$ induces an isomorphism between the inductive limits of the inductive systems. Finally, a consequent utilization of the functorial property of the colimit and the exhaustiveness of both filtrations concludes the proof.     
 \end{proof}
 \begin{lemma}
 \label{basenormalbundleiso}
 There is a filtered isomorphism of~~$\Cs$-algebras \[F_{W_{x, H}^i}:\OO_{\widehat{W}_{x}}\otimes_{\OO_{W_{x}}}\mc{H}_{1, c, W_{x}, H}\rightarrow\zeta^{-1}\big(\OO_{\hat{V}_{\epsilon}}\otimes_{\OO_{V_{\epsilon}}}\mc{H}_{1, c, V_{\epsilon}, H}\big)\] on $W_{x, H}^i$.
 \end{lemma}
 \begin{proof}
For every open subset $W$ of $W_{x, H}^{i}$ and an open $H$-invariant neighborhood $\tilde{W}$ of $W$ in $W_{x}$, define a map \[F_{W_{x, H}^i}(W): \OO_{\widehat{W}_{x}}(W)\otimes_{\OO_{W_{x}}(\tilde{W})}H_{1, c}(\tilde{W}, H)\rightarrow\OO_{\hat{V}_{\epsilon}}(\zeta(W))\otimes_{\OO_{V_{\epsilon}}(\zeta(\tilde{W}))}H(\zeta(\tilde{W}), H)\] on the generators of $\OO_{\widehat{W}_{x}}(W)\otimes_{\OO_{W_{x}}(\tilde{W})}H_{1, c}(\tilde{W}, H)$ by
\begin{align}
(f_{\nu} \mod \mc{J}^{\nu+1})_{\nu}&\mapsto(f_{\nu}\circ\zeta^{-1} \mod \mc{I}^{\nu+1})_{\nu}\label{as1}\\
\mc{L}_{Z}+\sum_{(h, Y^{h})}\frac{2c((h, Y^{h}))}{1-\lambda_{(h, Y^{h})}}f_{Y^{h}}(h-1)&\mapsto\mc{L}_{\zeta^{\*-1}\circ Z\circ\zeta^{\*}}+\sum_{(h, Y^{h})}\frac{2c((h, Y^{h}))}{1-\lambda_{(h, Y^{h})}}f_{Y^{h}}\circ\zeta^{-1}(h-1)\label{as2}\\
 h&\mapsto h 
  \end{align}
where $Z$ is a holomorphic vector field on $W_{x}$ and as usual $h\in H$. Assignment \eqref{as1} is well-defined by virtue of Theorem \ref{completediso}, assignement \eqref{as2} is well-defined, too by the fact that as a biholomorphism $\zeta$ maps codimension $1$ hypersurfaces in $W_{x}$ to codimension $1$ hypersurfaces in $V_{\epsilon}$. The map $\zeta$ is invertible, whence $F_{W_{x, H}^i}(W)$ is invertible, too. Conspicuously, it respects the $\bb C$-algebra structure of the source and the target and by definition, it also accounts for the filtration degrees, whence $F_{W_{x, H}^i}(W)$ is a filtered isomorphism of $\bb C$-algebras. The claim follows.    
 \end{proof}
Let $\Csf:=(W_{x})_{x\in X_{H}^{i}}$ be a collection of linear slices on $X$ such that the set $\Csf_{H}^{i}:=(W_{x, H}^{i})_{x\in X_{H}^{i}}$  forms an open cover of the stratum $X_{H}^{i}$. Moreover, for any pair of distinct linear slices $W_{x}, W_{y}\in\Csf$ set $\Omega:=W_{x}\cap W_{y}$ and $\Omega_{H}^{i}:=\Omega\cap X_{H}^{i}$.\\
Recall that on each $W_{x, H}^{i}$ the sheaf $\mc{F}_{W_{x, H}^{i}}:=\OO_{\widehat{W}_{x}}\otimes_{\OO_{W_{x}}}\mc{H}_{1, c, W_{x}, H}$ is well-defined, because $W_{x, H}^{i}$ is analytic inside of $W_{x}$. Moreover, due to the obvious fact that $\OO_{\widehat{W}_{x}}|_{\Omega_{H}^{i}}=\OO_{\widehat{W}_{y}}|_{\Omega_{H}^{i}}$, 
the restrictions of the sheaves $\mc{F}_{W_{x, H}^{i}}$ and $\mc{F}_{W_{y, H}^{i}}$ 
to the open intersection $\Omega_{H}^{i}$ coincide. Since the identity map trivially satisfies the cocycle condition, the collection $(\mc{F}_{W_{x, H}^{i}}, \id)$ forms a gluing data. With this one can construct a new auxiliary sheaf $\widehat{\mc{H}}_{1, c, X_{H}^{i}, H}$ on the whole of $X_{H}^{i}$ by the definition 
\begin{equation}
\label{newsheaf}
\widehat{\mc{H}}_{1, c, X_{H}^{i}, H}(U):=\{ (s_{x})_{x\in X_{H}^{i}}~|~s_{x}\in\mc{F}_{W_{x, H}^{i}}(U\cap W_{x, H}^{i})~~\mathrm{and}~~s_{y}|_{U\cap\Omega_{H}^{i}}=s_{x}|_{U\cap\Omega_{H}^{i}}\}  
\end{equation}
for every open $U\subseteq X_{H}^{i}$, where $x$ in the notation $s_{x}$ is an index and is not to be confused with the germ of sections at $x$. Clearly, for every $W_{x, H}^{i}\in\Csf_{H}^{i}$, the definition \eqref{newsheaf} implies 
\[\widehat{\mc{H}}_{1, c, X_{H}^{i}}|_{W_{x, H}^{i}}=\mc{F}_{W_{x, H}^{i}}.\] 
Let $\zeta_{1}$ and $\zeta_{2}$ be the open embeddings of $W_{x}$ and $W_{y}$ in $N|_{W_{x, H}^{i}}$ and $N|_{W_{y, H}^{i}}$, respectively, given in the proof of Thorem \ref{completediso} 
 pursuant to Cartan's Lemma. In the next lemma, we compare the $\Cs$-algebras $\fl$ and $\widehat{\mc{H}}_{1, c, X_{H}^{i}}$ on $X_{H}^{i}$.
\begin{proposition}
 \label{non-canonicaliso}
There is a filtered isomorphism of $\Cs$-algebras \[\mc{Z}: \widehat{\mc{H}}_{1, c, X_{H}^{i}, H}\rightarrow\fl\] on $X_{H}^{i}$. 
\end{proposition}
\begin{proof} 
Pursuant to Proposition \ref{filterediso} and Lemma \ref{basenormalbundleiso} for each $W_{x, H}^{i}$ from the open cover $\mathsf{C}_{H}^{i}$, defined above, there is a well-defined filtered isomorphism of $\Cs$-algebras 
\[\mc{Y}_{W_{x, H}^{i}}\circ F_{W_{x, H}^{i}}: \widehat{\mc{H}}_{1, c, X_{H}^{i}}|_{W_{x, H}^{i}}\rightarrow\fl|_{W_{x, H}^{i}}.\] 
It remains to be verified whether for any two open $W_{x, H}^{i}, W_{y, H}^{i}\in\Csf_{H}^{i}$ with $\Omega_{H}^{i}\neq\varnothing$ the corresponding maps $\mc{Y}_{W_{x, H}^{i}}\circ F_{W_{x, H}^{i}}$ and $\mc{Y}_{W_{y, H}^{i}}\circ F_{W_{y, H}^{i}}$ restrict to the same map 
\[\widehat{\mc{H}}_{1, c, X_{H}^{i}}|_{\Omega_{H}^{i}}\rightarrow\fl|_{\Omega_{H}^{i}}.\] 
To that aim, notice that 
although in general $\zeta_{1}|_{\Omega}\neq\zeta_{2}|_{\Omega}$, one can still identify $\zeta_{1}(\Omega)$ with $\zeta_{2}(\Omega)$ via the biholomorphism $\kappa: \zeta_{1}(\Omega)\rightarrow\zeta_{2}(\Omega)$, given by $\kappa:=\zeta_{2}\circ\zeta_{1}^{-1}$. Thus, for every parametrization $\phi$ of $\zeta_{1}(\Omega)\subset N|_{\Omega}$, the post-composition $\kappa\circ\phi$ is a parameterization of $\zeta_{2}(\Omega)$ and vice versa.  Let $D\in\mc{F}_{W_{x, H}^{i}}(\Omega_{H}^{i})=\mc{F}_{W_{y, H}^{i}}(\Omega_{H}^{i})$. Then, a successive  application of the maps $F_{W_{x, H}^{i}}(\Omega_{H}^i)$ and $F_{W_{y, H}^{i}}(\Omega_H^i)$ on $D$ yields accordingly an element 
\begin{equation*}
F_{W_{x, H}^{i}}(\Omega_{H}^{i})(D)=\zeta_{1}^{\*-1}\circ D\circ\zeta_{1}^{\*}~~\mathrm{in}~~\OO_{\hat{V}_{\epsilon}}(\zeta_{1}(\Omega_{H}^{i}))\otimes_{\OO_{V_{\epsilon}}(\zeta_{1}(\Omega))}H_{1, c}(\zeta_{1}(\Omega), H),
\end{equation*}
defined on a formal neighborhood of $\zeta_{1}(\Omega_{H}^{i})=\Sigma_{0}|_{\Omega_{H}^{i}}$ in $\zeta_{1}(\Omega)$, and an element
\begin{equation*}
F_{W_{y, H}^{i}}(\Omega_{H}^{i})(D)=\kappa^{\*-1}\circ\zeta_{1}^{\*-1}\circ D\circ\zeta_{1}^{\*}\circ\kappa^{\*}~~\mathrm{in}~~\OO_{\hat{U}_{\epsilon}}(\zeta_{2}(\Omega_{H}^{i}))\otimes_{\OO_{U_{\epsilon}}(\zeta_{2}(\Omega))}H_{1, c}(\zeta_{2}(\Omega), H),
\end{equation*}
defined on a formal neighborhood of $\zeta_{2}(\Omega_{H}^{i})=\Sigma_{0}|_{\Omega_{H}^{i}}$ in $\zeta_{2}(\Omega)$, where $V_{\epsilon}=\zeta_{1}(W_{x})$ and $U_{\epsilon}=\zeta_{2}(W_{y})$. In order to compare the Taylor series at $\pmb{x}=\pmb{0}$ of different operators on the holomorphic  normal bundle $N|_{\Omega_{H}^{i}}$, they first need to be pulled back to the same region in $\bb C^{n-l}\times\bb C^{l}$. Consequently, for any local parameterization $\phi$ of the normal bundle $N|_{\Omega_{H}^{i}}$, with respect to which we compute the Taylor expansion of $\zeta_{1}^{\*-1}\circ D\circ\zeta_{1}^{\*}$, we have correspondingly  
\begin{align*}
\mc{Y}_{W_{y, H}^{i}}(\Omega_{H}^{i})(\kappa^{\*-1}\circ&\zeta_{1}^{\*-1}\circ D\circ\zeta_{1}^{\*}\circ\kappa^{\*})=\\
&s: [\phi]\mapsto\id\otimes\widehat{\Theta}_{c}T_{\pmb{x}=\pmb{0}}((\kappa\circ\phi)^{\*}\circ\kappa^{\*-1}\circ\zeta_{1}^{\*-1}\circ D\circ\zeta_{1}^{\*}\circ\kappa^{\*}\circ(\phi\circ\kappa)^{\*-1})\\
\end{align*}
where $\kappa^{\*-1}\circ\zeta_{1}^{\*-1}\circ D\circ\zeta_{1}^{\*}\circ\kappa^{\*}$ is pull-backed by $\kappa\circ\phi$ to the same region in $\bb C^{n-l}\times\bb C^{l}$ as $\zeta_{1}^{\*-1}\circ D\circ\zeta_{1}^{\*}$ by $\phi$. We see that $\mc{Y}_{W_{y, H}^{i}}(\Omega_{H}^{i})\circ F_{W_{y, H}^{i}}(\Omega_{H}^{i})(D)$ agrees with $\mc{Y}_{W_{x, H}^{i}}(\Omega_{H}^{i})\circ F_{W_{x, H}^{i}}(\Omega_{H}^{i})(D)$. Ergo, $\mc{Y}_{W_{x, H}^{i}}\circ F_{W_{x, H}^{i}}$ and $\mc{Y}_{W_{y, H}^{i}}\circ F_{W_{y, H}^{i}}$ coincide on $\Omega_{H}^{i}$ and more generally on every open subset $W\subset\Omega_H^i$. Therefore there exists a unique morphism of $\Cs$-algebras
\[\mc{Z}:\widehat{\mc{H}}_{1, c, X_{H}^{i}}\rightarrow\fl\]
such that $\mc{Z}|_{W_{x, H}^{i}}=\mc{Y}_{W_{x, H}^{i}}\circ F_{W_{x, H}^{i}}$ is a filtered isomorphism. From the later we infer that $\mc{Z}$ is a filtered $\Cs$-algebra isomorphism.  
\end{proof}   
\section{Gluing of sheaves on the orbit type strata in $X$}
\label{gluing}
In this section we present the main result of this paper. Namely, we carry out a gluing of the $\Cs$-algebras $\fl$, obtained via localization on different orbit type strata $X_{H}^{i}$, $H\subset G$, $i$ a finite index, into a single sheaf of deformations of the $\Cs$-algebra $\mc{D}_X\rtimes G$. It turns out that the sheaf of deformations we get, is isomorphic to Etingof's sheaf of global Cherednik algebras. The main merit of our construction is twofold. On one hand, it is the first example, which in analogy to Fedosov's quantization gives a detailed recipe how via the framework of formal geometry a sheaf of non-commutative algebras can be deformed. The tools, developed in the exhibited construction, are aimed at a later proof of Dolgushev-Etingof's conjecture for arbitrary symplectic orbifolds. On the other hand, this construction lays down the ground for trace densities, computation of the Hochschild and the cyclic homology of the sheaf of formal global Cherednik algebras and finally, a subsequent derivation of an index theorem in line with \cite{RT12}. 
We start with some preliminaries.  
\subsection{Gluing of the localization sheaf on the principal stratum with the localization sheaves on the codimenstion $1$ strata}
\label{gluingofstrata01}
Let $H\subset G$, not necessarily generated by complex reflections, be such that the codimension of $X_H^i$ is $1$. Then, this stratum lies within the intersection of the complex reflection hypersurfaces $(h, Y)$ of $H$. In a linear space reflection hyperplanes coincide if and only if the corresponding roots and consequently complex reflections, are identical. Thus by Cartan's Lemma no two distinct complex reflection hypersurfaces  coincide. Hence the intersection of any number hyperplanes of $G$ has necessarily a codimension larger than one. That is why, the group $H$ possesses only one complex reflection whose hyperplane $(h, Y)$ entirely contains the stratum $X_{H}^{i}$. We remark however that in general $(Y, h)$ might contain other orbit type strata, as well.\\ 
Consider the disjoint union $\mathring{X}\coprod X_H^i$ and the subspace topology on it and let $j_{1}: \mathring{X}\hookrightarrow\mathring{X}\coprod X_{H}^{i}$ and $j_{2}: X_{H}^{i}\hookrightarrow\mathring{X}\coprod X_{H}^{i}$ be the obvious embeddings.  It can be proven as in Lemma \ref{openfiltration} that $\mathring{X}\coprod X_H^i$ is open in $X$. Moreover, as for every $x\in\mathring{X}$ and $h\in H$, we have that $\Stab(hx)=\id$, $\mathring{X}\coprod X_H^i$ is $H$-invariant. Below we show how to glue the $\Cs$-algebras $\coor{\tilde{\pi}}_{\*}\OO_{\mathrm{flat}}(\mathring{X}\times \widehat{D}_n)$ and $\one$ on $\mathring{X}$ and $X_H^i$, respectively, into a single subsheaf of 
\[j_{1\*}\coor{\tilde{\pi}}_{\*}\OO_{\mathrm{flat}}(\mathring{X}\times \widehat{D}_n)\oplus j_{2\*}\one\] 
on the open $H$-invariant susbspace $\mathring{X}\coprod X_H^i$ which is isomorphic to $\mc{H}_{1, c, \mathring{X}\coprod X_H^i, H}$ as a $\Cs$-algebra. 
The basis $\mf{B}_{\mathsmaller{\mathring{X}\coprod X_H^i}}^H:=\{U\cap\mathring{X}\coprod X_H^i~|~U\in\mf{B}_X^G\}$ for the subspace topology of $\mathring{X}\coprod  X_H^i$ is equivalent to the basis given by 
\begin{align*}
\label{hequivbasis}
&\{W_x~|~W_x~\textrm{is a $H$-invariant linear slice}\}\bigcup\Big\{U\subset\mathring{X}~|\begin{matrix}\mathsmaller{U~\textrm{is an $H$-invariant open set}}\\\mathsmaller{\textrm{s.t. $gU\cap U=\varnothing$ for all}~g\in G/H}\end{matrix}\Big\}.
\end{align*} 
Thus, henceforth we shall abuse notation and language by denoting the latter basis by $\mf{B}_{\mathsmaller{\mathring{X}\coprod X_H^i}}^H$ and by calling it the basis for the $H$-equivariant topology of $\mathring{X}\coprod X_H^i$. Notice that even though $X_{H}^{i}\subseteq(h, Y)$, as per  Cartan's Lemma we have for every $H$-linear slice $W_{x}$ that $W_{x}\cap X_{H}^{i}=W_{x}\cap(h, Y)$. For any $H$-invariant linear slice $W_{x}$ from $\mf{B}_{\mathsmaller{\mathring{X}\coprod X_H^i}}^H$, define 
\begin{align*}
&\mc{P}_0(\mathsmaller{W_{x}}):=\\
&\Bigg\{\begin{pmatrix}s_0\\ s_1\end{pmatrix}\in\begin{matrix}\mathsmaller{\coor{\tilde{\pi}}_{\*}\OO_{\textrm{flat}}(\mathring{X}\times\widehat{D}_n)(W_x\setminus X_{H}^{i})\otimes\bb CH}\\ \times\\\mathsmaller{\one(W_{x, H}^i)}\end{matrix}~~\Bigg|~~\begin{matrix}\mathsmaller{i) \mc{Y}(W_{x}\setminus X_{H}^{i})^{-1}(s_0)\in\OO(W_{x})[\mc{R}(W_{x})^{-1}]\otimes\mc{D}_{X}(W_{x})\rtimes H}\\\mathsmaller{ii)\id\otimes\widehat{\Theta}_c(s_1([\phi]))=i_{\psi}(\mc{Y}(W_x\setminus X_{H}^{i})^{-1}(s_0))~\textrm{for all}~[\phi]\in\coor{\mc{N}}}\\\hspace{15em}\mathsmaller{\textrm{with}~\psi(0, 0)\in W_{x}}\end{matrix}\Bigg\}
\end{align*}
where $\psi=\zeta^{-1}\circ\phi$ is a local parametrization of $W_x$, induced by the parametrization of the normal bundle, and identifying the coordinates on $W_x\subset X$ with the ones on the normal bundle, and  
\begin{equation}
\label{taylorembedding}
i_{\psi}:\OO(W_x)[\mc{R}(W_x)^{-1}]\otimes\mc{D}_{X}(W)\rtimes H\hookrightarrow\bb C\{\pmb{x}, y\}[S^{-1}]\otimes\big<\bb C\{\pmb{x}, y\}, \pd{x^1},\dots, \pd{x^{n-1}}, \pd{y}\big>\rtimes H
\end{equation} 
where $S=\{\alpha_h^{m}~|~\alpha_h~\textrm{is the coroot coresponding to}~d\rho(h)_x~\textrm{and}~m\in\bb N_{0}\}$ 
is a multiplicatively closed set in $\bb C\{\pmb{x}, \pmb{y}\}$. For any $U\in\mc{B}$ with $U\subset\mathring{X}$, define 
\begin{equation}
\label{gluingconditioni}
\mc{P}_0(U):=\{\begin{matrix}s_0\in\coor{\tilde{\pi}}_{\*}\OO_{\textrm{flat}}(\mathring{X}\times\widehat{D}_n)(U)\otimes\bb CH~|~\mc{Y}(U)^{-1}(s_0)\in\mc{D}_{X}(U)\rtimes H\end{matrix}\}.
\end{equation}
\begin{proposition}
\label{Calgebraonbasis}
$\mc{P}_0$ is a $\Cs$-algebra on the basis $\mf{B}_{\mathsmaller{\mathring{X}\coprod X_H^i}}^H$ of the open subspace $\mathring{X}\coprod X_H^i$ of $X$.
\end{proposition}
\begin{proof}
For every $W\in\mf{B}_{\mathsmaller{\mathring{X}\coprod X_H^i}}^H$ with $W\in{\mathring{X}}$, we have $\coor{\tilde{\pi}}_{\*}\OO(\mathring{X}\times\widehat{D}_n)(W)\otimes\bb CH=\mc{P}_0(W)$, whence $\mc{P}_0(W)$ is trivially a $\bb C$-algebra. Further, for every $H$-invariant linear slice $W_{x}\in\mf{B}_{\mathsmaller{\mathring{X}\coprod X_H^i}}^H$, the products in $\coor{\tilde{\pi}}_{\*}\OO(\mathring{X}\times\widehat{D}_n)(W_{x}\setminus X_{H}^{i})\otimes\bb CH$ and $\one(W_{x, H}^{i})$ induce a well-defined binary operation on $\mc{P}_0(W)$ by $(s_0, s_1)\cdot(s'_0, s'_1)=(s_0s'_0, s_1s'_1)$ for all $(s_0, s_1), (s'_0, s'_1)\in\mc{P}_0(W_{x})$. Indeed, we have that
\begin{align*}
i)\quad&\mc{Y}(W_{x}\setminus X_{H}^{i})^{-1}(s_{0}s'_{0})
\in\OO_{X}(W_{x})[\mc{R}(W_{x})^{-1}]\otimes\mc{D}_{X}(W_{x})\rtimes H
\end{align*}
by the fact that $\mc{Y}(W_{x}\setminus X_{H}^{i})$ is a $\bb C$-algebra morphism, and also
\begin{align*}
ii)\quad&\id\otimes\widehat{\Theta}_{c}(s_{1}s'_{1}([\phi]))=i_{\psi}(\mc{Y}(W_{x}\setminus X_{H}^{i})^{-1}(s_{0}))i_{\psi}(\mc{Y}(W_{x}\setminus X_{H}^{i})^{-1}(s'_{0}))\\
&=i_{\psi}(\mc{Y}(W_{x}\setminus X_{H}^{i})^{-1}(s_{0}s'_{0}))~\textrm{for all}~[\phi]\in\coor{\mc{N}}
\end{align*}
by the fact that $\id\otimes\widehat{\Theta}_{c}$ and $i_{\psi}$ are $\bb C$-algebra morphisms, whence $(s_{0}s'_{0}, s_{1}s'_{1})\in\mc{P}_0(W_{x})$. We leave it to the reader to do the easy verification, that the so defined operation endows $\mc{P}_0(W_{x})$ with a $\bb C$-algebra structure.\\
Take $H$-invariant linear slices $W_x, U_{x'}\in\mf{B}_{\mathsmaller{\mathring{X}\coprod X_H^i}}^H$ with $U_{x'}\subset W_x$ and corresponding biholomorphisms $\zeta: W_x\rightarrow V_{\epsilon}$ and $\zeta': U_{x'}\rightarrow V_{\epsilon'}$ as in Thorem \ref{completediso}. Define a restriction morphism $\res_{U_{x'}}^{W_x}: \mc{P}_0(W_x)\rightarrow\mc{P}_0(U_{x'})$ by $(s_0, s_1)\mapsto(s_0|_{U_x\setminus X_H^i}, s_1|_{U_{x, H}^i})$. 
Any given parametrization $\phi$ of $N$ induces two separate sets of coordinates $\zeta^{-1}\circ\phi$ and $\zeta'^{-1}\circ\phi$ on $U_{x'}$. In order for the restriction morphism to be well-defined, we have to make sure that the gluing conditions on $\mc{P}_0(U_{x'})$ are independent with respect to the choice of a parametrization of $U_{x'}$. Requirement $i)$ is apparently coordinate-independent. As for condition $ii)$, suppose that 
\[s_{1}: [\phi]\mapsto\sum_{\pmb{\alpha, \beta}}\sum_{\mu_h, \pmb{I}, \pmb{J}}F_{\pmb{\alpha\beta}}^{\mu_h\pmb{IJ}}([\phi])x^{\pmb{\beta}}\partial^{\pmb{\alpha}}\otimes (y^{\pmb{I}}\otimes u^{\pmb{J}}h)\] 
is an element of $\one(U_{x', H}^{i})$. 
Then, as pointed out in the proof of Proposition \ref{filterediso} 
\begin{equation*}
\id\otimes\widehat{\Theta}_{c}(s_{1}([\phi]))=T_{\pmb{x}=\pmb{0}}\Big(\sum_{\substack{\A~\Jj\\\B~\Ii~\mu_{h}}}\big(F_{\A\pmb{0}}^{\mu_{h}\Ii\Jj}\circ\varphi\circ\pi\circ\phi(\*, 0)\big)~(\varphi_{p}^{-1}\circ\phi)^{\*}\hspace{-0.2em}\circ(\partial^{\A}y^{\Ii}D^{\Jj}h)\hspace{-0.2em}\circ\hspace{-0.2em}(\varphi_{p}^{-1}\circ\phi)^{-1~\*}\Big)
\end{equation*}
where $D_{i}$ is the the Dunkl operator on $\bb C^{l}$ corresponding to the basis vector $u_{i}$.
One can rewrite condition $ii)$ in the form
\begin{align*}
T_{\pmb{x}=\pmb{0}}\Big(\sum_{\substack{\A~\Jj\\\B~\Ii~\mu_{h}}}\big(F_{\A\pmb{0}}^{\mu_{h}\Ii\Jj}\circ\varphi\circ\zeta^{-1}\circ\phi(\*, 0)\big)~(\varphi_{p}^{-1}\circ\phi)^{\*}\hspace{-0.2em}\circ(\partial^{\A}&y^{\Ii}D^{\Jj}h)\hspace{-0.2em}\circ\hspace{-0.2em}(\varphi_{p}^{-1}\circ\phi)^{-1~\*}\Big)=\\
&T_{\pmb{x}=\pmb{0}}((\zeta^{-1}\circ\phi)^{\*}\circ D_0\circ(\zeta^{-1}\circ\phi)^{-1\*}).
\end{align*}
where we account that $\zeta^{-1}=\pi$ on $\Sigma_{0}$. Obviously, a change of the parametrization $\kappa:=\zeta'^{-1}\circ\zeta$ leaves the above equality invariant \footnote{Alternatively, we argue as follows. Notice that with respect to $\zeta^{-1}\circ\phi$ constraint $ii)$ can be rewritten in the form 
\begin{equation}
\label{conditionii}
T_{\pmb{x}=\pmb{0}}(\phi^{\*}\circ(\zeta^{-1\*}\circ D_1\circ\zeta^{\*})\circ\phi^{\*})=T_{\pmb{x}=\pmb{0}}T_{\pmb{y}=\pmb{0}}((\zeta^{-1}\circ\phi)^{\*}\circ D_0\circ(\zeta^{-1}\circ\phi)^{-1\*})
\end{equation}
where we used that $s_1=\mc{Z}|_{W_{x, H}^i}(U_{x', H}^i)(D_1)$ and $s_0=\mc{Y}(U_{x'}\setminus X_H^i)(D_0)$ for some operators $D_1\in\OO_{\widehat{W}_x}(U_{x', H}^i)\otimes_{\OO_{W_x}(U_{x'})}H_{1, c}(U_{x'}, H)$ and $D_0\in\OO(U_{x'})[\mc{R}(U_{x'})^{-1}]\otimes\mc{D}_{X}(U_{x'})\rtimes H$. With respect to the parametrization $\zeta'^{-1}\circ\phi$ of $U_{x'}$ the coordinate representation of $D\in\{D_1, D_0\}$ change according to 
\[(\zeta^{-1}\circ\phi)^{\*}\circ D\circ(\zeta^{-1}\circ\phi)^{-1\*}\mapsto\kappa^{-1\*}\circ(\zeta^{-1}\circ\phi)^{\*}\circ D\circ(\zeta^{-1}\circ\phi)^{-1\*}\circ\kappa^{\*}\] 
which implies that the equality \eqref{conditionii} of Taylor series expansions remains valid in the new set of coordinates.}. For any subset $U\in\mf{B}_{\mathsmaller{\mathring{X}\coprod X_H^i}}^H$ with $U\subset\mathring{X}\cap W_x$, define the corresponding restriction map simply by $(s_0, s_1)\mapsto s_0|_{U}$. For this case there are no gluing conditions. Hence, gluing condition $i)$ and $ii)$ are coordinate independent, and consequently, all restriction morphisms are well-defined. Moreover, it is evident that $\res_W^{W_x}\circ\res_U^W=\res_U^{W_x}$ and $\res_U^U=\id_U$ for all $U, W\in\mf{B}_{\mathsmaller{\mathring{X}\coprod X_H^i}}^H$ with $W\subset U$. Hence, $\mc{P}_0$ is a presheaf. \\
Let $\cup_kW_k$ with $W_k\in\mf{B}_{\mathsmaller{\mathring{X}\coprod X_H^i}}^H$, $k\in\bb Z_{\geq1}$, be a cover of $W_x$. If $(s_0, s_1)\in\mc{P}_0(W_{x})$ such that $(s_0, s_1)|_{W_k}=0$ for all $W_k$ of the cover, then by definition, we have accordingly $s_0|_{W_k\setminus X_H^i}=0$ and $s_1|_{W_k\cap X_H^i}=0$. Since $\coor{\tilde{\pi}}_{\*}\OO(\mathring{X}\times\widehat{D}_n)$ and $\one$ are sheaves, whereas $\cup_k W_k\setminus X_H^i$ and $\cup_kW_k\cap X_H^i$ are covers of $W_x\setminus X_H^i$ and $W_{x, H}^i$, respectively, it follows that $s_0=0$ and $s_1=0$, hence $(s_0, s_1)=0$.  Suppose $s^{(k)}=(s_0^{(k)}, s_1^{(k)})\in\mc{P}_0(W_k)$, such that for all pairs $W_k$ and $W_{k'}$ of the open cover and for all $H$-invariant slices $W_{x'}\in\mf{B}_{\mathsmaller{\mathring{X}\coprod X_H^i}}^H$ with $W_{x'}\subseteq W_k\cap W_{k'}$, we have that $(s_0^{(k)}, s_1^{(k)})|_{ W_{x'}}=(s_0^{(k')}, s_1^{(k')})|_{W_{x'}}$. This means $s_0^{(k)}|_{W_{x'}\setminus X_H^i}=s_0^{(k')}|_{W_{x'}\setminus X_H^i}$ and $s_1^{(k)}|_{W_{x', H}^i}=s_1^{(k')}|_{W_{x', H}^i}$. Since $\coor{\tilde{\pi}}_{\*}\OO_{\textrm{flat}}(\mathring{X}\times\widehat{D}_n)$ and $\one$ are sheaves, there is a section $s_0$ of $\coor{\tilde{\pi}}_{\*}\OO_{\textrm{flat}}(\mathring{X}\times\widehat{D}_n)$ over $W_x\setminus X_H^i$ such that
\[\res_{W_k\setminus X_H^i}^{W_x\setminus X_H^i}(s_0)=s_0^{(k)}\] 
and a section $s_1$ of $\one$ over $W_{x, H}^i$ with
\[\res_{W_k\cap X_H^i}^{W_{x, H}^i}(s_1)=s_1^{(k)}.\] 
Since the gluing condition is local, the pair $(s_0, s_1)$ satisfies the gluing condition on $W_x$. Hence, $\mc{P}_0$ is a sheaf on the basis $\mf{B}_{\mathsmaller{\mathring{X}\coprod X_H^i}}^H$.
\end{proof}     
The sheaf $\mc{P}_0$ on the basis $\mf{B}_{\mathsmaller{\mathring{X}\coprod X_H^i}}^H$ can be uniquely  extended to a sheaf $\mc{P}_H^i$ on the whole subspace $\mathring{X}\coprod X_{H}^{i}$ via the rule
\begin{equation*}
\mc{P}_H^i(U)=\invlim_{\mf{B}_{\mathsmaller{\mathring{X}\coprod X_H^i}}^H\ni W\subset U}\mc{P}_{0}(W)
\end{equation*}
for every open set $U$ in $\mathring{X}\coprod X_{H}^{i}$.
\begin{theorem}
There is an isomorphism of~~$\Cs$-algebras between and $\mc{P}_H^i$ and $\mc{H}_{1, c, \mathring{X}\coprod X_{H}^{i}}$.
\end{theorem}
\begin{proof}
Fix a $H$-invariant linear slice $W_{x}\in\mf{B}_{\mathsmaller{\mathring{X}\coprod X_H^i}}^H$. Consider the map $\mc{P}_{0}(W_{x})\rightarrow H_{1, c} (W_{x}, H)$ given by $(s_{0}, s_{1})\mapsto\mc{Y}(W_x\setminus X_{H}^{i})^{-1}(s_{0})$
for every $(s_{0}, s_{1})\in\mc{P}_{0}(W_{x})$. We claim that this map is a well-defined isomorphism of $\bb C$-algebras. Indeed, condition $i)$ implies  
\begin{equation}
\label{conditioni}
\mc{Y}(W_x\setminus X_{H}^{i})^{-1}(s_{0})\in\OO_{X}(W_x)[\mc{R}(W_x)^{-1}]\otimes_{\OO_{X}(W_x)}\mc{D}_{X}(W_x)\rtimes H\cong H_{1, c}(W_x\setminus X_{H}^{i}, X).
\end{equation} 
Condition $ii)$ implies
\begin{align}
\label{laurentseries}
T_{\pmb{x}=\pmb{0}}(\phi^{\*}\circ\zeta^{-1\*}\circ D_{1}\circ\zeta^{\*}\circ\phi^{-1\*})=T_{\pmb{x}=\pmb{0}}T_{y=0}((\zeta^{-1}\circ\phi)^{\*}\circ\mc{Y}(U\setminus X_{H}^{i})^{-1}(s_{0})\circ(\zeta^{-1}\circ\phi)^{-1\*})
\end{align}
where $D_{1}=\mc{Z}(W_{x, H}^{i})^{-1}(s_{1})$ is an element of $\OO_{\widehat{W}_{x}}(W_{x, H}^{i})\otimes_{\OO_{W_{x}}(W_x)}H_{1, c}(W_x, H)$. The left hand side of \eqref{laurentseries} represents a Laurent series of an element of the formally completed global Cherednik algebra on $W_x$ with respect to $W_{x, H}^i$ which is convergent in $\pmb{x}$ and formal in the single transversal coordinate $y$. At the same time the expression on the right hand side is a convergent Laurent series of an element of the global Cherednik algebra on the open set $W_x\setminus X_H^i$. By the uniqueness of the Laurent series of holomorphic functions, \[\phi^{\*}\circ\zeta^{-1\*}\circ D_{1}\circ\zeta^{\*}\circ\phi^{-1\*} =(\zeta^{-1}\circ\phi)^{\*}\circ\mc{Y}(W_{x}\setminus X_{H}^{i})^{-1}(s_{0})\circ(\zeta^{-1}\circ\phi)^{-1\*}\]
on a small polydisc in $\bb C^{n-1}\times\bb C$ punctured at $y=0$. Ergo, $D_{1}=\mc{Y}(W_{x}\setminus X_{H}^{i})^{-1}(s_{0})$ on a small neighborhood in $W_{x}$. Eventually, the identity theorem implies  equality in $W_{x}$. Therefore, $\mc{Y}(W_{x}\setminus X_{H}^{i})^{-1}(s_{0})$ is an element of $H_{1, c}(W_{x}, H)$. That the map is a $\bb C$-algebra morphism follows from the fact that $\mc{Y}(W_{x}\setminus X_{H}^{i})$ is a $\bb C$-algebra morphism. Further,  suppose $\mc{Y}(W_{x}\setminus X_{H}^{i})^{-1}(s_{0})=\mc{Y}(W_{x}\setminus X_{H}^{i})^{-1}(s'_{0})$ for some $(s_{0}, s_{1}), (s'_{0}, s'_{1})\in\mc{P}(W_{x})$. The bijectivity of $\mc{Y}(W_{x}\setminus X_{H}^{i})$ coupled with condition $ii)$ and the injectivity of $\id\otimes\widehat{\Theta}_{c}$ imply successively $s_{0}=s'_{0}$ and then $s_{1}([\phi])=s'_{1}([\phi])$ for all $[\phi]\in\coor{\mc{N}}$, hence $(s_{0}, s_{1})=(s'_{0}, s'_{1})$. This corroborates the injectivity. With regard to the surjectivity, suppose $D\in H_{1, c}(W_x, H)$. Then, $(D \mod \mc{J}^{\nu+1})_{\nu}$, where $\mc{J}$ is the ideal sheaf of $W_{x, H}^i$ in $W_x$, is an element of the formally completed global Cherednik algebra on $W_x$, while $D$ is in particular an element of $H_{1, c}(W_x\setminus X_H^i, H)$. Ergo, $s_0:=\mc{Y}(W_x\setminus X_H^i)(D)$ and $s_1:=\mc{Z}(W_{x, H}^i)((D \mod \mc{J}^{\nu+1}))$ are sections of $\coor{\tilde{\pi}}_{\*}\OO_{\textrm{flat}}(\mathring{X}\times\widehat{D}_n)(W_x\setminus X_{H}^{i})\otimes\bb CH$ and $\one(W_{x, H}^i)$, respectively. For any $[\phi]\in\coor{\mc{N}}$, we have 
\begin{align*}
\id\otimes\widehat{\Theta}_c(s_1([\phi]))&=\big(T_{\pmb{x}=\pmb{0}}(\phi^{\*}\circ\zeta^{-1\*}\circ D\circ\zeta^{\*}\circ\phi^{-1\*}) \mod \mc{I}^{\nu+1}\big)_{\nu}\\
&=T_{\pmb{x}=\pmb{0}}T_{y=0}(\phi^{\*}\circ\zeta^{-1\*}\circ D\zeta^{\*}\circ\phi^{-1\*})\\
&=i_{\psi}(\mc{Y}(W_x\setminus X_H^i)(s_0))
\end{align*}
from which we conclude that $(s_0, s_1)\in\mc{P}_0(W_x)$. 
In the case of $U\in\mf{B}_{\mathsmaller{\mathring{X}\coprod X_H^i}}^H$ with $U\subseteq\mathring{X}$, define a map $\mc{P}_{0}(U)\rightarrow H_{1, c}(U, H)$ by $s_{0}\mapsto\mc{Y}(U)^{-1}(s_{0})$ for every $s_{0} \in\mc{P}_{0}(U)$. From gluing condition $i)$ in \eqref{gluingconditioni} it immediately follows that $\mc{P}_0(U)\cong\mc{D}_{X}(U)\rtimes H=H_{1, c}(U, H)$ as $\bb C$-algebras. Hence for every basic open set $W$, the maps 
\begin{align*}
&\mc{X}_{0}(W): \mc{P}_{0}(W)\rightarrow H_{1, c}(W, H)\\
&(s_{0}, s_{1})\mapsto\mc{Y}(W\setminus X_{H}^{i})^{-1}(s_{0})
\end{align*}
are well-defined. Since these maps furthermore commute with the restriction maps on $\mf{B}_{\mathsmaller{\mathring{X}\coprod X_H^i}}^H$, they give rise to an isomorphism $\mc{X}_{0}: \mc{P}_{0}\rightarrow\mc{H}_{1, c, \mathring{X}\coprod X_{H}^{i}, H}$ of $\Cs$-algebras on the basis $\mf{B}_{\mathsmaller{\mathring{X}\coprod X_H^i}}^H$. This isomorphism extends in turn to an isomorphism $\mc{X}: \mc{P}_H^i\rightarrow\mc{H}_{1, c, \mathring{X}\coprod X_{H}^{i}, H}$ of the induced $\Cs$-algebras on $X$. We are left to show that this map respects the filtartion..... 
\end{proof}
For every codimension $1$ stratum $X_H^i$ in $X$, we can construct a sheaf $\mc{P}_{H}^{i}$ repeating the above steps. Applying Puig's induction functor from Definition \ref{puiginductionfunctor} on each of the sheaves $\mc{P}_H^i$, we obtain sheaves of $\bb CG$-interior algebras on $G(\mathring{X}\coprod X_H^i)$, each of which is isomorphic to $\mc{H}_{1, c, G(\mathring{X}\coprod X_H^i), G}$. Since $\bigcup_{\codim(X_{H}^{i})=1}G(\mathring{X}\coprod X_{H}^{i})$ constitutes an open cover of $F^1(X)$ and the induced sheaves on each of the open sets of the cover coincide on $\bigcap_{\codim(X_H^i)=1}G(\mathring{X}\coprod X_H^i)=\mathring{X}$, we can glue them the standard way into a sheaf $\mc{S}_{1}$ on $F^1(X)$, which is unique up to an isomorphism, and isomorphic to $\mc{H}_{1, c, F^1(X), G}$. 
\subsection{Gluing of the deformation sheaf $\mc{S}_1$ on $F^{1}(X)$ with the localization sheaves on strata with codimension $2$ and higher}
We arrive at the main result of this paper, namely the successive gluing of the sheaf $\mc{S}_{1}$ with the localization sheaves on orbit types of strata which results in a unique extension of $\mc{S}_1$ to $X$ isomorphic to $\mc{H}_{1, c, X, G}$. When $c$ is taken to be a formal parameter, the extension $\mc{S}$ constitutes a formal deformation of $\mc{D}_X\rtimes\bb CG$ spelled entirely in terms of Gel'fand-Kazhdan formal geometry. This construction gives, as argued earlier the in the text, an important insight how to produce formal deformationzs of general Hecke orbifold algebras like the skew-group algebras of formal quantizations of $G$-manifolds. In particular, the formal deformation of $\mc{D}_X\rtimes\bb CG$ we provide is a constructive proof of Dolgushev-Etingof's conjecture in the case of a cotangent orbifold. On a more practical level, the construction of a deformation through localization sheaves gives us access to the homology theory and algebraic index theorems for $\mc{H}_{1, c, X, G}$.\\  
Let $x\in X_{K}^{j}$ such that $\codim(X_{K}^{j})=l+1\geq2$. Let $W_{x}$ be a $K$-invariant linear slice at $x$. As expounded in Subsection \ref{basis/charts}, all strata $X_H^i$ intersecting $W_x$ are such that $X_K^j\subset\overline{X_H^i}$ and all of them have necessarily a codimension lower than $l+1$. 
The complement of the analytic stratum $X_K^j$ in $W_x$, $W:=W_x\setminus X_K^j$, yields an open set in the standard topology of $F^{l}(X)$. Consequently, the inductions $\ind_{K}^{G}(W_x)$ and $\ind_K^G(W)$ are a basic open set in the $G$-equivariant topology of $X$, see Section \ref{basis/charts}, and a non-basic open set of $F^l(X)$ in the $G$-equivariant topology, respectively. In the following, for any $G$-invariant open subspace $Y$ of $X$, we denote by $\mf{B}_Y^G$ the set of all basic open sets $\ind_K^G(W_x)\in\mf{B}_X^G$, which lie in $Y$, and abuse the language by calling that set the basis of the $G$-equivariant subspace topology of $Y$. With these clarifications we can start extending the sheaf $\mc{S}_{1}$ on $F^{1}(X)$ successively to $F^i(X)$ for all $i\in\{1, \dots, l_{\max}\}$, where $l_{\max}= \max\limits_{X_j~\textrm{stratum}}\{\codim(X_j)\}\leq n$.      
\begin{theorem}
On each open set $F^i(X)$ there is a unique $\Cs$-algebra $\mc{S}_{i}$ such that $\mc{S}_{i}\cong\mc{H}_{1, c, F^i(X), G}$ and $\mc{S}_{i}|_{F^{i-1}(X)}\cong\mc{S}_{i-1}$. Consequently, there is maximal positive integer $l_{\max}\leq n$ such the $\Cs$-algebra $\mc{S}:=\mc{S}_{l_{\max}}$ is a sheaf on $X$ with $\mc{S}\cong\mc{H}_{1, c, X, G}$ and $\mc{S}|_{F^i(X)}\cong\mc{S}_{i}$ for all $i\leq l_{\max}$.  
\end{theorem}
\begin{proof}
We start with an extension of $\mc{S}_{1}$ on $F^2(X)$. Define first   
\begin{align}
\label{highergluing}
&\mathsmaller{\mc{S}_{2}^{\mathsmaller{GX_K^j}}(\ind_{K}^{G}(W_{x}))}:=\nonumber\\
&\Bigg\{\begin{pmatrix}\mathsmaller{q}\\ \mathsmaller{\sum\limits_{g, g'\in G/K}\hspace{-0.5em}g\otimes s_{g, g'}\otimes g'}\end{pmatrix}\in\begin{matrix}\hspace{-0.2em}\mathsmaller{\mc{S}_{1}(\ind_{K}^{G}(W))}\\ \hspace{-0.2em}\times\\\hspace{-0.2em}\mathsmaller{\Induc_K^G\big(\coor{\pi}_{\*}\OO_{\textrm{flat}}(\coor{\mc{N}}\times\mc{A}_{n-2, 2}^{K})(W_{x, K}^j)\big)}\end{matrix}~~\Bigg|~~\begin{matrix} \hspace{-0.7em}\mathsmaller{\sum\limits_{\hspace{0.5em}g\in G/K}\hspace{-0.7em}g\otimes\big(\hspace{-0.5em}\sum\limits_{g'\in G/K}\hspace{-0.5em}1\otimes\widehat{\Theta}_{c}(s_{g, g'}([\phi]))\otimes g'\big)=}\\\hspace{-1.3em}\mathsmaller{\sum\limits_{g\in G/K}i_{g\circ\psi}\big(\mc{X}(\ind_{K}^{G}(W))(q)|_{gW}\big)}
\end{matrix}\Bigg\}
\end{align}
where we recall that $\psi=\zeta^{-1}\circ\phi$ is the local coordinate chart of $W_x$ and $g\circ\psi$ are the corresponding local coordinate charts of the translates $gW_x$ for all $g\in G/K$, $\mc{X}(\ind_{K}^{G}(W))(q)|_{gW}$ are elements of the global Cherednik algebra on each of the open translates $gW$, and $i_{g\circ\psi}$ is defined in a similar fashion to \eqref{taylorembedding}. Taking into consideration the natural isomorphism \eqref{inductionidentificationsections} from Corollary \ref{inductionssections}, we can rewrite the gluing condition in \eqref{highergluing} in the following form 
\begin{align}
\label{highergluingcond2}
\sum\limits_{g, g'\in G/K}g\otimes(1\otimes\widehat{\Theta}_{c})(s_{g, g'}([\phi]))\otimes g'=\sum\limits_{\bar g, \bar g'\in G/K}\bar g\otimes i_{\psi}(d_{\bar g, \bar g'})\otimes \bar g' 
\end{align} 
where by abuse of notation we implicitly identified $\mc{X}(\ind_{K}^{G}(W))(q)$ with $\sum\limits_{\bar g, \bar g'\in G/K}\bar g\otimes d_{\bar g, \bar g'}\otimes \bar g' $ with $d_{\bar g, \bar g'}\in H_{1, c}(W, K)$ for all $(\bar g, \bar g')\in G/K\times G/K$. We distinguish two cases of restriction. In the first one, when we have $W_y\subseteq W_x$ such that $y\in F^1(X)$ and $\Stab(y)=:L<K$, there is an  (injective) restriction map
\begin{align*}
\mc{S}_{2}^{\mathsmaller{GX_K^j}}(\ind_K^G(W_x))&\longrightarrow\mc{S}_{1}(\ind_L^G(W_y))\\
(q, \sum\limits_{g, g\in G/K}g\otimes s_{g, g'}\otimes g)&\mapsto q|_{\ind_L^G(W_y)}.
\end{align*}
This restriction map is natural and coordinate independent. In the second case, when we have $W_y\subset W_x$ with $y\in X_K^j$, the restriction is given by 
\begin{align*}
\mc{S}_{2}^{\mathsmaller{GX_K^j}}(\ind_K^G(W_x))&\longrightarrow\mc{S}_{2}^{\mathsmaller{GX_K^j}}(\ind_K^G(W_y))\\
(q, \sum\limits_{g, g\in G/K}g\otimes s_{g, g'}\otimes g)&\mapsto (q|_{\ind_K^G(W_y\setminus X_K^j)}, \sum\limits_{g, g\in G/K}g\otimes s_{g, g'}|_{W_{y, K}^j}\otimes g).
\end{align*}
We show that the so-defined restriction map is coordinate independent following verbatim the proof of Proposition \ref{Calgebraonbasis}. In order to avoid unnecessary repetition, we leave the messy affirmation of this claim to the interested reader. As this restriction map is well-defined, as well, $\mc{S}_{2}^{\mathsmaller{GX_K^j}}$ is a presheaf on the basis of the $G$-equivariant subspace topology of $F^1(X)\coprod GX_K^j$. The gluing condition in \eqref{highergluing} or equivalently \eqref{highergluingcond2} is a local statement, whence the sheaf axioms are trivially satisfied by the same token as in the proof of Proposition \ref{Calgebraonbasis}. Therefore, $\mc{S}_{2}^{\mathsmaller{GX_K^j}}$ is a sheaf in the $G$-equivariant subspace topology of $F^1(X)\coprod GX_K^j$. In an analogous manner we obtain a family of sheaves $\mc{S}_{2}^{\mathsmaller{GX_{K'}^j}}$ on all open covering sets $F^1(X)\coprod GX_{K'}^j$ of $F^2(X)$ for all $X_{K'}^j$ with $\codim(X_{K'}^j)=2$. Since all those coincide identically on the sole common intersection $F^1(X)$, they can naturally be glued into a single $\Cs$-algebra $\mc{S}_2$ on $F^2(X)$. The only nontrivial statement from the theorem is to show that $\mc{S}_2\cong\mc{H}_{1, c, F^2(X), G}$. To that aim, for every basic open set $\ind_K^G(W_x)\in\mf{B}_{F^2(X)}^G$ with $x\in X_K^j$,
define a mapping $\mf{X}(\ind_K^G(W_x)): \mc{S}_2^{GX_K^j}(\ind_K^G(W_x))\rightarrow H_{1, c}(\ind_K^G(W_x), G)$ by 
\[(q, \sum_{g, g'\in G/K}g\otimes s_{g, g'}\otimes g')\mapsto\mc{X}(\ind_K^G(W))(q)\] 
for every $(q, \sum_{g, g'\in G/K}g\otimes s_{g, g'}\otimes g')\in\mc{S}_{2}^{GX_K^j}(\ind_K^G(W_x))$. This mapping is well-defined. Indeed, the second Riemann's extension theorem on complex manifolds implies that the restriction map $H_{1, c}(\ind_K^G(W_x), G)\rightarrow H_{1, c}(\ind_K^G(W), G)$ induces at the level of the graded associated algebras an isomorphism 
\begin{align*}
\gr_{\bullet}(H_{1, c}(\ind_K^G(W), G))&\cong\Sym_{\mathsmaller{\OO_{F^1(X)}(\ind_K^G(W))}}^{\bullet}\big(\mc{T}_{F^1(X)}(\ind_K^G(W))\big)\rtimes\bb CG\\
&\cong\Sym_{\mathsmaller{\oplus_g\OO_{F^1(X)}(gW)}}^{\bullet}\big(\oplus_{g\in G/K}\mc{T}_{F^1(X)}(gW)\big)\rtimes\bb CG\\
&\cong\Sym_{\mathsmaller{\oplus_g\OO_{F^1(X)}(gW)}}^{\bullet}\big(\oplus_{g\in G/K}\OO_{F^1(X)}(gW)^{\oplus n}\big)\rtimes\bb CG\\
&\cong\Big(\bigoplus\limits_{g, I}\OO_X(gW)\Big)\rtimes\bb CG\\
&\cong\bigoplus_I\OO_X(\ind_K^G(W_x))\rtimes\bb CG\\
&\cong\gr_{\bullet}(H_{1, c}(\ind_K^G(W_x), G))
\end{align*}
where in the third line we utilized the assumption that $W_x$, consequently $W$ are open sets trivializing the holomorphic tangent bundle. By the Five Lemma, the isomorphism of associated graded algebras implies 
\begin{equation}
\label{filteredisom}
F^i\big(H_{1, c}(\ind_K^G(W_x), G)\big)=F^i\big(H_{1, c}(\ind_K^G(W), G)\big)
\end{equation} 
for all $i$. As the filtration of the global Cherednik algebra is increasing and exhaustive,  \eqref{filteredisom} delivers an isomorphism of $\bb C$-algebras $H_{1, c}(\ind_K^G(W_x), G)\cong H_{1, c}(\ind_K^G(W), G)$, as desired. The map $\mf{X}(\ind_K^G(W_x))$ is a $\bb C$-algebra morphism because $\mc{X}(\ind_K^G(W))$ is a $\bb C$-algebra morphism. Assume now that 
\[\mf{X}(\ind_K^G)\big((q_1, \sum_{g_1, g_1'}g_1\otimes s_{g_1, g_1'}\otimes g_1')\big)=\mf{X}(\ind_K^G)\big((q_2, \sum_{g_2, g_2'}g_2\otimes s_{g_2, g_2'}\otimes g_2')\big).\]
Then, by definition $\mc{X}(\ind_K^G(W))(q_1)=\mc{X}(\ind_K^G(W))(q_2)$. As $\mc{X}(\ind_K^G(W))$ is an isomorphism itself, the aforementioned equality is true if and only if $q_1=q_2$.   
On the other hand, from the last identity together with the gluing condition in \eqref{highergluing} we deduce 
\[\sum\limits_{g_1, g_1'\in G/K}g_1\otimes(1\otimes\widehat{\Theta}_{c})(s_{g_1, g_1'}([\phi]))\otimes g_1'=\sum\limits_{g_2, g_2'\in G/K}g_2\otimes(1\otimes\widehat{\Theta}_{c})(s_{g_2, g_2'}([\phi]))\otimes g_2'.\] 
The latter is obviously true exactly when $\sum_{g_1, g_1'}g_1\otimes s_{g_1, g_1'}\otimes g_1'=\sum_{g_2, g_2'}g_2\otimes s_{g_2, g_2'}\otimes g_2'$. Therefore the morphism $\mf{X}(\ind_K^G(W_x))$ is injective. Before proceeding further with the surjectivity of $\mf{X}(\ind_K^G(W_x))$, note that for all $g\in G/K$, the translates $gW_{x, K}^j$ are analytic within $gW_x$ and hence, the union $\coprod_{g\in G/K}gW_{x, K}^j$ is analytic inside of $\mf{U}:=\coprod_{g\in G/K} gW_x$. 
Now, take an arbitrary element $D\in H_{1, c}(\ind_K^G(W), G)$ and consider the standard embedding 
\[\chi: H_{1, c}(\ind_K^G(W_x), G)\hookrightarrow\OO_{\widehat{\mf{U}}}(\coprod gW_{x, K}^j)\otimes_{\OO_{\mf{U}}(\mathsmaller{\mathsmaller{\coprod gW_x}})}H_{1, c}(\coprod gW_x, G).\]
Again by virtue of Corollary \ref{corinductionisom},  Corollary \ref{inductionssections} and in particular thanks to Isomorphism \eqref{inductionidentificationsections},  there is a family of elements $d_{g, g'}\in\OO_{\widehat{W_x}}(W_{x, K}^j)\otimes_{\OO_{W_x}(W_x)}H_{1, c}(W_x, G)$ such that $\chi(D)$ is uniquely identified with $\sum_{g, g'\in G/K} g\otimes d_{g, g'}\otimes g'$. Then, $\sum_{g, g'\in G/K} g\otimes s_{g, g'}\otimes g'$ with $s_{g, g'}:=\mc{Z}(W_{x, K}^j)(d_{g, g'})$ is an element of $\Induc_K^G\big(\coor{\pi}_{\*}\OO_{\textrm{flat}}(\coor{\mc{N}}\times\mc{A}_{n-l, l}^{K})(W_{x, K}^j)\big)$. On the other hand, since, as expounded in the above, there is a $\bb C$-algebra ismorphism $H_{1, c}(\ind_K^G(W_x), G)\cong H_{1, c}(\ind_K^G(W), G)$, $\mc{X}(\ind_K^G(W))^{-1}(D)$ is an element of $\mc{S}_{1}(\ind_K^G(W))$, as desired. 
Now, we are left to verify whether the gluing condition holds true for the element $(\mc{X}(\ind_K^G(W))^{-1}(D), \sum_{g, g'} g\otimes s_{g, g'}\otimes g')$, induced by $D$. Recall from Proposition \ref{non-canonicaliso} that $\mc{Z}(W_{x, K}^j)=\mc{Y}_{W_{x, K}^j}(W_{x, K}^j)\circ F_{W_{x, K}^j}(W_{x, K}^j)$. Hence, 
\begin{equation*}
\mc{Z}(W_{x, K}^j)(d_{g, g'})=[\phi]\mapsto1\otimes\widehat{\Theta}_c^{-1}\Big(T_{\pmb{x}=\pmb{0}}\big(\phi^{\*}\circ\zeta^{-1\*}\circ d_{g, g'}\circ\zeta^{\*}\circ\phi^{\*}\big)\Big).
\end{equation*}
Inserting this in the left hand side of the gluing condition in \eqref{highergluing}  yields 
\begin{align*}
\sum\limits_{g, g'\in G/K}g\otimes(1\otimes\widehat{\Theta}_c(s_{g, g'}([\phi]))\otimes g'&=\sum\limits_{g, g'\in G/K}g\otimes T_{\pmb{x}=\pmb{0}}\big(\phi^{\*}\circ\zeta^{-1\*}d_{g, g'}\circ\zeta^{\*}\circ\phi^{\*}\big)\otimes g'\\
&=\sum\limits_{g, g'\in G/K}g\otimes i_{\psi}(d_{g, g'})\otimes g'
\end{align*} 
which is identical with \eqref{highergluingcond2}. We infer that $(\mc{X}(\ind_K^G(W))^{-1}(D), \sum_{g, g'}\in g\otimes s_{g, g'}\otimes g')$ is an element in $\mc{S}_2^{\mathsmaller{GX_K^j}}(\ind_K^G(W_x))$ and consequently that $\mf{X}(\ind_K^G(W_x))$ is surjective. The surjectivity means that there are always non-trivial sections in \eqref{highergluing} satisfying the gluing condition. Consequently, $\mf{X}(\ind_K^G(W_x))$ is a $\bb C$-algebra isomorphism. We omit the trivial verification that $\mf{X}(\ind_K^G(W_x))$ is compatible with restriction maps on basic open sets, wherefore $\mf{X}$ is an isomorphism of $\Cs$-algebras on $\mf{B}_{F^2(X)}^G$ which induces an isomorphism  $\widehat{\mf{X}}: \mc{S}_2\rightarrow\mc{H}_{1, c, F^2(X), G}$ on $F^2(X)$.
Assume that there is an extension $\mc{S}_l$ on $F^l(X)$ with the properties specified in the theorem. Repeating the exact same steps as above, we construct an extension $\mc{S}_{l+1}$ on $F^{l+1}(X)$. This implies that for all members $F^i(X)$ of the filtration of $X$ there are $\Cs$-algebras $\mc{S}_i$ with the desired properties. Since the filtration is finite, there is a positive integer $l_{\max}\leq n$ for which $F^{l_{\max}}(X)=X$. Hence, there is a $\Cs$-algebra $\mc{S}$ satisfying $\mc{S}|_{F^i(X)}\cong\mc{S}_i$ for all $i\leq l_{\max}$. We are done.      
\end{proof}
\medskip
\medskip
\subsection*{Acknowledgements}
This work stems from my PhD at ETH Zurich. I am indebted to my supervisor Prof. Giovanni Felder for his assistance in my research process. This paper would not have been completed if not for his involvement in my work. I am grateful to him for his many insightful comments, suggestions and his academic generosity as a whole. His deep and broad knowledge of mathematics steered me out of fallacies on many occasions. I am equally grateful to Prof. Ajay Ramadoss for his interest in this work, his many valuable comments and the fruitful discussions I had with him throughout the completion of that paper.  
\newpage
\appendix
\section{Gel'fand-Kazhdan Formal Geometry}
\label{Chapter2} 
This appendix is a self-contained brief introduction to Gel'fand-Kazhdan formal geometry. It is structured the following way.
In the first subsection we define the notions Harish-Chandra pair, Harish-Chandra module and a (transitive) Harish-Chandra torsor. We devote a paragraph to the localization sheaf associated to a Harish-Chandra module. After a brief discussion of the formal disc and its automorphism group, we continue with a thorough discussion of the bundle of formal coordinates of a complex manifold $X$. In this part the key ingredients of formal geometry are defined using an algebraic-geometric language like in \cite{BK04}. Although we follow the original works by Gel'fand, Kazhdan and Fuks \cite{GK71},  as well as \cite{BR73} by Bernstein and Rozenfeld, our treatise is at least partially inspired by also more modern references such as \cite{calaque10, calaque12} and \cite{shilin15, shilin16, shilin17}. Finally, inspired by \cite{EF08} we generalize the bundle of formal coordinates to the case of vector bundles. In the course of our exposition we try to reconcile our definitions expressed in the language of ringed spaces with the well-known geometric ones given by Gelfand, Kazhdan and Fuks and by Bernstein and Rozenfeld based on jets and jet-bundles. Along the way we will try to be as comprehensive as possible in giving all the relevant references to sources whose material we use. 

\subsection*{Harish-Chandra torsors}
\label{hc}
\addtocontents{toc}{\protect\setcounter{tocdepth}{1}}
In the next we list the definitions of Harish-Chandra pairs and Harish-Chandra modules as found in the literature. Apart from the established introductory material in the theory of Harish-Chandra modules like \cite{BB93} we point to \cite{ggw16} as a very helpful summary of the most important facts related to Harish-Chandra modules. The following definitions are a rehash thereof.  
\begin{definition}
\label{hcpair}
A Harish-Chandra pair is a pair $(\mf{g}, K)$ of a Lie algebra $\mf{g}$ and a Lie group $K$ equipped with
\begin{enumerate}
\item[i)] a linear action $\rho: K\rightarrow\GL(\mf{g})$ of $K$ on $\mf{g}$,
\item[ii)] a $K$-equivariant embedding of Lie algebras $i:\Lie(K)\hookrightarrow\mf{g}$ such that the differential $\rho_{\ast}$ of the $K$-action coincides with the adjoint action of $\Lie(K)$ on $\mf{g}$ induced from the embedding $i$.
\end{enumerate}
\end{definition}
A central derived notion is the $(\mf{g}, K)$-module attached to a given Harish-Chandra pair $(\mf{g}, K)$. 
\begin{definition}
A $(\mf{g}, K)$-module is a complex vector space which has the structure of a $\mf g$-module via a Lie algebra homomorphism $\rho_{\ms{\mf g}}: \mf g\rightarrow \End(V)$ and a $K$-module via a Lie group homomorphism $\rho_{\ms{K}}: K\rightarrow\GL(V)$ such that the composition $\Lie(K)\xrightarrow{i}\mf g\xrightarrow{\rho_{\ms{\mf{g}}}}\End(V)$ recovers the differential  $\rho_{\ms{K}\ast}$ of $\rho_{\ms K}$. 
\end{definition} 
The definition of a \emph{principal $\mf g$-space} for a Lie algebra $\mf g$ was originally given by Gel'fand and Kazhdan in \cite{GK71}. Here we present a sheaf-theoretic definition which is equivalent to the one found in \cite{GK71} and  \cite{BR73}. 
\begin{definition}
A principal homogeneous $\mf g$-space or simply a principal $\mf{g}$-space is a complex manifold $X$ (not neccesarily finite dimensional) equipped with a morphism of Lie algebras $\phi: \mf g\rightarrow \mc{V}(X)$ which induces an isomorphism of vector bundles $\tilde{\phi}:\mf g\otimes_{\bb C}\OO_X\rightarrow\mc{T}_X $. 
\end{definition}

At the level of stalks we we have an isomorphism $\tilde{\phi}_x:\mf g\otimes_{\bb C}\OO_{X, x}\rightarrow\mc{T}_{X, x}$ for every $x\in X$ which induces a $\bb C$-linear isomorphism $\mf g\cong\mf g\otimes_{\bb C}\OO_{X, x}/\mf g\otimes\mf m_x\OO_{X, x}\cong\mc T_{X, x}/\mf m_x\mc{T}_{X, x}=T_x^{(1, 0)}X$, where $\mf m_x$ is the sole maximal ideal of the local $\bb C$-algeba $\OO_{X, x}$. This isomorphism coincides with the restriction $\phi_x$ of $\phi$ to $x$. This way we recover the original definition of a principal $\mf g$-space from \cite{BR73}, as desired. The inverse mapping $\omega$ of $\tilde{\phi}$ is interpreted as a $(1, 0)$-form with values in $\mf g$. 
From the fact that $\phi$ is a Lie algebra homomorphism it can be inferred that $\omega$ is a Maurer-Cartan form \cite[Proposition~2.1]{BR73}, that is, $\omega$ saturates $d\omega+\frac{1}{2}[\omega, \omega]=0$, where $[\omega, \omega]$ is given by the formula $[\omega, \omega](\zeta_1, \zeta_2)=[\omega(\zeta_1), \omega(\zeta_2)]-[\omega(\zeta_2), \omega(\zeta_1)]=2[\omega(\zeta_1), \omega(\zeta_2)]$ and $d=\partial+\bar{\partial}$ is the full de Rham differential on $X$. Herewith we have the needed prerequisites to define the \emph{Harish-Chandra torsor} and the special class of \emph{transitive Harish-Chandra torsors}, which in the present paper we are primarily interested in. From here on till the end of this section we follow \cite{BK04} except for slight changes where deemed necessary. %
\begin{definition}[Definition 2.3, \cite{BK04}]
\label{hctorsor}
A Harish Chandra $(\mf g, K)$-torsor over a complex manifold $X$ is a holomorphic principal $K$-bundle $P\rightarrow X$ equipped with a $K$-equivariant $(1, 0)$-form $\theta_{\ms P}: \mc T_P\rightarrow\mf g\otimes_{\bb C}\OO_P$ such that
\begin{enumerate}
\item[i)] $\theta_{\ms P}$ satisfies the Maurer-Cartan equation $d\theta_{\ms P}+\frac{1}{2}[\theta_{\ms P},  \theta_{\ms P}]=0$,
 \item[ii)] the diagram of $\OO_P$-modules
\begin{equation}
\begin{tikzcd}
\Lie(K)\otimes_{\bb C}\OO_P\arrow[rr, bend right, "i_{\ms P}" below]\arrow[r, hook, "j"]&\mc{T}_P\arrow[r, "\theta_{\ms P}"]&\mf g\otimes_{\bb C}\OO_P
\end{tikzcd} 
\end{equation}
commutes, where $j$ is the $K$-equivariant embedding of $\Lie(K)\otimes_{\bb C}\OO_P$ into $\mf g\otimes_{\bb C}\OO_P$ and $i_{\ms P}$ is the $K$-equivariant embedding induced by the Lie algebra embedding $i: \Lie(K)\hookrightarrow\mf g$ in the given Harish-Chandra pair $(\mf g, K)$. 
\end{enumerate}
\end{definition}
The condition $\theta_{\ms P}\circ j=i_{\ms{P}}$ implies that this differential $(1, 0)$-form restricts to a connection $(1, 0)$-form on the underlying principal $K$-bundle $P\rightarrow X$. %
Since the Maurer-Cartan condition is imposed by defualt on $\theta_{\ms P}$, the restriction of $\theta_{\ms P}$ to the holomorphic principal $K$-bundle is a flat connection $(1, 0)$-form for the principal $K$-bundle. It follows that in the spacial case when $\mf g=\Lie(K)$, the Harish-Chandra $(\mf g, K)$-torsor reduces to a standard principal $K$-bundle with a flat connection form. For that reason we refer to the $K$-equivariant $(1, 0)$-differential form $\theta_{\ms P}$ as the \emph{flat $\mf g$-valued connection $(1, 0)$-form on the Harish-Chandra $(\mf g, K)$-torsor $P\rightarrow X$}.\\
We remark that the Harish-Chandra $(\mf{g}, K)$-torsors over $X$ build a category, which in \cite{BK04} is denoted by $\mc{H}^{1}(X, (\mf g, K))$. As stated earlier, in this work we are almost exclusively interested in the special class of transitive Harish-Chandra torsors. Here we give the definition thereof.%
\begin{definition}[Definition 2.4, \cite{BK04}]
A transitive Harish-Chandra $(\mf g, K)$-torsor over a complex manifold $X$ is a Harish-Chandra $(\mf g, K)$-torsor whose $K$-equivariant connection $(1, 0)$-form $\theta_{\ms P}:\mc T_P\rightarrow\mf g\otimes_{\bb C}\OO_P$ is an isomorphism of vector bundles.
\end{definition}
By \eqref{ksadjointfunc} the inverse bundle map $\theta_{\ms P}^{-1}$ is equivalent to a $K$-equivariant mapping $\mf g\rightarrow\mc{V}(P)$, which due to the flatness of $\theta_{\ms P}$ is a Lie algebra homomorphism. Hence the total space $P$ of a transitive Harish-Chandra torsor has the structure of a homogeneous principal $\mf g$-space which justifies the used terminology. Conversely, if the total space of a holomorphic principal $K$-bundle $P\rightarrow X$ is equipped with the structure of a homogeneous principal $\mf g$-space with a $\mf g$-action on $P$ which is equivariant with respect to the $K$-action on $P$, and such that simultaneously condition $ii)$ of Definition \ref{hctorsor} is fulfilled, then $P\rightarrow X$ is a transitive Harish-Chandra $(\mf g, K)$-torsor.%
\subsection*{Localisation functor $\Loc$}
Given a principal $K$-bundle $P\rightarrow X$ and a pro-finite $K$-module $V$, we denote by $\underline{V_{\ms P}}$ the locally free pro-coherent $\OO_X$-module of sections of the associated vector bundle $P\times_{K}V$ over $X$. The importance of Harish-Chandra torsors derives from the fact that for any given $(\mf g, K)$-torsor together with a $(\mf g, K)$-module $V$, the $\mf g$-action on $V$ produces a flat connection on the associated vector bundle $\underline{V_{\ms P}}$. More concretely, if $\theta_{\ms P}$ is the flat connection $(1, 0)$-form of the $(\mf g, K)$-torsor $P\rightarrow X$, then the mapping $\nabla:\underline{V_{\ms P}}\rightarrow\Omega_{\ms P}^1\otimes\underline{V_{\ms P}}$ defined by %
$\nabla s:=ds+\theta_{\ms P}\cdot s$ 
for any local section $s\in H^0(U, \underline{V_{\ms P}})$ over an open set $U$ in $X$,  where the dot $\cdot$ signifies the $\mf g$-action on the $(\mf g, K)$-module $V$, is a flat connection on $\underline{V_{\ms P}}$. This way, if we denote by $(\mf g, K)-\Mod$ the category of $(\mf g, K)$-modules and by $\VB_{\fflat}$ the category of flat vector bundles over $X$, then a fixed torsor $P\rightarrow X$ defines a functor 
\begin{equation}
\label{loc}
\Loc(P,-): (\mf g, K)-\Mod\rightarrow\VB_{\fflat}
\end{equation}
by $\Loc(P, V)=\underline{V_{\ms P}}$ which in \cite{BK04} is referred to as the  \emph{localization functor}. 
In formal geometry one exploits this localization property of Harish-Chandra torsors and modules to construct various canonical sheaves on manifolds as sheaves of flat sections of vector bundles that descend from Harish-Chandra torsors and modules.  The most wide spread examples of Harish-Chandra torsors in formal geometry are the torsor of formal coordinates on a manifold and that of formal coordinates on a vector bundle. We shall discuss both examples in a greater detail in the upcoming section.
%
%
%
%
\subsection*{The automorphism group of the formal disc}
\label{formaldisc}
Let $\widehat{\OO}_n=\bb C[[y_1, \dots, y_n]]$ be the sheaf of germs of functions on a formal neighborhood of $0$ in $\bb C^n$. Endowed with the Frechet topology $\widehat{\OO}_n$ can be seen as a topological local $\bb C$-algebra. The unique maximal ideal $\widehat{\mf{m}}_0$ of this $\bb C$-algebra consists of formal power series with vanishing constant term and $\widehat{\OO}_n$ is complete with respect to it. We call the local $\bb C$-ringed space $\Spf(\widehat{\OO}_n):=(0, \widehat{\OO}_n)$ the \emph{formal polydisc} at $0$ in $\bb C^n$.
Adhering to the nomenclatur in \cite{ggw16}, we denote by $\Aut_n$  the group of filtration-preserving automorphisms of $\widehat{\OO}_n$. Alternatively, one can see $\Aut_n$ as the group of automorphisms of $\Spf(\widehat{\OO}_n)$ which is the same as $\infty$-jets of local biholomorphisms of $\bb C^n$ fixing the origin $0$. Each automorphism $f\in\Aut_n$ induces an automorphisms of the quotient $\widehat{\OO}_n/\widehat{\mf{m}}_0^{k+1}$ for every $k\in\bb{Z}_{\geq0}$. This gives a group homomorphism $\sigma: \Aut_n\rightarrow\Aut(\widehat{\OO}_n/\widehat{\mf{m}}_0^{k+1})$ whose image in $\Aut(\widehat{\OO}_n/\widehat{\mf{m}}_0^{k+1})$ we designate by $\Aut_{n, k}$ . Since each quotient algebra $\widehat{\OO}_n/\widehat{\mf{m}}_0^{k+1}$ is a finite dimensional vector space of complex dimension $\binom{n+k}{n}$, the group $\Aut_{n, k}$ is a finite dimensional Lie group. The group $\Aut_n$ of automorphisms of $\widehat{\OO}_n$ is the projective limit of the inverse system $\dots\rightarrow\Aut_{n,k}\rightarrow\Aut_{n,k-1}\rightarrow\dots\rightarrow\Aut_{n, 1}=\GL_{n}(\bb C)$. The kernel of the group homomorphism $p: \Aut_n\rightarrow\GL(n, \bb C)$, given by $p(\phi)=d_0\phi$, $\phi\in\Aut_n$, is a normal subgroup of $\Aut_n$, denoted by $\Aut_n^+$. It is a pro-nilpotent group, hence topologically contractible.\\
Let us denote by $W_n$ the Lie algebra of derivations of $\widehat{\OO}_n$, that is, the Lie algebra of formal vector fields on the formal polydisc $\Spf(\widehat{\OO}_n)$. The Lie bracket of formal vector fields $v$ and $w$ from $W_n$ is expressed by their commutator which in coordinates $(y_1, \dots, y_n)$ on $V$ is given by
\[[v, w]=\sum_{i, j=1}^n\big(v_i\pdf{w_j}{ y_i}-w_i\pdf{v_j}{ y_i}\big)\pd{y_j},\] 
where $v_i, w_j\in\widehat{\OO}_n$. Set $W_{n, k}:=W_n/\widehat{\mf{m}}_0^{k+1}$. These are finite dimensional Lie algebras of complex dimension $n\binom{n+k}{n}$, which form a projective system
\[\dots\rightarrow W_{n,k}\rightarrow W_{n,k-1}\rightarrow\dots\rightarrow W_{n, 1}=\bb C\ltimes\mf{gl}_n\]
in which the connecting morphisms are the natural projection mappings. The inverse limit of that projective system is isomorphic to $W_n$ which makes $W_n$ a pro-finite Lie algebra. The Lie algebra of the infinite dimensional Lie group $\Aut_n$ is the Lie subalgebra $W_n^0\subset W_n$ of formal vector fields on $\Spf(\widehat{\OO}_n)$ with zero constant coefficient. Respectively, the Lie algebra of $\Aut_{n, k}$ is the $n\binom{n+k}{n}-n$ dimensional quotient Lie algebra $W_{n, k}^0:=W_n^0/\hat{\mf m}_0W_n^0$ of formal vector fields with vanishing constant terms and terms of power $k+1$ and higher. It is a matter of a straightforward verification that $(\Aut_{n, k}, W_{n, k})$ correspondingly $(\Aut_n, W_n)$ are Harish-Chandra pairs.
\subsection*{The bundle of formal coordinates $\coor{X}$}
\label{formalbundle}
Henceforth in this section we assume that $X$ is an $n$-dimensional complex manifold.  
\begin{definition}
\label{algebraicdef}
$X^k$ is the set of closed immersions of $\bb C$-ringed spaces~~$\pmb\varphi:(0, \OO_{\bb{C}^n}/\mc{I}_0^{k+1})\hookrightarrow(X, \OO_X)$, where $\mc{I}_0$ denotes the ideal sheaf of the origin in $\bb C^n$.
\end{definition}
Let us define a projection map $\pi_k: X^k\rightarrow X$ by $\pi_k((\varphi, \varphi^{\#}))=\varphi(0)$. We equip $X^k$ with the identification topology with respect to which the projection map $\pi^k$ becomes open, closed and continuous. We denote by $X_x^k$ the fiber $\pi_k^{-1}(x)$ of $x\in X$.
\begin{lemma}
\label{fiberiso}
$X^k$ is the disjoint union of the fibers $X_x^k$, that is, $X^k=\coprod_{x\in X}X_x^k$. Moreover, each fiber $X_x^k$ consists of  $\bb C$-algebra isomorphisms $\OO_{X, x}/\mathfrak{m}_x^{k+1}\rightarrow\bb C[t_1, \dots, t_n]/\mf{m}^{k+1}$, where $\mf{m}_x$ is the unique maximal ideal of $\OO_{X, x}$ and $\mf{m}$ is the maximal ideal $(t_1, \dots, t_n)$ in $\bb C[t_1, \dots, t_n]$. 
\end{lemma}
\begin{proof}
By default the morphism $\pmb{\varphi}$ between the local $\bb C$-ringed spaces $(0, \OO_{\bb{C}^n}/\mathfrak{m}_0^{k+1})$ and $(X, \OO_X)$ is given by a pair $(\varphi, \varphi^{\#})$ of a closed injective holomorphic map $\varphi: \{0\}\hookrightarrow X$ and a  surjective homomorphism of $\Cs$-algebras $\varphi^{\#}:\OO_X\twoheadrightarrow\varphi_{\ast}\OO_{\bb{C}^n}/\mc{I}_0^{k+1}$. Since 
$\supp(\varphi_{\ast}\OO_{\bb{C}^n}/\mc{I}_0^{k+1})=\overline{\varphi(\supp(\OO_{\bb{C}^n}/\mc{I}_0^{k+1}))}=\overline{\{\varphi(0)\}}=\{\varphi(0)\}$, 
the direct image $\varphi_{\ast}\OO_{\bb{C}^n}/\mc{I}_0^{k+1}$ can be viewed as a skyscraper sheaf on $X$ centered at $\varphi(0)\in X$. The morphism of sheaves $\varphi^{\#}$ is a collection of $\bb C$-algebra epimorphisms
\begin{align}
\label{epimor}
&\varphi^{\#}(U):\OO_X(U)\twoheadrightarrow\varphi_{\ast}(\OO_{\bb{C}^n}/\mc{I}_0^{k+1})(U)
\end{align}
for all open set $U$ in $X$, which are compatible with the restriction morphisms. The fact that $\varphi_{\ast}\OO_{\bb{C}^n}/\mc{I}_0^{k+1}$ is a skyscraper sheaf on $X$ entails that
\begin{equation*}
\varphi_{\ast}(\OO_{\bb{C}^n}/\mc{I}_0^{k+1})(U)=\begin{cases}\varphi_{\ast}(\OO_{\bb{C}^n}/\mc{I}_0^{k+1})_{\varphi(0)},~~\textrm{if}~~\varphi(0)\in U\\ 0,~~\textrm{if}~~\varphi(0)\notin U.\end{cases}
\end{equation*}
Therefore, for every open subset $U\subset X$ with $\varphi(0)\in U$ the morphism \eqref{epimor} reduces to the surjective morphism $\varphi^{\#}(U):\OO_X(U)\twoheadrightarrow\varphi_{\ast}(\OO_{\bb{C}^n}/\mc{I}_0^{k+1})_{\varphi(0)}$. For each pair of open sets $V\subseteq U$ the surjective maps $\varphi^{\#}(U)$ and $\varphi^{\#}(V)$ are such that $\varphi^{\#}(V)\circ\res^U_V=\varphi^{\#}(U)$ and hence the universal property of the stalk $\OO_{X,  \varphi(0)}$ induces a surjective morphism $\varphi^{\#}_{\mathsmaller{\varphi(0)}}: \OO_{X, \varphi(0)}\rightarrow\varphi_{\*}(\OO_{\bb C^n, 0}/\mc{I}_0^{k+1})_{\varphi(0)}$ which makes the diagram
\begin{equation}
\begin{tikzcd}[row sep=tiny, column sep=normal]
\label{univpropogerm}
\OO_{X}(U)\arrow[dd, "\res^U_V" left]\arrow[rd, hook]\arrow[rrd, "\varphi^{\#}(U)", bend left, two heads]&&\\
&\OO_{X, \varphi(0)}\arrow[r, dashed, "\varphi^{\#}_{\mathsmaller{\varphi(0)}}"]&\varphi_{\*}(\OO_{\bb C^n, 0}/\mc{I}_0^{k+1})_{\varphi(0)}\\
\OO_X(V)\arrow[ur, hook]\arrow[rru, "\varphi^{\#}(V)" below, bend right, two heads]&&
\end{tikzcd}
\end{equation}
commutative. We observe that the map $\varphi$ is proper since $\varphi^{-1}(\varphi(0))=\{0\}$ is a closed and compact space. Thus by Remark $2.5.3$ in \cite{KS94} we have $\varphi_{\ast}(\OO_{\bb{C}^n}/\mc{I}_0^{k+1})_{\varphi(0)}\cong\Gamma(0, \OO_{\bb{C}^n}/\mc{I}_0^{k+1}|_{0})=\Gamma(0, \OO_{\bb{C}^n}/\mc{I}_0^{k+1})=\OO_{\bb{C}^n, 0}/\mf{m}_0^{k+1}$, 
where in the last equality we accounted for the fact that $\OO_{\bb{C}^n}/\mc{I}_0^{k+1}$ is a constant sheaf on $0$ and $\mf{m}_0$ denotes the stalk of the ideal sheaf $\mc{I}_0$ at $0$. 
Ergo, the commutative diagram \eqref{univpropogerm} implies that the epimorphism $\varphi^{\#}: \OO_{X}\twoheadrightarrow\varphi_{\*}(\OO_{\bb C^n}/\mc{I}_0^{k+1})_{\varphi(0)}\cong\OO_{\bb C^n, 0}/\mf{m}_0^{k+1}$
factors through the surjective morphism on stalks 
\begin{equation}
\label{factoringmap}
\varphi^{\#}_{\mathsmaller{\varphi(0)}}: \OO_{X, \varphi(0)}\twoheadrightarrow\OO_{\bb C^n, 0}/\mf{m}_0^{k+1}
\end{equation}
by means of the natural embedding $\OO_X\hookrightarrow\OO_{X, \varphi(0)}$. Herewith and the natural isomorphism $\OO_{\bb C^n, 0}/\mf{m}_0^{k+1}\cong\bb C[t_1, \dots, t_n]/\mf{m}^{k+1}$ emanating from the $k$-th order Taylor expansion of germs of holomorphic functions on $\bb C^n$ at $0$ we finally get
\begin{align*}
X^k&=\{\pmb{\varphi}: (0, \OO_{\bb C^n}/\mc{I}_0^{k+1})\rightarrow(X, \OO_X)~|~\pmb{\varphi}~\textrm{a closed immersion of ringed spaces}~\}\\
&=\{(\varphi, \Phi_{\ms{\varphi(0)}})~|~\varphi:\{0\}\hookrightarrow X~\textrm{embedding and}~ \Phi_{\ms{\varphi(0)}}:\OO_{X, \varphi(0)}/\mf{m}_{\varphi(0)}^{k+1}\cong\bb C[t_1, \dots, t_n]/\mf{m}^{k+1} \}\\
&=\coprod_{x\in X}\{(x,\Phi_x)~|~\Phi_x:\OO_{X, x}/\mf{m}_{x}^{k+1}\cong\bb C[t_1, \dots, t_n]/\mf{m}^{k+1}\}
\end{align*}
Lastly, from the above we infer that for a fixed $x\in X$ the corresponding fiber $X^k_x$  is comprised of isomorphism of $\bb C$-algebras $\Phi_x:\OO_{X, x}/\mf{m}_x^{k+1}\cong\varphi_{\ast}(\OO_{\bb{C}^n}/\mc{I}_0^{k+1})_x$, as desired. 
\end{proof}
Let $\Delta: X\hookrightarrow X\times X$ be the diagonal embedding. The diagonal $\Delta(X)$ is a closed submanifold, whose $k$-infinitesimal neighborhood $(\Delta(X),\OO_{\Delta^{(k)}}:=\OO_{X\times X}/\mc{I}_{\Delta(X)}^{k+1})$ in $X\times X$ is denote by $\Delta^{(k)}$. Fiurthermore, let $\pr_i: X\times X\rightarrow X$, $i=1, 2$, be the projections onto the first and second component of $X\times X$, respectively. Define the \emph{jet bundle of order $k$} as the sheaf of $\bb C$-algebras $\mc{J}_X^k:=\pr_{1\*}\OO_{\Delta^{(k)}}$. It also has the structure of a finite locally free $\OO_X$-module of rank $\binom{n+k}{n}$. The jet bundle $\mc{J}_X^{\infty}$ of infinite order is defined as the projective limit of the corresponding inverse system. At each $x\in X$ the fiber $\mc{J}_X^{k}$, respectively $\mc{J}_X^{\infty}$ can naturally be identified with the algebra of $k$-jets, respectively $\infty$-jets of holomorphic functions at $x$. 
\begin{corollary}
\label{fiberjet}
The fiber $X_x^k$ is bijective to the set of k-jets of local biholomorphisms $\phi: \bb C^n\rightarrow X$ with $\phi(0)=x$. 
\end{corollary}
\begin{proof}
$\OO_{X, x}$ and $\bb C[t_1, \dots, t_n]$ are commutative rings with maximal ideals $\mf{m}_x$ and $\mf{m}$, respectively. Therefore, the quotient $\bb C$-algebras $\OO_{X, x}/\mf{m}_x^{k+1}$ respectively $\bb C[t_1, \dots, t_n]/\mf{m}^{k+1}$ are local rings with corresponding maximal ideals $\mf{m}_x/\mf{m}_x^{k+1}$ and $\mf{m}/\mf{m}^{k+1}$. Let $\Phi: \OO_{X, x}/\mf{m}_x^{k+1}\rightarrow\bb C[t_1, \dots, t_n]/\mf{m}^{k+1}$ be an element in the fiber $X^k_x$. The fact that $\Phi$ is a $\bb C$-algebra isomorphism between two local rings has as a consequence that $\Phi(\mf{m}_x/\mf{m}_x^{k+1})=\mf{m}/\mf{m}^{k+1}$. This along with the fact that $\mf{m}/\mf{m}^{k+1}$ has a basis consisting of the residue classes $t_1\mod \mf{m}^{k+1},\dots, t_n\mod\mf{m}^{k+1}$ implies that $\mf{m}_x/\mf{m}_x^{k+1}$ has $n$ basis elements, too which we designate by $f_1 \mod \mf{m}_x^{k+1}, \dots, f_n\mod\mf{m}_x^{k+1}$. Define a map $\psi: U\rightarrow\bb C^n$ by $\psi(y):=(f_1(y), \dots, f_n(y))$ for every $y\in U$ where $U$ is an open  neighbourhood of $x\in X$ and $f_1, \dots, f_n$ are representatives of the respective residue classes of basis vectors in $\mf{m}_x/\mf{m}_x^{k+1}$. By shrinking the image $\Ima(\psi)$ we can make $\psi$ a biholomorphism. Let us designate by $\phi$ the local inverse $\bb{C}^n\supset V\rightarrow U$ of $\psi$ which defines a local parametrization at $x$ and let us set $\Phi(f_i \mod \mf{m}_x^{k+1})=: r_i \mod \mf{m}^{k+1}$ for $i=1, \dots, n$ which are basis elements of $\mf{m}/\mf{m}^{k+1}$. Then the map $\Phi$ factors through the $\OO_{X, x}$-module $\mc{J}_{X, x}^{k}$ of $k$-jets of holomorphic functions on $X$ at $x$ with factoring $\bb C$-algebra isomorphisms $F: \OO_{X, x}/\mf{m}_x^{k+1}\rightarrow\mc{J}_{X, x}^{k}$, $f \mod \mf{m}_x^{k+1}\mapsto[f]_x^k$ and $G: \mc{J}_{X, x}^{k}\rightarrow\bb C[t_1, \dots, t_n]/\mf{m}^{k+1}$, $[f]_x^{k}\mapsto\sum_{|\alpha|\leq k}([f\circ\phi]_x^{k})_{\alpha}r^{\alpha}\mod\mf{m}^{k+1}$, respectively, as shown in the diagram
\[
\begin{tikzcd}
\OO_{X, x}/\mf{m}_x^{k+1}\arrow[d, "F"]\arrow[r, "\Phi"]&\bb C[t_1, \dots, t_n]/\mf{m}^{k+1}\\
\mc{J}_{X, x}\arrow[ru, "G", dashed]&
\end{tikzcd}
\]
Indeed, $G\circ F=\Phi$ on generators $f_i \mod \mf{m}_x^{k+1}$ and since $\OO_{X, x}/\mf{m}_x^{k+1}\cong\mf{m}_x/\mf{m}_x^{k+1}\oplus\bb C$ and $\bb C[t_1, \dots, t_n]/\mf{m}^{k+1}\cong\mf{m}/\mf{m}^{k+1}\oplus\bb C$ as $\bb C$-vector spaces, it holds that $G\circ F=\Phi$ on every element from $\OO_{X, x}/\mf{m}_x^{k+1}$. This implies \[\Iso_{\bb C}(\OO_{X, x}/\mf{m}_x^{k+1}, \bb C[t_1, \dots, t_n]/\mf{m}^{k+1})\cong\Iso_{\bb C}(\mc{J}_{X, x}^k,  \bb C[t_1, \dots, t_n]/\mf{m}^{k+1}).\] 
The claim follows.

\end{proof}
The last corollary provides an alternative definition of $X^k$ as a collection of pairs consisting of a point $x$ in $X$ and a $k$-jet at $0\in\bb C^n$ of a parametrization $\phi: \bb C^n\supset V\rightarrow U$ with $0\in V$ and $\phi(0)=x\in U$. Contrary to Definition \ref{algebraicdef} which makes sense only in the algebraic and in the complex analytic set ups where one can define a completion of schemes along closed Noetherian subschemes respectively a completion of complex spaces along analytic subsets, Corollary \ref{fiberjet} makes sense also in the category of  smooth manifolds in which the concept of completed spaces cannot be defined. Furthermore, it has the merit that it permits the definition of a finite-dimensional complex manifold structure on $X^k$. In the following we outline how this is done as shown in \cite{BR73}. Let $U$ be a local coordinate chart in $X$ with local coordinates $x_1, \dots, x_n$. By default $U^{k}=\pi_k^{-1}(U)$ is an open subset of $X^k$. Each point $(x, \Phi_x)$ in $U^k$ determines by Corollary \ref{fiberjet} a unique $k$-jet  $[\phi]_x^k$ of a local parametrization map $\phi: \bb C^n\supset V\rightarrow U$. The latter defines in turn an embedding \[\beta_U^k: U^k\rightarrow\bb C[t_1, \dots, t_n]/\mf{m}^{k+1}\times\dots\times\bb C[t_1, \dots, t_n]/\mf{m}^{k+1}\cong\underbrace{\bb C^{\ms{\binom{n+k}{n}}}\times\dots\times\bb C^{\ms{\binom{n+k}{n}}}}_{\textrm{$n$-times}}\cong\bb C^{\ms{n\binom{n+k}{n}}}\] 
by $(x, [\phi]_x^k)\mapsto([x_1\circ\phi)]_0^k, \dots, [x_n\circ\phi]_0^k)$. In this way every open cover $\{U_i\}$ of $X$ induces a complex analytic structure $\{(U_i^k, \beta_{U_i})\}$ on $X^k$ which makes $X^k$ an $n\binom{n+k}{n}$-dimenional complex analytic manifold. Moreover, it is a direct consequence of  Lemma \ref{fiberiso} that the Lie group $\Aut_{n, k}$ acts freely and transitively on the fibers of $\pi_k: X^k\rightarrow X$ thus making $\pi_k: X^k\rightarrow X$ into a principal $\Aut_{n, k}$-bundle.
\begin{lemma}
\label{xcoorthct}
$X^k$ is a transitive Harish-Chandra $(W_{n, k}, \Aut_{n, k})$-torsor.
\end{lemma}
\begin{proof}
As indicated in Corollary \ref{fiberjet} every point $(x, \Phi_x)$ in $X^k$ is identified with a pair $(x, [\phi]_x^k)$, where $x\in X$ and $\phi:\bb C^n\rightarrow X$ is some local parametrization of $X$ with $\phi(0)=x$. We define the following $\bb C$-linear map 
\begin{align*}
T_{\ms{(x, \Phi_x)}}^{(1, 0)}&X^k\rightarrow W_{n, k}\nonumber\\
&\xi\mapsto\mathrm{Taylor~expansion~ at~0~ of~}~-d\phi_{y}^{-1}(\frac{d}{dt}\big|_{t=0}\phi_t)
\end{align*}
where $[\phi_t]_x^k$ is a path in $X^k$ such that $\phi_{t=0}(0)=x$ and $[\frac{d}{dt}\big|_{t=0}\phi_t]_x^k=\xi$. We leave to the reader the straighforward verification that the above mapping is in fact an $\Aut_{n, k}$-equivariant $\bb C$-linear isomorphism. Moreover, it induces an $\Aut_{n, k}$-equivariant holomorphic vector bundle isomorphism $T^{(1, 0)}X^k\rightarrow X^k\times W_{n, k}$ which is equivalent to an $\Aut_{n, k}$-equivariant $\OO_{X^k}$-module isomorphism $\omega: \mc{T}_{X^k}\rightarrow\OO_{X^k}\otimes_{\bb C}W_{n, k}$ which in turn we interpret as an $\Aut_{n, k}$-equivariant differential $(1, 0)$-form.
The inverse thereof yields a homomorphism of $\bb C$-modules $f:W_{n, k}\rightarrow\mc{V}(X^k)$, which one explicitly checks to be a Lie algebra homomorphism as well. This Lie algebra homomorphism implies that $\omega$ satisfies the Maurer-Cartan condition. Furthermore, it is a matter of direct computation to demonstrate that the composition of mappings $\OO_{X^k}\otimes_{\bb C}W_{n, k}^0\overset{j}{\hookrightarrow}\mc{T}_{X^k}\overset{\omega}{\rightarrow}\OO_{X^k}\otimes_{\bb C}W_{n, k}$, where $j$ is the morphism of $\OO_{X^k}$-modules induced by the embedding of the vertical subbundle of $T^{(1, 0)}X^k$ into $T^{(1, 0)}X^k$, coincides with the embedding $i_{\ms{X^k}}:\OO_{X^k}\otimes_{\bb C}W_{n, k}^0\hookrightarrow W_{n, k}\otimes_{\bb C}\OO_{X^k}$ emanating from the Harish-Chandra pair $(W_{n, k}, \Aut_{n, k})$. Herewith the proof is concluded.
\end{proof}
Consider now the projective system of transitive Harish-Chandra $(W_{n, k},  \Aut_{n, k})$-torsors
\begin{align*}
X\xleftarrow{\pi_1} X^1\xleftarrow{\pi_1^2}\dots \longleftarrow X^{k-1}\xleftarrow{\pi_{k-1}^k} X^k\xleftarrow{\pi_k^{k+1}}X^{k+1}\longleftarrow\dots.
\end{align*}
The projective limit of the above inverse system is in turn a transitive Harish-Chandra $(W_n, \Aut_n)$-torsor over $X$ which is called \emph{the bundle of formal coordiante systems} on $X$. The subgroup of linear transformations $\GL(n, \bb C)$ in $\Aut_n$ acts (from the right) on $\coor{X}$ and the quotient $\aff{X}:=\coor{X}/\GL(n, \bb C)$ is called \emph{the bundle of formal affine coordinate systems} on $X$. Although the projection $\aff{X}\rightarrow X$ can be given the structure of a principal $\Aut_n^+$-bundle, it is neither a Harish-Chandra torsor nor is $\aff{X}$ a principal $W_n$-space. Nevertheless, $\coor{X}\rightarrow\aff{X}$ is a $(W_{n}, \GL(n, \bb C))$-torsor.
\subsection*{The bundle of formal coordinates $\coor{\mc{E}}$ of a holomorphic vector bundle $E\rightarrow X$}
\label{ecoor}
Following the expositions of formal geometry in \cite{EF08} and \cite{Kho07} we aim in the current subsection to extend the notion of a formal coordinate bundle of a complex manifold defined in the preceding subsection to the case of a holomorphic vector bundle of a finite rank. The definitions given in the aforementioned sources are geometric in nature. Here we attempt to formulate equivalent definitions using the language of ringed spaces which suits the work in this paper better.\\
Let from now on until the end of this section $\pi:E\rightarrow X$ denote a holomorphic vector bundle of a finite rank $l$ over the $n$-dimensional complex manifold $X$ and let $\mc{E}$ be the corresponding finite locally free $\OO_X$-module. In the course of this paper we shall interchangeably use the notations $\mc{E}$ and $E$ for a holomorphic vector bundle depending on whether we want to emphasise the algebraic $\OO_X$-module structure or the geometric structure of the bundle. Under a \emph{local parametrization} of $E$ we shall understand a holomorphic vector bundle isomorphism $U\times\bb C^l\rightarrow E|_V$ from the trivial bundle of rank $l$ over some open neighborhood $U\subset\bb C^n$ of $0$ to the restriction of $E$ to some trivializing open set $V\subset X$. Denote by $G$ the structure Lie group of $E$ and by $\mf{g}$ its Lie algebra, respectively. Furthermore, for any holomorphic morphism $f$ from $X$ to any other complex manifold $Y$, let $f^{\*}\mc{E}$ stand for the inverse image $\OO_{Y}$-module $\OO_Y\otimes_{f^{-1}\OO_X}f^{-1}\mc{E}$. Then
\begin{definition}
$\mc{E}^k$ is defined as the set consisting of pairs $(\pmb{\varphi}, f)$ of a closed immersion of $\bb C$-ringed spaces $\pmb{\varphi}=(\varphi, \varphi^{\#}): (0, \OO_{\bb C^n}/\mc{J}_0^{k+1})\hookrightarrow(X, \OO_X)$ and an isomorphism of $\OO_{\bb C^n/\mc{I}_0^{k+1}}$-modules $f:\big(\OO_{\bb C^n}/\mc{I}_0^{k+1}\big)^{\oplus l}\rightarrow\varphi^{\*}\mc{E}$.
\end{definition}  
In a similar fashion to $X^k$ the map $\presup{\mc{E}}{\pi}_k:\mc{E}^k\rightarrow X$ given by $\presup{\mc{E}}{\pi}_k(\pmb{\varphi}, f)=\varphi(0)$ defines a projection which induces a topology $\mc{T}_{\mc{E}}:=\{U^k:=\presup{\mc{E}}{\pi}_k^{-1}(U)~|~U~\textrm{is open in $X$}\}$ on $\mc{E}^k$ 
with respect to which the projection map is continuous, open and closed. Designate by $\mc{E}_x^k$ the topologically closed fiber $\presup{\mc{E}}{\pi}_k^{-1}(x)$ of $x\in X$. In an analogous manner to the proof of Lemma \ref{fiberiso} one demonstrates that
\begin{align}
\label{ekformal}
\mc{E}^k&=\{(\pmb{\varphi}, f)~|~\pmb{\varphi}=(\varphi, \varphi^{\#}):(0, \OO_{\bb C^n}/\mc{I}_0^{k+1})\hookrightarrow(X, \OO_X), f:\big(\OO_{\bb C^n}/\mc{I}_0^{k+1}\big)^{\oplus l}\rightarrow\varphi^{\*}\mc{E}\}\nonumber\\
&=\{(X_{\varphi(0)}^k, f_0)~|~f_0:\big(\OO_{\bb C^n, 0}/\mf{m}_0^{k+1}\big)^{\oplus l}\rightarrow(\varphi^{\*}\mc{E})_0 \}\nonumber\\
&=\{(X_{\varphi(0)}^k, f_{0})~|~f_0:\big(\OO_{\bb C^n, 0}/\mf{m}_0^{k+1}\big)^{\oplus l}\rightarrow\big(\OO_{\bb C^n, 0}/\mf{m}_0^{k+1}\big)\otimes_{\OO_{X, \varphi(0)}}\mc{E}_{\varphi(0)}\}\nonumber\\
&=\coprod_{x\in X}\{(X_x^k, f_0)~|~f_0:\big(\OO_{\bb C^n, 0}/\mf{m}_0^{k+1}\big)^{\oplus l}\rightarrow\big(\OO_{\bb C^n, 0}/\mf{m}_0^{k+1}\big)\otimes_{\OO_{X, x}}\mc{E}_{x}\}\nonumber\\
&=\coprod_{x\in X}\mc{E}_x^k,
\end{align}
that is, the disjoint union of fibers of $\presup{\mc{E}}{\pi}_k$ gives $\mc{E}^k$.\\
Let $\Delta$, $\Delta^{(k)}$ and $\pr_i$, $i=1, 2$, be as in the previous subsection. Further, let  $p_i$ denote the composition of the natural morphism of $\bb C$-ringed spaces $\Delta^{(k)}\rightarrow X\times X$ with $\pr_{i}$ for $i=1, 2$. We define the \emph{jet bundle of $k$-th order} of $\mc{E}$  of rank $l\binom{n+k}{n}$ as the finite locally free $\OO_X$-module $J^k(\mc{E}):=p_{1\*}p_2^{\*}\mc{E}$. The corresponding holomorphic vector bundle over $X$ of the same rank $l\binom{n+k}{n}$ is denoted by $J^k(E)$. We observe that $p_1$ is a proper map whose fiber is an one-point space. Therefore, the stalk of $J^k\mc{E}$ at a point $x\in X$ is given by
\begin{align}
\label{jetbundlestalk}
J^k(\mc{E})_x&=\big(p_{1\*}p_2^{\*}\mc{E}\big)_x\nonumber\\
&=\Gamma(p_1^{-1}(x), p_2^{\*}\mc{E}|_{p_1^{-1}(x)})\nonumber\\
&\cong\OO_{X, x}/\mf{m}_x^{k+1}\otimes_{\OO_{X, x}}\mc{E}_x
\end{align}     
where in the second line we again applied \cite[Remark 2.5.3]{KS94} and in the second to the last line we used the fact that $(p_2^{\*}\mc{E})_{(x, x)}$ is a constant sheaf on the  one-point space $(x, x)\in\Delta(X)$.\\
The isomorphism induced by epimorphism \eqref{factoringmap} delivers an $\OO_{\bb C^n, 0}/\mf{m}_0^{k+1}$-module isomorphism 
\[\big(\OO_{\bb C^n, 0}/\mf{m}_0^{k+1}\big)\otimes_{\OO_{X, \varphi(0)}}\mc{E}_{\varphi(0)}\cong\big(\OO_{\bb C^n, 0}/\mf{m}_0^{k+1}\big)\otimes_{\frac{\OO_{X, \varphi(0)}}{\mf{m}_{\varphi(0)}^{k+1}}}\big(\OO_{X, \varphi(0)}/\mf{m}_{\varphi(0)}^{k+1}\big)\otimes_{\OO_{X, \varphi(0)}}\mc{E}_{\varphi(0)}\]
which combined with the $\OO_{\bb C^n, 0}/\mf{m}_0^{k+1}$-module isomorphism $f_0$ from the third line in \eqref{ekformal} along with isomorphism \eqref{jetbundlestalk} and the well known identification $\OO_{\bb C^n, 0}/\mf{m}_{0}^{k+1}\cong\mc{J}^k_{\bb C^n, 0}$ ultimately yields the isomorphism of $\OO_{\bb C^n, 0}/\mf{m}_0^{k+1}$-modules
$\big(\mc{J}_{\bb C^n, 0}^{k}\big)^{\oplus l}\cong\OO_{\bb C^n, 0}/\mf{m}_0^{k+1}\otimes_{\OO_{X, \varphi(0)}/\mf{m}_{\varphi(0)}^{k+1}}J^k(\mc{E})_{\varphi(0)}$.
Since $\OO_{\bb C^n, 0}/\mf{m}_0^{k+1}$ is local, the latter isomorphism induces a unique linear $\bb C$-isomorphism between the fibers $J_0^k(\bb C^n\times\bb C^l)$ and  $J_{\varphi(0)}^k(E)$ of the vector bundles $J^k(\bb C^n\times\bb C^l)$ and $J^k(E)$ corresponding to the sheaves $\mc{J}_{\bb C^n}^k$ and $J^k(\mc{E})$, respectively.
The linear $\bb C$-isomorphism is equivalent to a $k$-jet at $0\in\bb C^n$ of a pointed local biholomorphism $(\bb C^n\times\bb C^l, 0)\rightarrow (E, \varphi(0))$ which in turn is equivalent to a $k$-jet at $0\in\bb C^n$ of a local parametrization of the vector bundle $E$. Thus $\mc{E}^k_x$ is bijective to the set $\{\textrm{$k$-jet at $0\in\bb C^n$ of a local parametrization $\phi: \bb C^n\times\bb C^l\rightarrow E$ at $0$}\}$. In analogous manner to $X^k$ this identification between abstract points on $\mc{E}^k$ and $k$-jets of local parametrizations of $E$ makes it possible to give complex analytic structure on $\mc{E}^k$. More precisely, let $U$ be a local chart on $X$ which trivializes $E$ by virtue of a trivialization mapping $\psi: E|_U\rightarrow U\times\bb C^l$.
First, we have by default that $U^k\in\mc{B}$. Second, as elucidated in the previous paragraph, each point in $\mc{E}^k$ is uniquely identified with a $k$-jet of a local parametrization $\phi: W\times\bb C^l\rightarrow E|_U$ at $0\in\bb C^n$, $W$ open in $\bb C^n$. This data allows us to set up an embedding
\begin{equation*}
\label{ecoorchart} 
\beta_U^k: U^k\rightarrow \bb{C}[t_1, \dots, t_n]/\mf{m}^{k+1}\times\dots\times \bb{C}[t_1, \dots, t_n]/\mf{m}^{k+1}\times\bb{C}^{l^2}\cong\prod_{i=1}^n\bb C^{\ms{\binom{n+k}{n}}}\times\bb C^{q^2}\cong\bb C^{\ms{n\binom{n+k}{n}+q^2}} 
\end{equation*} 
by $(x, \Phi_x, f_0)\mapsto[\phi]_0^k\mapsto([\psi^1\circ\phi]_0^k, \dots, [\psi^n\circ\phi]_0^k, [\psi^{n+1}\circ\phi]_0^k, \dots, [\psi^{n+l}\circ\phi]_0^k)$, where $q=\dim_{\bb C}\mf g$. We also take into account the fact that the parametrization $\phi$ and trivialization $\psi$ are bundle maps which impose the additional conditions $\psi^i\circ\phi(x, y)=f^i(x)$ for $i=1, \dots, n$ and $\psi^{n+i}\circ\phi(x, y)=\sum_{j=1}^lb(x)_{ij}y_j$ for $i=1, \dots, l$ and $(x,y)\in\bb C^n\times\bb C^l$. The set $\{U^k, \beta_U^k\}$ defines an $n\binom{n+k}{n}+q^2$-dimensional complex analytic structure on $\mc{E}^k$. From \eqref{ekformal} we draw the inference that the fibers of $\mc{E}^k$ are right $\Aut_{n,k}\times G$-torsors where $G$ is the structure group of of the vector bundle $E$. 
This implies that $\presup{\mc{E}}{\pi}_k: \mc{E}\rightarrow X$ is a principal $\Aut_{n,k}\times G$-bundle. The Lie algebra of $\Aut_{n, k}\times G$ is the direct sum $W_{n, k}^0\oplus\mf{g}$ of the Lie algebras $W_{n, k}^0$ and $\mf{g}$. It is embedded in $W_{n, k}\ltimes\mf{g}\otimes_{\bb C}\widehat{\OO}_n/\hat{\mf{m}}^{k+1}$, the semidirect product of the Lie algebra $W_{n, k}$ and the Lie algebra of $\mf{g}$-valued power series up to $k$-th order. The semidirect product here betokens that, while the underlying  complex vector space of $W_{n, k}\ltimes\mf{g}\otimes_{\bb C}\widehat{\OO}_n/\hat{\mf{m}}^{k+1}$ is isomorphic to the direct sum of its constituents, the Lie bracket is twisted by the $W_{n, k}$-action on $\mf g\otimes_{\bb C}\hat{\OO}_/\hat{\mf{m}}^{k+1}$, that is,
 \begin{equation}
 \label{twliebracket}[v+g_1\otimes p_1, w+g_2\otimes p_2]=[v, w]+[g_1, g_2]\otimes p_1p_2+g_2\otimes v(p_2)-g_1\otimes w(p_1),
 \end{equation} 
where $v, w\in W_{n, k}$,~$g_1, g_2\in\mf{g}$ and $p_1, p_2\in\widehat{\OO}_n/\hat{\mf{m}}^{k+1}$. Next, we continue with the Harish-Chandra torsor structure of $\mc{E}^k$ but beforehand we show that $(W_{n, k}\ltimes\mf{g}\otimes\widehat{\OO}_n/\hat{\mf{m}}^{k+1}, \Aut_{n, k}\times G)$ is a Harish-Chandra pair. 
\begin{lemma}
\label{extendedhcpair}
Let $G$ be an arbitrary matrix Lie subgroup of $\GL(m, \bb C)$ with corresponding Lie algebra $\mf{g}$. Then, $(W_{n, k}\ltimes\mf{g}\otimes\widehat{\OO}_n/\hat{\mf{m}}^{k+1}, \Aut_{n, k}\times G)$ is a Harish-Chandra pair. 
\end{lemma}
\begin{proof}
The action of an element $(\phi, A)\in\Aut_{n, k}\times G$ on an element $D+B\otimes p$ from $W_{n, k}\ltimes\mf{g}\otimes\widehat{\OO}_n/\hat{\mf{m}}^{k+1}$ is given by 
$(\phi, A)\cdot(D+B\otimes p)=\phi
\circ D\circ\phi^{-1}+ABA^{-1}\otimes p+B\otimes\phi(p)$.
From that we infer that the action of the Lie algebra $W_{n, k}^0\oplus\mf{g}$ of $\Aut_{n, k}\times G$ on $W_{n, k}\ltimes\mf{g}\otimes\widehat{\OO}_n/\hat{\mf{m}}^{k+1}$ coincides precisely with the adjoint action of  $W_{n, k}\ltimes\mf{g}\otimes\widehat{\OO}_n/\hat{\mf{m}}^{k+1}$, restricted to $W_{n, k}^0\oplus\mf{g}$. This concludes the proof.  
\end{proof}
\begin{corollary}
Let $G$ be as in Lemma \ref{extendedhcpair}. Then $(W_n\ltimes\mf{g}\otimes_{\bb C}\widehat{\OO}_n, \GL(n, \bb C)\times G)$ is a Harish-Chandra pair. 
\end{corollary}
\begin{proof}
This follows from the fact that $\GL(n, \bb C)\times G$ is a closed Lie subgroup of the Lie group $\Aut_{n, k}\times G$. 
\end{proof}
At any point $(x, \Phi_x, f_0)\in\mc{E}^k$ the mapping $T_{(x, \Phi_x, f_0)}^{(1, 0)}\mc{E}^k\rightarrow W_{n, k}\ltimes\mf{g}\otimes_{\bb C}\widehat{\OO}_n/\hat{\mf{m}}^{k+1}$ given by 
\begin{equation}
\label{connection1formonE}
\xi\mapsto\textrm{Taylor expansion at $0\in\bb C^n$ of} -d\phi_{(x, y)}^{-1}\frac{d}{dt}|_{t=0}\phi_t,
\end{equation}
where $[\phi]_0^k$ is the unique $k$-jet at $0\in\bb C^n$ of a local parametrization with  corresponding to $(x, \Phi_x, f_0)$ and $[\phi_t]_0^k$ is a path in $\mc{E}^k$ such that $[\phi_{t=0}]_0^k=[\phi]_0^k$ and $\xi=[\frac{d}{dt}\big|_{t=0}\phi_t]_0^k$, is a $\bb C$-linear isomorphism which intertwines with the $\Aut_{n, k}\times G$-action . This in turn induces a holomorphic parallelism on $\mc{E}^k$ and consequently an $\Aut_{n, k}\times G$-equivariant differential $(1, 0)$-Maurer Cartan form on $\mc{E}^k$ which endow $\mc{E}^k$ with the structure of a transitive $(W_{n, k}\ltimes\mf{g}\otimes_{\bb C}\widehat{\OO}_n/\mf{m}^{k+1}, \Aut_{n, k}\times G)$-torsor. The projective limit of the inverse system $(\mc{E}^k, \pi_{k-1}^{k})$ of Harish-Chandra torsors, where $\pi_{k-1}^{k}$ are the natural connecting morphisms, yields the \emph{the bundle of formal coordinate systems} $\coor{\mc{E}}$ on $\pi: E\rightarrow X$, which is a transitive Harish-Chandra $(W_n\ltimes\mf{g}\otimes_{\bb C}\widehat{\OO}_n, \Aut_n\times G)$-torsor. The quotient of $\coor{\mc{E}}$ by the subgroup $\GL(n, \bb C)\times G$ of $\Aut_n\times G$ yields $\aff{\mc{E}}\rightarrow X$, the \emph{the bundle of affine coordinate systems} on $E\rightarrow X$. It is a principal $\Aut_n^+\times G$-bundle. Moreover, $\GL(n, \bb C)\times G$ as a closed subgroup of the Lie group $\Aut_{n, k}\times G$ for every $k\in\bb Z_{\geq0}$, is in fact an admissible Lie subgroup of the pro-finite Lie group $\Aut_n\times G$. Therefore, the projection map $\Aut_n\times G\rightarrow\Aut_n\times G/\GL(n, \bb C)\times G$ can be given the structure of a principal $\GL(n, \bb C)\times G$-bundle. Consequently, the mapping
\[\coor{\mc{E}}\times_{\Aut_n\times G}(\Aut_n\times G)\rightarrow\coor{\mc{E}}\times_{\Aut_n\times G}\big(\Aut_n\times G/\GL(n, \bb C)\times G\big),\]
 which is identical to the projection $\coor{\mc{E}}\rightarrow\aff{\mc{E}}$, is a principal $\GL(n, \bb C)\times G$-bundle. Beyond that, the total space of this principal bundle is endowed with a $W_n\ltimes\mf{g}\otimes_{\bb C}\widehat{\OO}_n$-action, compatible with the preexisting $\Aut_n\times G$-action on $\coor{\mc{E}}$, hence compatible with the $\GL(n, \bb C)\times G$-action. This gives $\coor{\mc{E}}\rightarrow\aff{\mc{E}}$ the structure of a Harish-Chandra $(W_n\ltimes\mf{g}\otimes_{\bb C}\widehat{\OO}_n, \GL(n, \bb C)\times G)$-torsor. 
\section{Induction of sheaves of $\bb CG$-interior algebras}
\label{Appendix A}
Let $\Sh_{H}(X)$ and $\Sh_{G}(X)$ designate the categories of $\Cs$-algebras on $X$ in the $H$ and the $G$-equivariant topologies, respectively. 
We discuss a number of relevant to this work isomorphisms between inductions of $G$-algebras and $\bb CG$-interior algebras. The main results here are Corollary B.8 and Corollary B.9.    
\begin{definition}
A $G$-algebra is an associative $\bb C$-algebra $A$ with an algebraic representation $G\rightarrow\Aut(A)$ of a finite group $G$. 
\end{definition}
For a $H$-algebra $A$ of a subgroup $H$ of $G$, Turull defined in \cite{tur06} the induced $G$-algebra $\Indu_{H}^{G}(A):=\bb CG\otimes_{\bb CH}A$ with product, given by 
\begin{equation}
\label{turullprod}
(g_{1}\otimes a_{1})(a_{2}\otimes g_{2})=
\begin{cases}
g\otimes a_{1}a_{2}~~ \textrm{if}~~g_{1}=g_{2}=g\\
0~~ \textrm{if}~~g_{1}g_{2}^{-1}\notin H 
\end{cases}
\end{equation} 
and the $G$-action is given by 
\begin{equation*}
\presup{g}{(g'\otimes a)}=gg'\otimes a.
\end{equation*}
for all $g_{1}, g_{2}, g, g'\in G$, $a_{1}, a_{2}, a\in A$.
\begin{definition}
A $\bb CG$-interior algebra is an associative $\bb C$-algebra endowed with a homomorphism of~~$\bb C$-algebras $\bb CG\rightarrow A$.
\end{definition}
The most important example of an $\bb CG$-interior algebra is the smash product algebra $A\rtimes \bb CG$ of a $G$-algebra $A$ and the group $G$. As $G$ is among the generators of the rational Cherednik algebra $H_{t, c}(\mf{h}, G)$, as well as Etingof's global Cherednik algebra $H_{t, c}(X, G)$, both are $\bb CG$-interior algebras. For any subgroup $H$ of $G$, Luis Puig defined in \cite{puig81} an induction functor from the category of $\bb CH$-interior algebras to the category of $\bb CG$-interior algebras by 
$\Indu_H^G(A):=\bb CG\otimes_{\bb CH}A\otimes_{\bb CH}\bb CG$ for a $\bb CH$-interior algebra $A$ with product given by 
\begin{equation}
\label{puigprod}
(g_{1}\otimes a_{1}\otimes g'_{1})(g_{2}\otimes a_{2}\otimes g'_{2})=\begin{cases}
g_1\otimes a_1g'_1g_2a_2\otimes g'_2~~\textrm{if} ~~g'_1g_2\in H\\
0~~\textrm{if}~~ g'_1g_2\notin H 
\end{cases}
 \end{equation}
for every pair of elements $g_{1}\otimes a_{1}\otimes g'_{1}, g_{2}\otimes a_{2}\otimes g'_{2}\in\Indu_H^{G}(A)$. From the definition of the product it is immediately evident that every element in $\Indu_H^G(A)$ is a nonzero zero divisor: for any $g_{1}\otimes a_{1}\otimes g'_{1}$, the product with $g_1'^{-1}g\otimes1\otimes1$, where $g\in G/H$ some representative, yields zero by \eqref{puigprod}. 
In the following we generalize Turull's and Puig's induction functors to the case of sheaves of $G$-algebras and $\bb CG$-interior algebras, respectively and prove some properties thereof  which are later employed in the gluing procedure carried out in Section \ref{gluing}. Let $\Sh_{H}(X)$ and $\Sh_{G}(X)$ designate the categories of $\Cs$-algebras on $X$ in the $H$- and the $G$-equivariant topologies, respectively.  
\begin{definition}
Define the functor $\mathsf{Ind}_{H}^{G}: \Sh_{H}(X)\rightarrow\Sh_{G}(X)$ by $\mc{F}\mapsto\mathsf{Ind}_{H}^{G}(\mc{F})$ with 
\[\mathsf{Ind}_{H}^{G}(\mc{F}):=\Sheaf(\bb CG\otimes_{\bb CH}\mc{F})\] 
where~~$\Sheaf$ is the sheafification functor and the product of sections in $\mathsf{Ind}_{H}^{G}$ is given by \eqref{turullprod}. 
\end{definition}
In a similar fashion, we introduce a sheaf-theoretic version of Puig's induction of interior algebras.  
\begin{definition}
\label{puiginductionfunctor}
For every sheaf of~~$\bb CH$-interior algebras $\mc{F}$ in the $H$-equivariant topology define the functor~~$\Induc_{H}^{G}: \Sh_H(X)\rightarrow\Sh_G(X)$ by
$\mc{F}\mapsto\Induc_{H}^{G}(\mc{F})$ with 
\[\Induc_{H}^{G}(\mc{F}):=\Sheaf(\bb CG\otimes_{\bb CH}\mc{F} \otimes_{\bb CH}\bb CG)\] 
where the product of sections in $\Induc_{H}^{G}$ is given by \eqref{puigprod}.
\end{definition}
The sheaf of skew group $\Cs$-algebras $\mc{F}\rtimes\bb CG$ of a sheaf of $G$-algebras $\mc{F}$ is, as explained in the above, naturally a $\bb CG$-interior algebra in the $G$-equivariant topology. The focus of our work are sheaves of $\bb CH$-interior algebras, defined in the $G$-equivariant topology. We state the following important isomorphism for sheaves of skew group $\Cs$-algebras.  
\begin{theorem}
\label{theoreminduction1}
Let $Y$ be a $H$-invariant subset of the $G$-space $X$. Suppose $\mc{F}$ is a sheaf of $H$-algebras in the $H$-equivariant topology on $Y$. Then, there exists an isomorphism of sheaves of~~$\bb CG$-interior algebras between $\Induc_{H}^{G}(\mc{F}\rtimes H)$ and $\mathsf{Ind}_{H}^{G}(\mc{F})\rtimes\bb CG$ on $GY$ in the $G$-equivariant topology.
\end{theorem}
\begin{proof}
A sheaf theoretic version of the proof of \cite[Theorem 1]{tib09}. 
\end{proof}
\begin{theorem}
\label{theoreminduction2}
There is an isomorphism of sheaves of~~$\bb CG$-interior algebras 
\[\mathsf{Ind}_{H}^{G}(\mc{D}_{\mathring{X}\coprod X_{H}^{i}})\rtimes\bb CG\rightarrow\big(\mc{D}_{G(\mathring{X}\coprod X_{H}^{i})}\big)\rtimes\bb CG \]
on $G(\mathring{X}\coprod X_{H}^{i})$ in the $G$-equivariant topology.
\end{theorem}
\begin{proof}
Consider the basis $\mc{B}$ on $\mathring{X}\coprod X_{H}^{i}$,  given at the beginning of Section \ref{gluingofstrata01}. Then  $\mc{B}^{G}:=\{\ind_{H}^{G}(U)\in\mc{B}~|~U\in\mc{B}\}$ is a basis for $G(\mathring{X}\coprod X_{H}^{i})$. Clearly, $\big(\mc{D}_{G(\mathring{X}\coprod X_{H}^{i})}\big)\rtimes\bb CG$ is a sheaf on $\mc{B}^{G}$. We infer from $\mc{D}_{\mathring{X}\coprod X_{H}^{i}}(\coprod_{g\in G/H}gU)\cong\oplus_{g\in G/H}\mc{D}_{\mathring{X}\coprod X_{H}^{i}}(gU)=\mc{D}_{\mathring{X}\coprod X_{H}^{i}}(U)$ that $\mathsf{Ind}_{H}^{G}(\mc{D}_{\mathring{X}\coprod X_{H}^{i}})\rtimes\bb CG$ is a well-defined sheaf on $\mc{B}^{G}$. For any $\ind_{H}^{G}U\in\mc{B}^G$, the product in $\mc{D}_{G(\mathring{X}\coprod X_{H}^{i})}(\ind_{H}^{G}(U))$ is defined by 
\begin{equation*}
(d_{1}|_{g_{1}U})(d_{2}|_{g_{2}U})=
\begin{cases}
d_{1}d_{2}|_{gU},~~\textrm{if}~~g_{1}=g_{2}=g\\
0,~~\textrm{if}~~ g_{1}g_{2}^{-1}\notin H
\end{cases}
\end{equation*}
for all $d_{1}, d_{2}\in\mc{D}_{G(\mathring{X}\coprod X_{H}^{i})}(\ind_{H}^{G}(U))$. We notice that the so-defined product resembles Turull's product \eqref{turullprod}. With that knowledge define a mapping 
\begin{align*}
\eta: \bb CG\otimes_{\bb CH}\mc{D}_{\mathring{X}\coprod X_{H}^{i}}(U)\rtimes\bb CG&\rightarrow\mc{D}_{G(\mathring{X}\coprod X_{H}^{i})}(\ind_{H}^{G}(U))\rtimes\bb CG\\
(g\otimes d)\*\bar g&\mapsto \presup{g}{d}\*\bar g. 
\end{align*}
A verification shows that $\eta$ is compatible with Turull's product and the product on $\mc{D}_{G(\mathring{X}\coprod X_{H}^{i})}(\ind_{H}^{G}(U))$. Hence $\eta$ is a $\bb C$-algebra morphism. 
Moreover, this morphism is compatible with restriction maps on open basic sets. Surjectivity and injectivity of $\eta$ are evident, so the morphism defines an isomorphism at the level of $\mc{B}$.  By the equivalence of the categories of sheaves on $X$ and $\mc{B}$ the isomorphism extends uniquely to the desired isomorphism in the theorem.
\end{proof}
The preceding theorems imply the next relevant result.
\begin{corollary}
\label{corollaryinduction1}
There is an isomorphism of sheaves of~~$\bb CG$-interior algebras 
\begin{equation}
\label{inducmorphism}
\Induc_{H}^{G}(\mc{D}_{\mathring{X}\coprod X_{H}^{i}}\rtimes\bb CH)\rightarrow\mc{D}_{G(\mathring{X}\coprod X_{H}^{i})}\rtimes\bb CG
\end{equation}
in the $G$-equivariant topology of $G(\mathring{X}\coprod X_{H}^{i})$. 
\end{corollary}
\begin{proof}
The isomorphism in question follows from the composition of the isomorphism from Theorem \ref{theoreminduction2} with the isomorphism
$\Induc_{H}^{G}(\mc{D}_{\mathring{X}\coprod X_{H}^{i}}\rtimes\bb CH)\cong\mathsf{Ind}_{H}^{G}(\mc{D}_{\mathring{X}\coprod X_{H}^{i}})\rtimes\bb CG$ from Theorem \ref{theoreminduction1}. 
\end{proof}
We arrive at the foolowing aimed result
\begin{corollary}
\label{corinductionisom}
There is an isomorphism of sheaves of~~$\bb CG$-interior algebras \[\Induc_{H}^{G}(\mc{H}_{1, c, \mathring{X}\coprod X_{H}^{i}, H})\rightarrow\mc{H}_{1, c, G(\mathring{X}\coprod X_{H}^{i}), G}\]
in the $G$-equivariant topology of $G(\mathring{X}\coprod X_{H}^{i})$. 
\end{corollary}
\begin{proof}
Denote by the same symbol $D$ the union of all reflection hypersurfaces in $\mathring{X}\coprod X_{H}^{i}$ and $G(\mathring{X}\coprod X_{H}^{i})$, respectively, and let $j_{H}$ and $j_{G}$ be the corresponding embeddings of $\mathring{X}\coprod X_{H}^{i}\setminus D$ in $\mathring{X}\coprod X_{H}^{i}$ and of $G(\mathring{X}\coprod X_{H}^{i})$ in $D\hookrightarrow G(\mathring{X}\coprod X_{H}^{i})$, respectively. Then by definition we have that $\Induc_{H}^{G}(\mc{H}_{1, c, \mathring{X}\coprod X_{H}^{i}, H})$ and $\mc{H}_{1, c, G(\mathring{X}\coprod X_{H}^{i}), G}$ are subsheaves of $\Induc_{H}^{G}(j_{H\*}j_{H}^{\*}\mc{D}_{\mathring{X}\coprod X_{H}^{i}}\rtimes\bb CH)$ and $j_{G\*}j_{G}^{\*}\mc{D}_{G(\mathring{X}\coprod X_{H}^{i})}\rtimes\bb CG$, respectively. We check that Isomorphism \eqref{inducmorphism} induces a morphism
 \begin{align}
 \label{inducedmorphism}
\Induc_{H}^{G}(\mc{H}_{1, c, \mathring{X}\coprod X_{H}^{i}, H})(\ind_H^G(U))&\rightarrow\mc{H}_{1, c, G(\mathring{X}\coprod X_{H}^{i}), G}(\ind_H^G(U))\nonumber\\
g\otimes d\otimes\bar g&\mapsto\presup{g}{(d)}g\bar g
\end{align} 
for all basic open sets $\ind_H^G(U)$ of the $G$-equivariant topology. Indeed, the  generators of the algebra on the left hand side are mapped on elements of the Cherednik lagebra on the right hand side of \eqref{inducedmorphism}. The injectivity thereof follows from the fact that \eqref{inducedmorphism} is induced by the isomorphism \eqref{inducmorphism}. Therefore, the only thing that is left to be verified is the surjectivity. It suffices to check whether the induced morphism is surjective on generators of $\mc{H}_{1, c, G(\mathring{X}\coprod X_{H}^{i}), G}$. Assume that $d|_{\ind_H^G(U)}$ is either a function or a Dunkl opdam operator. Accounting for the fact that thanks to the disjointness of $\ind_H^G(U)$ we have $d=\sum_{g\in G/H}d|_{gU}$, where $d|_{gU}$ is the restriction to $gU$, the image of $1\otimes\sum_{g\in G/H}\presup{g^{-1}}{(d|_{gU})}\otimes1$ from $\Induc_{H}^{G}(\mc{H}_{1, c, \mathring{X}\coprod X_{H}^{i}, H})(\ind_H^G(U))$ is precisely $d$. Thus, the induced morphism \eqref{inducedmorphism} is surjective, hence an isomorphism of $\bb CG$-interior algebras for all basic open sets $\ind_H^G(U)$ of the $G$-equivariant topology. This implies the claim of the theorem.     
\end{proof}
The statement of the above corollary holds true also for sections of $\mc{H}_{1, c, X, G}$ over disjoint unions $\coprod_{g\in G/K}gU$ of $K$-invariant open sets. 
\begin{corollary}
\label{inductionssections}
Let $K$ be an arbitrary subgroup of $G$ and let $U$ be a an arbitrary $K$-invariant open subset of $F^1(X)$ such that $gU\cap U=\varnothing$ for all $g\in G/K$. Then 
\begin{align}
\label{inductionidentificationsections}
\bb CG\otimes_{\bb CH}H_{1, c}(U, K)\otimes_{\bb CK}\bb CG&\rightarrow H_{1, c}(\ind_H^G(U), G)\nonumber\\
g\otimes d\otimes g'&\mapsto\presup{g}{(d)}gg'
\end{align}
is an isomorphism of $\bb C$-algebras.
\end{corollary}
\begin{proof}
Clear.
\end{proof}


 \printbibliography
\end{document}